\documentclass[12pt,a4paper,reqno]{amsart}

\usepackage{calc,verbatim,enumitem,tikz,url,hyperref,mathrsfs,cite}

\usepackage{amsmath, amssymb, latexsym, amsthm, amsbsy}

\usepackage{graphics}

\usepackage{graphicx,relsize}

\usepackage[mathscr]{eucal}

\usepackage{bbm}

\usepackage[a4paper,top=3cm,bottom=3cm,inner=3cm,outer=3cm,includehead,includefoot]{geometry}

\newcommand{\ssq}{\subseteq} 
\newcommand{\subgp}[1]{\langle{{\hash}1}\rangle}

\def\F{\mathcal{F}}

\def\X{\mathbb{X}}

\def\ind{\textrm{ind}}
\def\ourmin{\mathrm{min}}
\def\ourmax{\mathrm{max}}

\def\vect{\rightharpoonup}
\newcommand{\Ex}[1]{\mathbb{E}\left[#1\right]}

\newcommand{\pr}[1]{\mathbb{P}\left(#1\right)}

\def\exff{\mathbb{E}_{F,F'}(G_{i-1})}
\newcommand{\fl}[1]{\ensuremath{\lfloor #1 \rfloor}}

\newcommand{\eq}[1]{\begin{equation}\label{eq:#1}}
\newcommand{\eqe}{\end{equation}}
\newcommand{\eqr}[1]{\eqref{eq:#1}}

\numberwithin{equation}{section}

\def\Var{\textup{Var}}

\def\owedge{\mathsmaller{\bigwedge}}
\newcommand{\obinom}[2]{\scalebox{1.2}{$\binom{#1}{#2}$}}

\newtheorem{theorem}{Theorem}[section]
\newtheorem{lem}[theorem]{Lemma}

\newtheorem{cor}[theorem]{Corollary}

\newtheorem{prop}[theorem]{Proposition}
\theoremstyle{definition}\newtheorem{definition}[theorem]{Definition}
\theoremstyle{definition}\newtheorem{remark}[theorem]{Remark}

\def\le{\leqslant}
\def\ge{\geqslant}

\def\eps{\varepsilon}

\usepackage{soul}

\definecolor{sgreen}{rgb}{0.3, 0.9, 0.3} 

\begin{document}

\title[Moderate deviations of subgraph counts]{Moderate deviations of subgraph counts in the Erd\H os-R\' enyi random graphs $G(n,m)$ and $G(n,p)$}

\author{Christina Goldschmidt} \vspace{0.8cm}

\address{Department of Statistics and Lady Margaret Hall, University of Oxford, 24-29 St Giles', Oxford OX1 3LB, UK}
\email{goldschm@stats.ox.ac.uk}

\author{Simon Griffiths} 
\address{Departamento de Matemática, PUC-Rio, Rua Marqu\^es de S\~ao Vicente 225, G\'avea, Rio de Janeiro 22451-900, Brazil}		
\email{simon@mat.puc-rio.br}

\author{Alex Scott}
\address{Mathematical Institute, University of Oxford, Oxford OX2 6GG, UK}
\email{scott@maths.ox.ac.uk}

\maketitle

\begin{abstract} The main contribution of this article is an asymptotic expression for the rate associated with moderate deviations of subgraph counts in the Erd\H os-R\'enyi random graph $G(n,m)$.  Our approach is based on applying Freedman's inequalities for the probability of deviations of martingales to a martingale representation of subgraph count deviations.  In addition, we prove that subgraph count deviations of different subgraphs are all linked, via the deviations of two specific graphs, the path of length two and the triangle.  We also deduce new bounds for the related $G(n,p)$ model.
\end{abstract}

\section{Introduction}

Deviations of subgraph counts in random graphs, and in particular in the Erd\H os-R\'enyi random graph $G(n,p)$, have been the focus of intense study in recent years.  Almost all of the results have concerned either small deviations (of the order of the standard deviation) or large deviations (of the order of the mean).  Less is known about the intermediate range of moderately large deviations.

Corresponding to the first category, deviations of the order of the standard deviation, Ruci\'nski established~\cite{Ruc} that for the entire range of densities $p$ such that $np^{e(H)},(1-p)n^2\to \infty$ the number of copies of a fixed graph $H$ in $G(n,p)$ is asymptotically normally distributed.   Articles with results that are quantitively stronger have followed~\cite{BKR,KRT,RR,Rol}.  On the other hand Janson~\cite{Jan} (building on the earlier articles, Janson~\cite{JanRSA} and Janson and Nowicki~\cite{JN}) gives a general framework in which to think about random graph statistics.   Among other results, he proves a functional central limit theorem for the evolution of subgraph count deviations, and that subgraph counts in $G(n,m)$ are also asymptotically normally distributed.

In the second category, deviations of the order of the mean, usually referred to as large deviations, have become a major focus in recent years.  Interest in these problems grew after the seminal articles of Vu~\cite{Vu} and Janson and Ruci\'nski~\cite{JR} in the early 2000s provided many results, using a large range of techniques, which were still far from best possible.  Important subsequent advances include the translation of such deviation problems into variational problems for graphons (Chatterjee and Varadhan~\cite{CV}) and solutions to these variational problems for certain values of the parameters (Lubetzky and Zhao~\cite{LZ} and Zhao~\cite{Z}).   We recommend the survey of Chatterjee~\cite{Chat} and the references therein for a more detailed overview.  Note that the approach of Chatterjee and Varadhan~\cite{CV}, which is applied in the context of the model $G(n,p)$, has been generalised to apply in $G(n,m)$ by Dembo and Lubetzky~\cite{DL}.  Very recently, a major breakthrough by Harel, Mousset and Samotij~\cite{HMS2019} has greatly extended the range of such large deviation results.

In this article, we focus on deviation events of some intermediate size, usually called moderate deviations.  We shall focus on the random graph model $G(n,m)$, with a fixed number of edges, which we believe to be the more natural context in which to study moderate deviations of subgraph counts.  For example, in the dense case, the standard deviation of the number of triangles in $G(n,m)$ is of order $n^{3/2}$, while it is of order $n^2$ in $G(n,p)$.  This expresses the fact that by far the easiest way for $G(n,p)$ to have extra triangles is simply to have extra edges.  By fixing the number of edges and working in $G(n,m)$ one studies the finer problem of other possible causes of triangle count deviations.

Our main contributions are as follows:
\begin{enumerate}
\item[(i)] We give a general martingale-type expression for subgraph count deviations in $G(n,m)$ (see Theorem~\ref{thm:Mart}).
\item[(ii)] We prove that subgraph count deviations are generally well predicted by the deviations of the counts of two specific graphs, $P_2$ and $K_3$ (see Theorem~\ref{thm:relate}).
\item[(iii)] Using the above results, we determine the asymptotic rate associated with moderately large subgraph count deviations.  That is we determine the function $r=r(n)$ such that a deviation of this type has probability $\exp\big(-r(1+o(1))\big)$ (see Theorem~\ref{thm:main}).
\item[(iv)] We deduce results concerning moderately large subgraph count deviations in $G(n,p)$ which are significantly stronger than previously known bounds (see Theorems~\ref{thm:smalldelta} and~\ref{thm:largerdelta}).
\end{enumerate} 

We state other auxiliary results along the way, such as an approximate bound on deviation probabilities across the whole range of deviations, Theorem~\ref{thm:upto} and an estimate for the tail of the binomial distribution, Theorem~\ref{thm:bah}.

We require the following notation.  We write $N_H(G)$ for the number of embeddings of a graph $H$ in a graph $G$.  That is, the number of injective functions $\phi:V(H)\to V(G)$ such that
\[
\phi(u)\phi(v)\in E(G) \qquad \text{for all} \qquad uv\in E(H)\, .
\]
This is also referred to in the literature as the number of \emph{isomorphic copies} of $H$ in $G$.
When we count without multiplicity we write $\binom{G}{H}$, so that, for example
\[
N_{K_3}(K_4)\, =\, 24\qquad \text{and}\qquad \binom{K_4}{K_3}\, =\, 4\, .
\]

We shall be interested interested in $N_H(G)$, where $H$ is a fixed graph and $G$ is a large random graph.  For example, we think of a fixed graph $H$ with $v=v(H)$ vertices, and $e=e(H)$ edges, and a large graph $G$ with $n$ vertices and $m$ edges, where $n$ is taken very large, and $m$ behaves as a function of $n$.  (In view of this choice of notation, we will never use $e$ to denote the base of the natural logarithm, but will rather write $\exp(1)$.)

Let $N:=\binom{n}{2}$.  For a graph $H$ with $v$ vertices, and $e$ edges, the expected number of embeddings (isomorphic copies) of $H$ in $G(n,m)$ is
\eq{Ldef}
L_{H}(m)\, :=\, \frac{(n)_v(m)_e}{(N)_e}\, ,
\eqe
where $(n)_k:=n(n-1)\dots (n-k+1)$ denotes the falling factorial.  It will be useful at times to note that
\eq{Lident}
L_{H}(m)\, -\, L_{H}(m-1)\,  =\, \frac{1}{N-m+1}\sum_{f\in E(H)}\big(L_{H\setminus f}(m-1)\, -\, L_{H}(m-1)\big)\, .
\eqe
The intuition behind the identity is that both sides represent the increase in the expected number of embeddings of $H$ caused by the addition of an edge: the sum on the right hand side corresponds to the expected number of almost complete embeddings, in the sense that a single edge is not present.  Alternatively, direct calculation shows that both sides have value $e(n)_v(m-1)_{e-1}/(N)_e$.

A natural way to generate $G(n,m)$ is to add the edges one at a time.  The Erd\H os-R\'enyi random graph process $(G_i:i=0,\dots ,N)$ is defined as follows.  Let $G_0$ be the empty graph, and for each $i\ge 0$ let $G_{i+1}$ be obtained by adding a uniformly chosen edge to $G_i$.  Clearly $G_m$ is distributed as $G(n,m)$.  The process ends with $G_N$ being the complete graph $K_n$.  We observe that the process is Markovian.  We refer the reader to the books \cite{Bela, JLR} for further background on random graphs.

Our focus will be on subgraph count deviations in $G_m$.  We write $D_H(G_m)$ for the deviation of the $H$-count in $G_m$.  That is,
\eq{Ddef}
D_H(G_m)\, :=\, N_H(G_m)\, -\, L_H(m)\, .
\eqe
We shall see that paths of length two, which we denote $\owedge$, and triangles, which we denote $\triangle$, play a particularly important role.  We write $\binom{H}{\owedge}$ for the number of paths of length two in a graph $H$ and $\binom{H}{\triangle}$ for the number of triangles in $H$.

Let us define the function $\gamma_H(t)$ for $t\in (0,1)$ by
\eq{gdef}
\gamma_H(t)\, :=\, \left(4\obinom{H}{\owedge}^2 t^{2e-2}(1-t)^2\, +\, 12\obinom{H}{\triangle}^2 t^{2e-3}(1-t)^3\right)^{-1}\, .
\eqe

We now state our main result concerning the asymptotic rate of moderate deviations of subgraph counts.  We use the notation $f\ll g$ for $f=o(g)$.  We express the deviation as a multiple of $n^{v-3/2}$ as this is the order of the standard deviation (in the dense case).

The model we consider is defined as follows.  Let $(G_{n,m}:m=0,\dots ,N)$, $n\ge 1$ be independent copies of the Erd\H os-R\'enyi random graph process, and let $(G_{n,t})_{n\ge 1}$ denote the sequence of random graphs $(G_{n,m_{n}})_{n\ge 1}$, where $m_n=\lfloor tN\rfloor$.  We will be interested in $G_{n,t}$ both in the case that $t\in (0,1)$ is a constant, and the case that $t=t(n)$ is a function of $n$.  We exclude the case that $t(n)$ converges to $1$ (see Remark~\ref{rem:to1}).

\begin{samepage}
\begin{theorem}\label{thm:main} Let $t=t(n)\in (0,1)$ be a sequence bounded away from $1$, and let $H$ be graph with $v$ vertices, $e$ edges, and $\binom{H}{\owedge}\ge 1$.    Then
\[
\pr{D_H(G_{n,t})\phantom{\Big|} >\, \alpha_n n^{v-3/2}}\, =\, \exp\big(-\gamma_H(t) \alpha_n^2 (1+o(1))\big)\, ,
\]
for every sequence $(\alpha_n:n\ge 1)$ which satisfies either
\begin{enumerate}
\item[(i)] $1\ll \alpha_n \ll n^{1/2}$ and $t(n)=t\in (0,1)$ is constant, or
\item[(ii)]   $\max\{t^{1/2}n^{-1/2}\log{n}, t^{e-3/2}\} \, \ll\,  \alpha_n\, \ll\, \min\{t^{2e-5/2}n^{1/2}, t^{e+2}n^{1/2}\}$.
\end{enumerate}

Furthermore the same holds for $\pr{D_H(G_{n,t})\, <\, -\alpha_n n^{v-3/2}}$.
\end{theorem}
\end{samepage}

\begin{remark} We initially proved the results of this article in the dense case (i.e., with $t\in (0,1)$ a constant), and have now partially extended them to sparser regimes.  The problem of finding the asymptotic rate across the whole range of sparse densities remains open. \end{remark}

\begin{remark} In the sparse case, $t=o(1)$, we may simplify $\gamma_H(t)$ to
\begin{enumerate}
\item[(i)] $\big(4\obinom{H}{\owedge}^2 t^{2e-2}(1-t)^2\big)^{-1}$ in the case $\obinom{H}{\triangle}=0$, or
\item[(ii)] $\big(12\obinom{H}{\triangle}^2 t^{2e-3}(1-t)^3\big)^{-1}$ in the case $\obinom{H}{\triangle}\ge 1$.
\end{enumerate}
We may also note that the same dichotomy applies to $\Lambda_H(G_{n,t})$, see~\eqr{Ladef}, in the sense that the term involving $D_{\triangle}(G_{n,t})$ dominates, in the sparse case, if $\obinom{H}{\triangle}\ge 1$.
\end{remark}

\begin{remark}\label{rem:to1} Our proof of Theorem~\ref{thm:main} breaks down as $t$ approaches $1$.  However, with an alternative approach one may obtain the same bound provided:
\[
(1-t)^{e-3/2}\, \ll\, \alpha_n\, \ll\, (1-t)^{e+2}n^{1/2}\, .
\]
The alternative approach is to approximate $D_H(G_{n,t})$ by $\Lambda_{H}(G_{n,t})$ (using Theorem~\ref{thm:relate}), and apply Corollary~\ref{cor:comp} to each of $D_{\owedge}(G_{n,t})$ and $D_{\triangle}(G_{n,t})$ to express these deviations in terms of deviations in the complement, and then apply Theorem~\ref{thm:main} to the complement.  (Since the complement is sparse the deviation event is more easily achieved by $D_{\owedge}$ and the contribution of $D_{\triangle}$ is essentially trivial.) 
\end{remark}

\begin{remark} In the dense case, $t\in (0,1)$ constant, the range of deviations considered $(\omega(n^{v-3/2}),o(n^{v-1}))$, corresponds to the range strictly between the orders of magnitude of the standard deviation of $D_{H}(G)$ for $G\sim G(n,m)$ and $G\sim G(n,p)$ respectively.  This range is best possible, in the sense that the asymptotics of $\log(\pr{D_H(G_{n,t})>\, \alpha_n n^{v-3/2}})$ are different if $\alpha_n=O(1)$ or $\alpha_n=\Omega(n^{1/2})$.  For $\alpha_n=O(1)$ this follows from the central limit theorem of Janson~\cite{Jan}.  On the other hand, if $\alpha_n=\Omega(n^{1/2})$ then the asymptotic log probability is larger\footnote{As a particular example, if any vertex has degree $n-1$ then $D_{\owedge}(G_{n,t})\, \ge\, (1-t)^2n^2$, and this has probability at least $\Omega(t^n)$ which is larger than $\exp(-\gamma_{\owedge}(t) (1-t)^4 n (1+o(1)))$ for certain values of $t\in (0,1)$.}.  Theorem~\ref{thm:upto} below gives an exponent which is best possible up to multiplication by constant (in the dense case) across the whole range of deviations $(\omega(n^{v-3/2}),\Theta(n^v))$.
\end{remark}

A key step in proving Theorem~\ref{thm:main} is to establish a relation between the subgraph count deviations $D_H(G_{n,t})$ of different graphs $H$.  Specifically, we prove that $D_H(G_{n,t})$, the deviation of the $H$-count in $G_{n,t}$ is generally well predicted by a certain linear combination of $D_{\owedge}(G_{n,t})$, the deviation of the $P_2$ count, and $D_{\triangle}(G_{n,t})$, the deviation of the triangle count.  Let us define
\eq{Ladef}
\Lambda_{H}(G_{n,t})\, :=\, n^{v-3}t^{e-2}\left(\obinom{H}{\owedge}-3\obinom{H}{\triangle}\right)  D_{\owedge}(G_{n,t})\, +\, n^{v-3}t^{e-3}\obinom{H}{\triangle} D_{\triangle}(G_{n,t})
\eqe
to be this linear combination, where $v=v(H)$ and $e=e(H)$.  Note that $\Lambda_{H}(G_{n,t})$ is $n^{v-3}$ times a linear combination $\kappa D_{\owedge}(G_{n,t})+\rho D_{\triangle}(G_{n,t})$, in which the coefficients depend only on $H$ and $t$.

\begin{theorem}\label{thm:relate} 
Let $H$ be a graph with $v$ vertices and $e$ edges.  There exists a constant $C=C(H)$ such that for all $n$, and all $t=t(n)\in (0,1)$, we have
\eq{of two}
\pr{\big| D_H(G_{n,t})-\Lambda_H(G_{n,t})\big |\phantom{\Big|}>\, Cbt^{1/2}n^{v-2}}\, \le\, \exp(-b)
\eqe
for all $3\log{n}\le b\le t^{1/2} n$.  Furthermore
\eq{oftwoer}
\pr{\big| D_H(G_{n,t})-\Lambda_H(G_{n,t})\big |\phantom{\Big|}>\, Cbn^{v-2}}\, \le\, \exp(-b)
\eqe
for all $b\ge 3\log{n}$. 
\end{theorem}

We also state a weaker version of Theorem~\ref{thm:main} which applies across the entire range of possible deviations.

\begin{theorem}\label{thm:upto} 
Let $H$ be graph with $v$ vertices and $e$ edges.    Then there is a constant $c=c(H)$ such that for all $t=t(n)\in (0,1)$, and for all $\alpha, n \ge c^{-1}$, we have 
\[
\pr{|D_H(G_{n,t})|\phantom{\Big|} >\, \alpha n^{v-3/2}}\, \le \, \exp\big(-c \alpha  \min\{\alpha, n^{1/2}\}\big)\, .
\]
\end{theorem}

\subsection*{A discussion of our approach}

Our main results, and Theorem~\ref{thm:main} in particular, are proved using a pair of lemmas of Freedman~\cite{F}, stated in Section~\ref{sec:ineqs}, which provided an upper and a lower bound on deviation probabilities of martingales.  In particular, in certain circumstances, they imply that the probability that a martingale $(S_i)_{i=0}^{m}$ has a certain deviation $\alpha$ from its mean, is given by
\[
\exp\left(\frac{-\alpha^2 \, (1+o(1))}{2\beta}\right)\, ,
\]
where $\beta$ is asymptotic to the discrete quadratic variation
\[
\sum_{i=1}^{m}\Ex{(S_{i}-S_{i-1})^2\, \big|\,\F_{i-1}}\, 
\]
of the process.

In order to apply these results in our setting we are presented with two main challenges.  The first is to give a martingale expression for subgraph count deviations $D_H(G_m)$.  We state both a precise martingale expression for $D_H(G_m)$, see Theorem~\ref{thm:Mart}, and an approximate (but simpler) martingale expression for $D_H(G_m)$, see Theorem~\ref{thm:approx}.  The precise martingale expression, Theorem~\ref{thm:Mart}, is relatively easy to prove.  To verify the accuracy of the approximate martingale expression, Theorem~\ref{thm:approx} is substantially more difficult and this is done in Section~\ref{sec:relate}, as part of the proof of Theorem~\ref{thm:relate}.

The second challenge is to understand the behaviour of the discrete quadratic variation of these martingale expressions.  The relevant result, Proposition~\ref{prop:var}, which follows from the more precise Proposition~\ref{prop:varbetter}, allows us to deduce that this discrete quadratic variation is very predictable -- it is very likely to be close to a particular deterministic function.

Our proof of Proposition~\ref{prop:varbetter} makes use of Theorem~\ref{thm:relate}, which concerns the relationship between subgraph count deviations, and Theorem~\ref{thm:upto}.

We remark that the Hoeffding-Azuma inequality, Lemma~\ref{lem:HA}, is simpler to use than Freedman's inequality and for this reason we use it to prove various auxiliary results.  However, we stress that the main theorem itself, Theorem~\ref{thm:main}, could not be proved using the Hoeffding-Azuma inequality.  In essence, the Hoeffding-Azuma inequality gives substantially weaker bounds than Freedman's inequality when the martingale increments are typically much smaller than their maximum possible value; more precisely, when the conditional second moment of the increments, $\Ex{X_i^2|\F_{i-1}}$, are typically much smaller than their essential supremum, $\|X_i\|_{\infty}$.

\begin{remark} We developed this discrete martingale approach to understanding subgraph count deviations precisely because this approach combines well with results, such as those of Freedman, about discrete martingales.  We would like to acknowledge that a continuous time martingale framework for subgraph counts, and random graph statistics in general, was developed by Janson~\cite{Jan} in the 1990s.  There are number of connections between our results and those of Janson.  In particular, the significance of $P_2$ and triangle counts is also evident from Janson's results.  We encourage the interested reader to read~\cite{Jan} for results on the central limit theorem in $G(n,m)$, results on functional limits of random graph statistics, and much more.

\end{remark}

\subsection{Moderate deviations of subgraph counts in $G(n,p)$}

Until this point we have focussed exclusively on deviation events in the model $G(n,m)$.  We now deduce results concerning the probabilities of moderate deviations of subgraph counts in the Erd\H os-R\'enyi random graph $G(n,p)$.  We write $q$ for $1-p$ here and throughout.

We shall suppress $n$ from the notation and write $G_p$ for a graph chosen according to the distribution $G(n,p)$, i.e., with each edge included in $G_p$ independently with probability $p$.  For a graph $H$ with $v$ vertices and $e$ edges, we write
\[
L_{H}(p)\, :=\, (n)_v p^e 
\]
for the expected number of isomorphic copies of $H$ in $G_p$, and
\[
D_H(G_p)\, :=\, N_H(G_p)\, -\, L_{H}(p)
\]
for the deviation of the $H$-count $N_H(G_p)$ from its mean.  We consider deviations of size $\delta_n L_H(p)$ (that is, $\delta_n$ times the mean) where $n^{-1}\ll \delta_n\ll 1$.  This corresponds to the range strictly between the standard deviation and the regime of large deviations (i.e., the order of the mean).

Our first result corresponds to the range $\delta_n\ll n^{-1/2}$.  In this range we obtain a precise asymptotic expression for the deviation probability.  We remark that this result has already been obtained using a completely different approach by F\'eray, M\'eliot and Nikeghbali, see Theorem 10.1 in~\cite{FGN}.  Their result, which is proved in the framework of mod-$\phi$ convergence, also gives an asymptotically tight expression for deviation probabilities in this range.

\begin{theorem}\label{thm:smalldelta}
Let $p \in (0,1)$, and let $H$ be a graph with $v$ vertices and $e$ edges.  Let $(\delta_n : n \ge 1)$ be a sequence such that $n^{-1} \ll \delta_n \ll n^{-1/2}$. 
 Then
\begin{align*}
& \pr{D_H(G_{n,p}) >  \delta_n  p^e(n)_v} \phantom{\Big|}\\
&  =\, (1+o(1)) \sqrt{\frac{e^2q}{\pi p}} \exp\left(- \frac{\delta_n^2 pn^2}{4 e^2q}\, +\, \frac{\big((3e-2)-(3e-1)p\big) \delta_n^3 pn^2}{12e^3q^2}\, -\, \log (n\delta_n)  \right)\, .
\end{align*}
\end{theorem}


\begin{remark}
Observe that the only dependence on the graph $H$ in this range is via the numbers of edges $e$.  This is related to the fact that it is vastly easier to achieve this deviation by having extra edges in $G_p$ than achieving the deviation in $G_m$ for $m\approx pN$.  In other words, the above expression corresponds to the probability of the appropriate deviation of the binomial distribution.
\end{remark}

The range of larger deviations, $n^{1/2}\ll \delta_n \ll 1$, is more difficult to study in that there is a non-trivial interplay between the deviation probabilities of the binomial distribution and subgraph count deviations in $G_m$.  In particular, we require the full strength of Theorem~\ref{thm:main} to obtain the following result.  It is for this reason that $\gamma_H(p)$ appears in the rate.

We will also require the following notation.  Recall that $N:=\binom{n}{2}$. Set
\[
x_*\, :=\, \left[(1+\delta_n)^{1/e}\, -\, 1\right]\, \sqrt{\frac{pN}{q}}\, ,
\]
and for $0<x<\sqrt{N}/2$ define
\[
E(x,N)\, =\, \sum_{i=1}^{\infty} \frac{(p^{i+1} + (-1)^i q^{i+1})x^{i+2}}{(i+1)(i+2) p^{i/2} q^{i/2} N^{i/2}}\, .
\]
We can now state our result for larger deviations.  In fact the result may be stated across the whole range $n^{-1}\ll \delta_n\ll 1$.

\begin{theorem}\label{thm:largerdelta}
Let $p \in (0,1)$, and let $H$ be a graph with $v$ vertices and $e$ edges.  Let $(\delta_n : n \ge 1)$ be a sequence such that 
$n^{-1} \ll \delta_n \ll 1$.  Then
\begin{align*}
& \pr{D_H(G_{n,p}) > \delta_n p^e(n)_v}\phantom{\Big|}\\
&  \qquad =\, \exp \left( -\frac{x_*^2}{2}\, +\, E(x_*,N)\, +\, (1+o(1))\frac{\delta_n^2 n}{16\gamma_H(p) e^4p^{2e-2}q^2}\, +\, O(\log{n})\right) \, .
\end{align*}
\end{theorem}

We remark that the asymptotic rate, which gives the bound
\begin{equation}\label{eq:asy}
\exp \left( -\frac{x_*^2}{2}(1+o(1))\right)\, =\, \exp\left(\frac{-\delta_n^2 p n^2}{4e^2 q}\, +\, o(\delta_n^2 n^2)\right)\, ,
\end{equation}
already appears in the articles of D\"oring and Eichelsbacher~\cite{DE} and~\cite{DE2}.  The difference between the results is the order of magnitude of the error term.  In the range $\delta_n\gg n^{-1/2}\sqrt{\log{n}}$ we have an error term of the form $o(\delta_n^2 n)$ in the exponent\footnote{We believe that it ought to be possible to reduce the error term in the missing range $\Omega(n^{-1/2})\le \delta_n\le O(n^{-1/2}\sqrt{\log{n}})$ to be of the form $o(\delta_n^2 n)$, rather than $O(\log{n})$.  For example, one might prove this by combining our approach with the central limit theorem for subgraph count deviations.}

On the other hand, D\"oring and Eichelsbacher in~\cite{DE} obtained the asymptotic rate for the range of parameters
\[
\sqrt{\frac{q}{p}} \, n^{-1}\, \ll\, \delta_n\, \ll\, p^{3e-2}q^2\, 
\]
by an estimation of the log-Laplace transform and the Gartner-Ellis theorem.
In~\cite{DE2} D\"oring and Eichelsbacher show that results may also be obtained through a moderate deviation principles via cumulants, in an approach based on a celebrated lemma of large deviations theory due to Rudzkis, Saulis and Statulevicius.  The results of ~\cite{DE2} include the asymptotic rate for the range of parameters
\[
\sqrt{\frac{q}{p}} \, n^{-1}\, \ll\, \delta_n\, \ll\, p^{(3e-4)/5}q^{4/5} \, n^{-4/5}\, .
\]

It may be of interest to investigate for which ranges of $\delta_n$ and $p$ our more precise expansion remains valid.  Janson and Warnke~\cite{JW} focussed on the lower tail and found the same asymptotic expression,~\eqref{eq:asy}, for the logarithm of $\pr{D_H(G_{n,p})< -\delta_n p^e(n)_v}$ across the whole range of moderate deviations and densities $p\ll n^{-1/m_2(H)}$ where $m_{2}(H)=\max_{J\subseteq H}(e(J)-1)/(v(J)-2)$.  Furthermore, their result also applies in the setting of $k$-uniform hypergraphs.

We also remark that a weaker result with the $1+o(1)$ replaced by $O(1)$ may be proved using only Theorem~\ref{thm:upto} to bound deviation probabilities for $D_H(G_m)$.  In this sense Theorem~\ref{thm:main} has a relatively minor impact on the strength of the bound obtained for deviations $D_H(G_p)$ in $G(n,p)$.  On the other hand, we believe that this reinforces our argument that $G(n,m)$ is the more natural setting in which to study these subgraph count deviations in the first place.

Finally, the reader may wonder why we gave the implicit definition $-x_*^2/2\, +\, E(x_*,N)$ rather than just writing out the expansion.  The problem is that the number (and complexity) of the terms in the expansion grows as $\delta_n$ increases.  We illustrate this by giving the expansion in the range $n^{-1/2}\log{n}\ll \delta_n\ll n^{-2/5}$. 

\begin{cor}\label{cor:largerdelta}
Let $p \in (0,1)$, and let $H$ be a graph with $v$ vertices and $e$ edges.  Let $(\delta_n : n \ge 1)$ be a sequence such that 
$n^{-1/2}\log{n} \ll \delta_n \ll n^{-2/5}$.  Then
\begin{align*}
& \pr{D_H(G_{n,p}) > \delta_n p^e(n)_v}\phantom{\Big|}\\
&=\,  \exp\Bigg( \!\! -\frac{\delta_n^2 p n^2}{4qe^2} -\frac{p[(3e-1)q-1]\delta_n^3 n^2}{12q^2e^3} +\frac{p\big[(e-1)q[(8e+11)q-6]+1-3pq\big]\delta_n^4n^2}{48q^3e^4} \\
&\hspace{8cm} \qquad+ (1+o(1))\frac{\delta_n^2 n}{16\gamma_H(p) e^4p^{2e-2}q^2}\Bigg)\, .
\end{align*}
\end{cor}

Naturally, both results will rely on an estimate for tail probabilities of the binomial distribution.  While estimates are available (Littlewood~\cite{Little} for example, see also McKay~\cite{McKay}), we shall give a proof of the following estimate for completeness.  This result is essentially due to Bahadur~\cite{Bah}.  In addition to $E(x,N)$ defined above, let us also define a truncated version of the sum:
\[
E(x,N,J)\, =\, \sum_{i=1}^{J} \frac{(p^{i+1} + (-1)^i q^{i+1})x^{i+2}}{(i+1)(i+2) p^{i/2} q^{i/2} N^{i/2}}\, .
\]

Adapting the argument from Theorem 2 of Bahadur~\cite{Bah}, we obtain the following asymptotics for
\[
b_N(k) = \pr{\mathrm{Bin}(N,p) = k}
\]
and
\[
B_N(k) = \pr{\mathrm{Bin}(N,p) \ge k}.
\]
in terms of $x_N = \frac{k-pN}{\sqrt{Npq}}$.  The theorem is valid for $p\in (0,1)$ a constant or $p=p_N$ a function.

\begin{samepage}
\begin{theorem} \label{thm:bah}
Suppose that $(x_N)$ is a sequence such that $1\ll x_N\ll \sqrt{Npq}$. Then
\[
b_N(\fl{pN+x_N\sqrt{Npq}})\,  =\, (1 + o(1))\frac{1}{\sqrt{2 \pi Npq}} \exp \left(-\frac{x_N^2}{2} - E(x_N,N) \right)
\]
and
\[
B_N(pN + x_N \sqrt{Npq})\, =\, (1 + o(1)) \frac{1}{x_N \sqrt{2\pi}} \exp \left(- \frac{x_N^2}{2} - E(x_N,N) \right)\, .
\]
Furthermore, if $1\ll x_N\ll (pqN)^{1/2} (pqN)^{-1/(J+3)}$ then the infinite sum $E(x_N,N)$ may be replaced by the finite sum $E(x_N,N,J)$ in both expressions. 
\end{theorem}
\end{samepage}

The proof of Theorem~\ref{thm:bah} is given in the appendix.

Let us now return to Theorems~\ref{thm:smalldelta} and~\ref{thm:largerdelta} and give an overview of their proofs.  We immediately observe, by conditioning on the number of edges of $G_p$, that we may express $\pr{D_H(G_p)\, >\, \delta_n p^e(n)_v}$ as a sum:
\[
\pr{D_H(G_p)\, >\, \delta_n p^e(n)_v}\, =\, \sum_{m=0}^{N}b_N(m)\, \pr{N_H(G_m)\, >\, (1+\delta_n)p^e (n)_v}\, .
\]
For $m>pN$ we have that the first term ($b_N(m)$) is decreasing while the second is increasing.  The proofs are therefore concerned with identifying which terms make the largest contribution.  

In the case of Theorem~\ref{thm:largerdelta} the problem reduces exactly to a calculation of the maximum, as all other effects are swallowed up in the $O(\log{n})$ error term in the exponent.

In the case of Theorem~\ref{thm:smalldelta} we exploit the fact that there is an interval $[m_-,m_+]$ over which the first term ($b_N(m)$) decreases very little and the second term grows from $o(1)$ to $1-o(1)$.

\subsection*{Layout of the article}

In Section~\ref{sec:Mart} we present Theorem~\ref{thm:Mart}, the general martingale expression for the subgraph count deviation $D_H(G_m)$.  We also present an important approximate representation, Theorem~\ref{thm:approx}, and a lemma relating subgraph counts to subgraph counts in the complementary graph.  In Section~\ref{sec:ineqs} we state the martingale inequalities that we shall use throughout the article.  These include the Hoeffding-Azuma inequality, a related inequality adapted to $G(n,m)$ and, crucially, Freedman's inequalities for the probability of deviations of martingales.  In Section~\ref{sec:degrees} we prove bounds concerning the behaviour of degrees and codegrees in $G(n,m)$.  In Section~\ref{sec:relate} we prove the approximate representation result, Theorem~\ref{thm:approx}, and deduce Theorem~\ref{thm:relate}.  

We then turn our focus towards deviation probabilities themselves.  In Section~\ref{sec:upto}, we prove Theorem~\ref{thm:upto}, which gives a general though not especially precise bound on subgraph count probabilities. In order to prove the tighter result Theorem~\ref{thm:main}, we must first understand better the variances and covariances of the increments of the martingale representation.  In Section~\ref{sec:var} we prove bounds on general covariances of increments in the martingale representation of $D_H(G_m)$, and in Section~\ref{sec:main} we prove Theorem~\ref{thm:main}.

Finally, in Section~\ref{sec:pproofs} we deduce our results for subgraph count deviations in $G(n,p)$.

\subsection*{Notation}

Throughout $N$ denotes $\binom{n}{2}$.  Let us also recall that $N_H(G)$ denotes the number of embeddings of a graph $H$ in a graph $G$, and that $\binom{G}{H}$ denotes the number of copies of $H$ in $G$ counted without multiplicity.

Use of $m$ and $t$:  In Sections~\ref{sec:Mart},~\ref{sec:ineqs} and ~\ref{sec:degrees}, we work with the Erd\H os-R\' enyi random graph process and use $m$ simply to denote the number of edges of $G_m$.  In later sections $m$ is used specifically to refer to $\lfloor tN\rfloor$.   The latter use corresponds to the use in the definition of $G_{n,t}$ as $G_{n,m}$ with $m=\lfloor tN\rfloor$.

Use of $i$ and $s$:  We think of $G_m$ as the result of a realisation of the random graph process $(G_i:i=0,\dots ,m)$.  In this context we use $s$ throughout to refer to $i/N$, the proportion of pairs that occur as edges of $G_i$.  This usage occurs below in the definitions of $\X_H(G_m), \Lambda^{**}_{H}(G_{n,t}), V_{F,F'}(i,n)$ and $W_{F,F'}(G_{i-1})$, for example.

Use of $v$ and $e$: We use $v$ and $e$ to denote the number of vertices and edges of the small graph we are currently working with.  The majority of the time this is the graph $H$.  However, in Section~\ref{sec:Ysmall} and Section~\ref{sec:var} it is the graph $F$.  When necessary we write $v(F)$ and $e(F)$, for example, to avoid ambiguity.

\noindent\begin{tabular}{lr}
Introductory notation:\phantom{\bigg)}\hspace{7cm} & First introduced:\\
\(\displaystyle L_{H}(m)\, :=\, \frac{(n)_v(m)_e}{(N)_e} \) &
\eqr{Ldef}\\
\(\displaystyle D_{H}(G_m)\, :=\, N_{H}(G_m)\, -\, L_{H}(m) \phantom{\Bigg)}\) & \eqr{Ddef}\\
\(\displaystyle A_{H}(G_m)\, :=\, N_{H}(G_m)\, -\, N_{H}(G_{m-1}) \phantom{\Bigg)}\)& \eqr{Adef}\\
\(\displaystyle \gamma_H(t)\, :=\, \left(4\obinom{H}{\owedge}^2 t^{2e-2}(1-t)^2\, +\, 12\obinom{H}{\triangle}^2 t^{2e-3}(1-t)^3\right)^{-1}\phantom{\Bigg)}\) & \eqr{gdef}
\end{tabular}

\noindent\begin{tabular}{lr}
Martingale increments related to subgraph count deviations: \qquad& \hspace{-2cm}First introduced:\\
\(\displaystyle X_{H}(G_m)\, :=\, A_{H}(G_m)\, -\, \Ex{A_{H}(G_m)\, \big|\,G_{m-1}}\phantom{\Bigg)} \)& \eqr{Xdef}\\
\( \displaystyle X^*_F(G_i)\, := \, n^{v-3}s^{e(F)-2}\left(\obinom{F}{\owedge}-3\obinom{F}{\triangle}\right)  X_{\owedge}(G_i)\, +\, n^{v-3}s^{e(F)-3}\obinom{F}{\triangle}X_{\triangle}(G_i)\phantom{\Bigg)}\) & \eqr{Xsdef}\\
\( \displaystyle Y_{F}(G_i)\, :=\, X_{F}(G_i)\, -\, X^{*}_F(G_i)\phantom{\Bigg)}\) &\eqr{Ydef}
\end{tabular}

Additionally, $\X_H(G_i;t)$, see~\eqr{XXdef}, is defined by
\begin{align*}
& \X_{H}(G_i;t)\, \\
& \quad :=\, n^{v-3}t^{e-3}\left(t\obinom{H}{\owedge}\frac{(1-t)^2}{(1-s)^2}X_{\owedge}(G_i)\, +\, \obinom{H}{\triangle} \frac{(1-t)^3}{(1-s)^{3}}\, \big(X_{\triangle}(G_i)-3sX_{\owedge}(G_i)\big) \right)\, .
\end{align*}

There are three random variables $\Lambda_{H}(G_{n,t}), \Lambda^*_{H}(G_{n,t})$ and $\Lambda^{**}_{H}(G_{n,t})$ that all approximate $D_H(G_{n,t})$ in some sense.  They are first defined respectively as equations~\eqr{Ladef},~\eqr{Lasdef} and~\eqr{Lassdef}.  In the following definition $m$ denotes $\lfloor tN\rfloor$ and $\X_H(G_i,t)$ and $X_F(G_i)$ are as defined above.
\begin{align*}
\Lambda_{H}(G_{n,t})\, & :=\,  n^{v-3}t^{e-2}\left(\obinom{H}{\owedge}-3\obinom{H}{\triangle}\right)  D_{\owedge}(G_{n,t})\, +\, n^{v-3}t^{e-3}\obinom{H}{\triangle} D_{\triangle}(G_{n,t})  \\
\Lambda^{*}_H(G_{n,t})\, & := \, \sum_{i=1}^{m}\X_{H}(G_i;t) \\
\Lambda^{**}_H(G_{n,t})\, & := \, \sum_{i=1}^{m}\, \sum_{F\ssq E(H)}\frac{(1-t)^{e(F)}(t-s)^{e-e(F)}}{(1-s)^e}\, X_F(G_i)
\end{align*}

\noindent\begin{tabular}{lr}
Degrees, codegrees and their deviations:& \hspace{2.3cm} First introduced:\\
\(\displaystyle d_u(G_m)\, :=\, \text{degree of $u$ in $G_m$}\phantom{\Bigg)}\) & \\
\(\displaystyle D_{u}(G_m)\, :=\, d_u(G_m)\, -\, \frac{2m}{n}\phantom{\Bigg)}   \)&\eqr{Dudef} \\
\(\displaystyle d_{u,w}(G_m)\, :=\, \text{codegree of $u$ in $G_m$}\phantom{\Bigg)} \)& \\
\(\displaystyle D_{u,w}(G_m)\, :=\, d_{u,w}(G_m)\, -\, \frac{(n-2)(m)_2}{(N)_2}\phantom{\Bigg)}\)& \eqr{Duwdef}\\
\(\displaystyle \Delta(e_i)\, :=\, D_{u}(G_{i-1})^2+D_{w}(G_{i-1})^2+D_{uw}(G_{i-1})^2\phantom{\Bigg)}\)& \eqr{Deldef}
\end{tabular}

\noindent\begin{tabular}{lr}
Functions related to covariance: & \hspace{-0.8cm} First introduced:\\
\(\displaystyle V_{F,F'}(i,n)\, :=\, n^{v+v'-5}s^{e+e'-4}(1-s)\big(s\theta_1(F,F')+(1-s)\theta_2(F,F')\big)\phantom{\Bigg)} \) & \eqr{Vdef}\\
\(\displaystyle \theta_1(F,F')\, :=\, 8\obinom{F}{\owedge}\obinom{F'}{\owedge} \phantom{\Bigg)} \) & \eqr{Tdef} \\
\(\displaystyle \theta_2(F,F')\, :=\, 36\obinom{F}{\triangle}\obinom{F'}{\triangle} \phantom{\Bigg)} \) & \eqr{Tdef} \\
\(\displaystyle W_{F,F'}(G_{i-1})\, :=\, 8n^{v+v'-7} s^{e+e'-4}\obinom{F}{\owedge}\obinom{F'}{\owedge} D_{\owedge}(G_{i-1})\phantom{\Bigg)} \) & \eqr{Wdef}
\end{tabular}

Finally, we write $\log$ for the natural logarithm.


\newpage

\section{Martingale expression for $D_{H}(G_m)$}\label{sec:Mart}

In this section we state and prove Theorem~\ref{thm:Mart}, our martingale expression for $D_H(G_m)$.  In doing so, we also prove Lemma~\ref{lem:expW}, which concerns the expected number of copies of $H$ created with the addition of the $m$th edge.  We also state an approximate expression for $D_H(G_m)$, see Theorem~\ref{thm:approx}, which will be proved in Section~\ref{sec:relate} as part of the proof of Theorem~\ref{thm:relate}.

For the duration of the section, let us fix $n$ and let $(G_m:m=0,\dots ,N)$ be a realisation of the Erd\H os-R\'enyi random graph process on $n$ vertices.  It is helpful to think of $G_m$ as also including the information of the order in which its edges were added.  Let us define 
\eq{Adef}
A_{H}(G_m)\, :=\, N_{H}(G_m)\, -\, N_{H}(G_{m-1})\, ,
\eqe
the number of embeddings (isomorphic copies) of $H$ created with the addition of the $m$th edge.  Our martingale expression for $D_H(G_m)$ will be based on centered versions of these random variables.  Let
\eq{Xdef}
X_{H}(G_m)\, :=\, A_{H}(G_m)\, -\, \Ex{A_{H}(G_m)\, \big|\, G_{m-1}}\, .
\eqe
Note that $X_{H}(G_m)$ is obtained from $A_{H}(G_m)$ by shifting it so that $\Ex{X_{H}(G_m)|G_{m-1}} = 0$.

We may now state the martingale expression for $D_{H}(G_m)$.

\begin{theorem}\label{thm:Mart} Let $H$ be a graph with $v$ vertices and $e$ edges.  Then
\eq{inMart}
D_{H}(G_m) \, =\, \sum_{i=1}^{m}\, \sum_{F\ssq E(H)}\frac{(N-m)_{e(F)}(m-i)_{e-e(F)}}{(N-i)_e}\, X_F(G_i)\, ,
\eqe
where the inner sum is taken over all $2^e$ graphs $F$ with $V(F)=V(H)$ and $E(F)\ssq E(H)$.
\end{theorem}

\begin{remark}
Equation~\eqr{Xdef} shows that $X_{H}(G_m)$ is a martingale increment with respect to the natural filtration of $G_0,\dots ,G_N$.  Since
\[
\sum_{F\ssq E(H)}\frac{(N-m)_{e(F)}(m-i)_{e-e(F)}}{(N-i)_e}\, X_F(G_i)
\]
is a linear combination of the random variables $X_F(G_i)$, it too is a martingale increment, and so~\eqr{inMart} is indeed a martingale.
\end{remark}

We begin with a lemma about $\Ex{A_{H}(G_m)|G_{m-1}}$, the expected number of embeddings (isomorphic copies) of $H$ created with the $m$th edge, given the graph $G_{m-1}$.

\begin{lem}\label{lem:expW}
In the Erd\H os-R\'enyi random graph process $(G_m:m=0,\dots ,N)$,
\begin{align*}
& \Ex{A_{H}(G_m)\, \big|\, G_{m-1}}\, =\, \frac{1}{N-m+1}\sum_{f\in E(H)}\big(N_{H\setminus f}(G_{m-1})-N_{H}(G_{m-1})\big)\\
& \quad =\, \big(L_{H}(m)-L_H(m-1)\big)\,  +\, \frac{1}{N-m+1}\sum_{f\in E(H)} \big(D_{H\setminus f}(G_{m-1})-D_{H}(G_{m-1})\big)\, ,
\end{align*}
where $H\setminus f$ denotes the graph obtained from $H$ by removing the edge $f$.
\end{lem}

\begin{proof}
Let us first observe that the second equality follows directly from the definitions.  Indeed, one may simply expand $N_{H}(G_m)$ as $L_{H}(m)+D_{H}(G_m)$, and use~\eqr{Lident}.

We now prove the first equality.  We may view $A_H(G_m)$, the number of embeddings (isomorphic copies) of $H$ created with the addition of the $m$th edge, $e_m$, as a sum
\[
A_H(G_m)\, =\, \sum_{f\in E(H)} A_{H,f}(G_m)\, ,
\]
where $A_{H,f}(G_m)$ denotes the number of embeddings of $H$ created with the addition of $e_m$, in which $e_m$ is the image of the edge $f$ of $H$.  It therefore suffices to prove that
\eq{WHf}
\Ex{A_{H,f}(G_m)\, \big|\, G_{m-1}}\, =\, \frac{1}{N-m+1}\big(N_{H\setminus f}(G_{m-1})-N_{H}(G_{m-1})\big)
\eqe
for each $f\in E(H)$.

Fix $f\in E(H)$.  In order for an injective function $\phi:V(H)\to V(G_m)$ to represent an embedding of $H$ in $G_m$, but not in $G_{m-1}$, and have $e_m=\phi(f)$, it is necessary and sufficient that $\phi$ embeds $H\setminus \{f\}$ into $G_{m-1}$, that $\phi(f)$ is not an edge of $G_{m-1}$, and, finally, that $e_m$ is chosen to be $\phi(f)$.

The number of injective functions obeying the first two conditions is precisely $N_{H\setminus f}(G_{m-1})-N_{H}(G_{m-1})$, and, for any such $\phi$, the probability that $e_m$ is chosen to be $\phi(f)$ is $1/(N-m+1)$.  The required equation, \eqr{WHf}, follows immediately.
\end{proof}

We now prove Theorem~\ref{thm:Mart}.

\begin{proof} The proof is by induction on $e=e(H)$, and on $m\in \{0,\dots ,N\}$.  If $e=1$ or $m=0$ the result holds trivially.  Now consider a graph $H$ with $v$ vertices and $e\ge 2$ edges, and $m\in \{1,\dots ,N\}$.  We may expand $D_H(G_m)$ as follows
\begin{align*}
D_H&(G_m)\,  =\, N_{H}(G_m)\, -\, L_H(m)\phantom{\Big)} \\
& =\, N_{H}(G_{m-1})\, +\, A_H(G_m) \, - \, L_H(m-1)\, -\, \big(L_{H}(m)-L_H(m-1)\big)\phantom{\Big)} \\
& =\, D_H(G_{m-1})\, +\, A_H(G_m) \, -\, \big(L_{H}(m)-L_H(m-1)\big)\phantom{\Big)}\\
& =\,  D_H(G_{m-1})\, +\, X_H(G_m)\, +\, \Ex{A_{H}(G_m)\, \big|\, G_{m-1}} -\, \big(L_{H}(m)-L_H(m-1)\big)\phantom{\Big)}\\
& =\, D_H(G_{m-1})\, +\, X_H(G_m)\, + \,  \frac{1}{N-m+1}\sum_{f\in E(H)} \big(D_{H\setminus f}(G_{m-1})-D_{H}(G_{m-1})\big)\, , \phantom{\Big)}
\end{align*}
where we have used the definition of $X_{H}(G_m)$ in the fourth line and Lemma~\ref{lem:expW} in the last line.

We have an expression for $D_H(G_m)$ in terms of $X_H(G_m)$ and a linear combination of deviations $D_{F}(G_{m-1})$ with $F\subseteq H$.  By the induction hypothesis each of these may be expressed as a linear combination of the $X_{F}(G_i)$ with $F\subseteq E(H)$ and $1\le i\le m$.  One may check that the resulting expression for $D_H(G_m)$ is that claimed.
\end{proof}

We now give a simpler expression which approximates $D_H(G_m)$ very well.  Since this expression is itself closely related to the quantity $\Lambda_H(G_{n,t})$ which appears in Theorem~\ref{thm:relate}, we use the notation $\Lambda^*_H(G_{n,t})$.  For a graph $H$ with $v$ vertices and $e$ edges, let us first define
\eq{XXdef}
\X_{H}(G_i;t)\, :=\, n^{v-3}\left(t^{e-2}\obinom{H}{\owedge}\frac{(1-t)^2}{(1-s)^2}X_{\owedge}(G_i)\, +\, t^{e-3}\obinom{H}{\triangle} \frac{(1-t)^3}{(1-s)^{3}}\, \big(X_{\triangle}(G_i)-3sX_{\owedge}(G_i)\big) \right)
\eqe
where $s=i/N$, as it is throughout the article, and define
\eq{Lasdef}
\Lambda^{*}_H(G_{n,t})\, := \, \sum_{i=1}^{m}\X_H(G_i;t) .
\eqe

We are now ready to state Theorem~\ref{thm:approx}.  The statement will be given for $t\in (0,1/2]$.  This form is sufficient for the proof of Theorem~\ref{thm:relate}.

\begin{theorem}\label{thm:approx}  Let $H$ be a graph with $v$ vertices and $e$ edges.  There exists a constant $C=C(H)$ such that for all $t=t(n)\in (0,1/2]$ we have
\eq{approx}
\pr{\big| D_H(G_{n,t})-\Lambda^{*}_H(G_{n,t})\big |\phantom{\Big|}>\, Cbt^{1/2}n^{v-2}}\, \le\, \exp(-b)
\eqe
for all $3\log{n}\le b\le t^{1/2} n$.  Furthermore,
\eq{approxer}
\pr{\big| D_H(G_{n,t})-\Lambda^{*}_H(G_{n,t})\big |\phantom{\Big|}>\, Cbn^{v-2}}\, \le\, \exp(-b)
\eqe
for all $b\ge 3\log{n}$. 
\end{theorem}

\begin{remark}
The curious reader may wonder why we express the terms $\X_H(G_{i};t)$ of $\Lambda^{*}_H(G_{n,t})$ as a linear combination of $X_{\owedge}$ and $X_{\triangle}-3sX_{\owedge}$ rather than directly as a linear combination of $X_{\owedge}$ and $X_{\triangle}$.  We consider this choice natural because $X_{\owedge}$ and $X_{\triangle}-3sX_{\owedge}$ are asymptotically orthogonal (i.e., uncorrelated) in the sense that
\[
\Ex{X_{\owedge}(G_i)
\big(X_{\triangle}(G_i)-3sX_{\owedge}(G_i)\big)\,\Big|\, G_{i-1}}
\]
is typically $o(n)$, while their individual variances are $\Theta(n)$.  See Section~\ref{sec:var} for more details.
\end{remark}

In fact, the result holds for all $t\in (0,1)$.  This follows directly from Theorem~\ref{thm:relate} and~\eqr{lambneed}.

\begin{theorem}\label{thm:approxbetter}
Theorem~\ref{thm:approx} holds for $t=t(n)\in (0,1)$.
\end{theorem}

Theorem~\ref{thm:approx} is proved in Section~\ref{sec:relate} as part of the proof of Theorem~\ref{thm:relate}.  After proving Theorem~\ref{thm:relate} we easily deduce Theorem~\ref{thm:approxbetter}.

\subsection{An aside: subgraph counts from subgraph counts in the complement}

We record a simple lemma that allows one to relate subgraph counts in $G$ to subgraph counts in the complement $G^c$.

\begin{lem}\label{lem:comp} Let $H$ and $G$ be graphs, and let $G^c$ be the complement of $G$, then
\eq{incomp}
N_{H}(G)\, =\, \sum_{H'\subseteq E(H)}(-1)^{e(H')}N_{H'}(G^c)\, ,
\eqe
where the sum is over all $2^{e(H)}$ subgraphs of $H$.
\end{lem}

\begin{proof} Writing $N^{\ind}_H(G)$ for the number of induced embeddings (isomorphic copies) of $H$ in $G$ we have, by inclusion-exclusion, that
\[
N_{F}^{\ind}(G)\, =\,\sum_{F\subseteq H\subseteq K_v}(-1)^{e(H)-e(F)}N_{H}(G)\, ,
\]
and, in the other direction,
\[
N_{H}(G)\, =\, \sum _{H\subseteq F\subseteq K_v} N^{\ind}_F(G)\, =\,  \sum _{F^{c}\subseteq H^{c}} N^{\ind}_{F^c}(G^c)
\]
where $K_v$ is the complete graph on the vertex set of $H$.

We now have
\begin{align*} 
N_{H}(G)\,&  =\,  \sum _{F^{c}\subseteq H^{c}} N^{\ind}_{F^c}(G^c)\phantom{\Bigg)} \\
& =\, \sum_{F^{c}\subseteq H^{c}}\quad \sum_{F^c\subseteq H'\subseteq K_v}(-1)^{e(H')-e(F^c)}N_{H'}(G^c)\phantom{\Bigg)}\\
& =\, \sum_{H'\subseteq K_v}(-1)^{e(H')}N_{H'}(G^c) \,\sum_{H\cup (H')^c\subseteq F\subseteq K_v}(-1)^{e(F^c)}\phantom{\Bigg)}\\
& =\, \sum_{H'\subseteq H}(-1)^{e(H')}N_{H'}(G^c)\, ,\phantom{\Bigg)}
\end{align*}
where the last line follows since the sum over $F$ in the line above gives $1$ if $H\cup (H')^c=K_v$ and $0$ otherwise.
\end{proof}

By linearity, the same identity holds for deviations.

\begin{cor}\label{cor:comp} Let $H$ and $G$ be graphs, and let $G^c$ be the complement of $G$, then
\[
D_{H}(G)\, =\, \sum_{H'\subseteq E(H)}(-1)^{e(H')}D_{H'}(G^c)\, ,
\]
where the sum is over all $2^{e(H)}$ subgraphs of $H$.
\end{cor}

\begin{proof}
This follows easy from Lemma~\ref{lem:comp} by linearity.  Indeed, by taking expectation (with $G\sim G(n,e(G))$) on both sides of~\eqr{incomp} we obtain that
\[
L_{H}(e(G))\, =\, \sum_{H'\subseteq E(H)}(-1)^{e(H')}L_{H'}(e(G^c))\, .
\]
Subtracting this from~\eqr{incomp} gives the required identity.
\end{proof}

\section{Martingale deviation inequalities}\label{sec:ineqs}

In this section we state the Hoeffding-Azuma inequality~\cite{Azuma,Hoeff} which bounds the probability of martingale deviations.  The particular form of the Hoeffding-Azuma inequality we shall use is stated as Corollary~\ref{cor:HA}.

For certain key results, including our main theorem, we need to use an inequality of Freedman~\cite{F} instead. Freedman's inequality gives significantly stronger bounds in certain contexts; in particular when the martingale increments, $X_i$, have conditional second moments, $\Ex{X_i^2|\F_{i-1}}$, much smaller than $\|X_i\|_{\infty}^2$.

We begin with the Hoeffding-Azuma inequality and its corollary.  The corollary is an application of the inequality to functions $f(G_m)$, where $G_m\sim G(n,m)$.

Let $(S_n)_{n \ge 0}$ be a martingale with respect to a filtration $(\F_n)_{n \ge 0}$.  Write $X_i = S_i - S_{i-1}$, $i \ge 1$ for its \emph{increments} and note that $\Ex{X_i | \F_{i-1}} = 0$ for all $i \ge 1$.

\begin{lem}[Hoeffding-Azuma inequality]\label{lem:HA}
Let $(S_m)_{m=0}^{M}$ be a martingale with increments $(X_i)_{i=1}^{M}$, and let $c_i=\|X_i\|_{\infty}$ for each $1\le i\le M$.
Then, for each $a>0$, 
\[
\pr{S_M-S_0\, >\, a}\, \le \, \exp\left(\frac{-a^2}{2\sum_{i=1}^{M}c_i^2}\right)\, .
\]
Furthermore, the same bound holds for $\pr{S_M-S_0\, <\, -a}$.
\end{lem}

Let us write $\mathcal{G}_{n,m}$ for the family of graphs with $n$ vertices and $m$ edges.  One may think of $\mathcal{G}_{n,m}$ as endowed with an edit distance, in which graphs which differ in two edges, $G$ and $G\setminus \{e_i\}\cup \{e_j\}$ for example, have distance $1$.  It is then natural to say a function $f:\mathcal{G}_{n,m}\to \mathbb{R}$ is \emph{$C$-Lipschitz}, if $|f(G)-f(G')|\le C$ for all pairs of adjacent graphs $G,G'$.  

Given a function $\psi:E(K_n)\to \mathbb{R}^{+}$, let us say that a function $f:\mathcal{G}_{n,m}\to \mathbb{R}$ is \emph{$\psi$-Lipschitz} if for every adjacent pair of graphs $G,G'\in \mathcal{G}_{n,m}$ we have
\[
\big|\, f(G)\, -\, f(G')\, \big|\, \le\, \psi(e_i)\, +\, \psi(e_j)\, ,
\]
where $G\bigtriangleup G'=\{e_i,e_j\}$.

\begin{cor}\label{cor:HA} Given $\psi:E(K_n)\to \mathbb{R}^+$ and a $\psi$-Lipschitz function $f:\mathcal{G}_{n,m}\to \mathbb{R}$, we have
\[
\pr{f(G_m)\, -\,\Ex{f(G_m)}\, \ge \, a}\, \le\, \exp\left(\frac{-a^2}{8\|\psi\|_2^2}\right)
\]
for all $a\ge 0$, where $\|\psi\|_2^2:=\sum_{e\in E(K_n)}\psi(e)^2$.

Furthermore, the same bound holds for $\pr{f(G_m)\, -\,\Ex{f(G_m)}\, \le \, -a}$.
\end{cor}

\begin{remark} While we include a proof of this corollary for completeness, we do not claim that it is an original result.  The statement is very close in spirit to that of McDiarmid's concentration inequality~\cite{Colin}, although in a slightly different setting, as we do not have independence.  See also Warnke~\cite{Lutz}, where generalisations of McDiarmid's inequality are proved, including one where the independence condition may be weakened.\end{remark}

\begin{proof} Let $e_1,\dots ,e_N$ be an ordering of the edges of $K_n$ in which $\psi$ is decreasing.  Consider the martingale
\[
Z_i\, =\, \Ex{f(G_m)\, \big|\, G_m\cap\{e_1,\dots ,e_i\}}\, ,\vspace{0.1cm}
\]
where the conditioning indicates that we reveal the first $i$ edges in the ordering. Observe that $Z_0=\Ex{f(G_m)}$ and $Z_N=f(G_m)$.  The result will follow immediately from the Hoeffding-Azuma inequality provided we prove that
\eq{incbound}
\big|Z_{i}\, -\, Z_{i-1}\big|\, \le\, 2\psi(e_i)\qquad \text{almost surely}.
\eqe
Let $G_{-}:=G_m\cap \{e_1,\dots ,e_{i-1}\}$ and let us set $m_{-}\, =\, |G_{-}|$ and $m'=m-m_{-}-1$.  We may generate $G_m$ as follows.  Let $J$ be a uniformly random subset of $\{i+1,\dots ,N\}$ of cardinality $m'$, and let $k$ be chosen uniformly in $\{i+1,\dots ,N\}\setminus J$.  We promise that $G_m$ will be given by either
\eq{case1}
G_{-}\, \cup \, \{e_i\}\, \cup \{e_j:j\in J\}\, =: \, G_{-,i,J}
\eqe
if $e_i \in G_m$ or by
\eq{case2}
G_{-}\, \cup \, \{e_k\}\, \cup \{e_j:j\in J\}\, =:\, G_{-,k,J}\,
\eqe
if $e_i \notin G_m$.  With this in mind, we have
\begin{align*}
& \Ex{f(G_m) \, \big|\, G_m\cap\{e_1,\dots ,e_i\}}\\ 
& = \, 1_{e_{i}\in G_m}\mathbb{E}_J[f(G_{-,i,J})|G_{-}]\, +\, 1_{e_{i}\not\in G_m}\mathbb{E}_{J,k}[f(G_{-,k,J})|G_{-}]\\
& =\, \mathbb{E}_{J,k}[f(G_{-,k,J})|G_{-}]\, +\, 1_{e_{i}\in G_m} \mathbb{E}_{J,k}[f(G_{-,i,J})-f(G_{-,k,J})|G_{-}]
\end{align*}
and 
\begin{align*}
&  \Ex{f(G_m) \, \big|\, G_m\cap\{e_1,\dots ,e_{i-1}\}} \\
 & =\, \pr{e_{i}\in G_m|G_{-}}\mathbb{E}_J[f(G_{-,i,J})|G_{-}]  +\, \pr{e_{i}\not\in G_m|G_{-}}\mathbb{E}_{J,k}[f(G_{-,k,J})|G_{-}]\\
& =\,  \mathbb{E}_{J,k}[f(G_{-,k,J})|G_{-}]\, +\, \pr{e_{i}\in G_m|G_{-}} \mathbb{E}_{J,k}[f(G_{-,i,J})-f(G_{-,k,J})|G_{-}]\, .
\end{align*}
It follows that
\[
Z_{i}\, -\, Z_{i-1}\, =\, \mathbb{E}_{J,k}[f(G_{-,i,J})-f(G_{-,k,J})|G_{-}] (1_{e_{i}\in G_m}-\pr{e_{i}\in G_m|G_{-}})\, .
\]
Since $|1_{e_{i}\in G_m}-\pr{e_{i}\in G_m|G_{-}}|\le 1$, we obtain
\[
|Z_{i}\, -\, Z_{i-1}|\, \le\, \big| \mathbb{E}_{J,k}[f(G_{-,i,J})-f(G_{-,k,J})|G_{-}] \big|\, \le\, \mathbb{E}_{J,k}\big[|f(G_{-,i,J})-f(G_{-,k,J})| \big|G_{-}\big]\, .
\]
and since $|f(G_{-,i,J}) - f(G_{-,k,J})| \le \phi(e_i) + \phi(e_k) \le 2 \phi(e_i)$ we obtain~\eqr{incbound}.
\end{proof}

We now state Freedman's inequality, and the related converse inequality. 

Probabilistic intuition would suggest that deviation probabilities ought to be more closely connected to the second moment of the increments $X_i$ than to $\|X_i\|_{\infty}$.  Freedman's inequality~\cite{F} essentially allows us to replace $\|X_i\|_{\infty}^2$ by $\Ex{X_i^2|\F_{i-1}}$, the conditional second moment.

\begin{lem}[Freedman's inequality]\label{lem:F} Let $(S_m)_{m=0}^{M}$ be a martingale with increments $(X_i)_{i=1}^{M}$ with respect to a filtration $(\F_m)_{m=0}^{M}$, let $R\in \mathbb{R}$ be such that $\max_i |X_i|\le R$ almost surely, and let 
\[
V(m):=\sum_{i=1}^{m}\,\, \Ex{\,  |X_i|^2\, \big|\, \F_{i-1}}\, .
\]  
Then, for every $\alpha,\beta >0$, we have
\[
\mathbb{P}\big(S_m-S_0\, \ge\,  \alpha\quad \text{and}\quad V(m)\le \beta \quad \text{for some } m\big)\, \le\, \exp\left(\frac{-\alpha^2}{2(\beta+R\alpha)}\right)\, .
\]
\end{lem}

In addition, Freedman~\cite{F} proved that this exponent is often close to best possible.  Before stating this converse, let us restate the above inequality.  Define the stopping time $m_\alpha$ to be the least $m$ such that $S_m> S_0+\alpha$, and define
\[
T_{\alpha}\, :=\, V(m_{\alpha})\, .
\]
The above inequality states that
\[
\pr{T_{\alpha}\le \beta}\, \le\, \exp\left(\frac{-\alpha^2}{2(\beta+R\alpha)}\right)\, .
\]
Freedman's converse inequality~\cite{F} is as follows.

\begin{lem}[Converse Freedman inequality] \label{lem:CF}
Let $(S_m)_{m=0}^{M}$ be a martingale with increments $(X_i)_{i=1}^{M}$ with respect to a filtration $(\F_m)_{m=1}^{M}$, let $R$ be such that $\max_i|X_i|\le R$ almost surely, and let $T_{\alpha}$ be as defined above.
Then, for every $\alpha,\beta >0$, we have
\[
\pr{T_{\alpha}\le \beta}\, \ge\, \frac{1}{2} \exp\left(\frac{-\alpha^2(1+4\delta)}{2\beta}\right),
\]
where $\delta>0$ is minimal such that $\beta/\alpha\ge 9R\delta^{-2}$ and $\alpha^2/\beta \ge 16\delta^{-2}\log(64\delta^{-2})$.
\end{lem}

From the point of view of our present applications, the essential content of these inequalities is that
\[
\pr{T_{\alpha}\le \beta}\, =\, \exp\left(\frac{-\alpha^2(1+o(1))}{2\beta}\right)\, 
\]
when $\alpha R \ll \beta \ll \alpha^2$.  See Section 3.4 of \cite{Colin} for other martingale inequalities in a similar spirit.

\subsection{A bound for the hypergeometric distribution}

The hypergeometric distribution represents the number of successes in a series of draws without replacement.  Given $N,K,m$, a random varianble $S_m$ has hypergeometric distribution with parameters $N,K,m$ if $\pr{S_m=k} = \binom{K}{k}\binom{N-K}{m-k}/\binom{N}{m}$.  If $\mu = \mathbb{E}[S_m]= Km/N$ we have the following bounds on the upper tail:
\eq{hyperupper}
\pr{S_m\, \ge \, \mu+a}\, \le\,\exp\left(\frac{-a^2}{2\mu+2a/3}\right)\, \le\,\exp\left(\frac{-a^2}{2\mu+a}\right)\,
\eqe
and the lower tail:
\eq{hyperlower}
\pr{S_{m}\, \le \, \mu -a}\, \le\,\exp\left(\frac{-a^2}{2\mu}\right) \, ,
\eqe
which were proved in~\cite{Hoeff}.  They also appear in Theorem 2.10 of~\cite{JLR}.

\section{Degrees and codegrees in $G(n,m)$}\label{sec:degrees}

There are many results on degree sequences of random graphs, for more information see the articles of Bollob\'as~\cite{Bdeg}, McKay and Wormald~\cite{MW} and Liebenau and Wormald~\cite{LW} and the refernces therein.  

We are not aware of a direct reference for the degree and codegree bounds that we require.  In this section we prove bounds on the probability of certain events related to degrees and codegrees in the model $G(n,m)$.  All the proofs are straightforward applications of Corollary~\ref{cor:HA}, a form of the Hoeffding-Azuma inequality.

The three degree deviation results we prove concern the largest degree deviation, the sum of fourth powers of degree deviations and the sum of squares of degree deviations.  After stating these results in Section~\ref{ssec:degrees}, we state the analogous codegree results in Section~\ref{ssec:codegrees}.  We make no effort to optimise the constants in any of these results.

Although it differs from the standard notation, we write $d_u(G)$ for the degree of a vertex $u$ in a graph $G$.  In the case of $G_m\sim G(n,m)$ the expected degree of $u$ is $2m/n$, and so
\eq{Dudef}
D_u(G_m)\, :=\, d_u(G_m)\, -\, \frac{2m}{n}
\eqe
is the deviation of the degree of $u$ from its mean.  We shall also consider codegrees, writing $d_{u,w}(G)$ for the number of common neighbours of vertices $u$ and $w$ in a graph $G$, and 
\eq{Duwdef}
D_{u,w}(G_m)\, :=\, d_{u,w}(G_m)\, -\, \frac{(n-2)(m)_2}{(N)_2}
\eqe
for the deviation of $d_{u,w}(G_m)$ from its mean.

\subsection{Degrees}\label{ssec:degrees}

We prove bounds related to the maximum degree deviation (Lemma~\ref{lem:maxdegdev}), the sum of fourth powers of degree deviations (Lemma~\ref{lem:deg4}), and the sum of squares of degree deviations (Lemma~\ref{lem:deg2}).

We first state the result about the maximum degree deviation.  Let 
\[
D_{\ourmax}(G_m)\, :=\, \max_{u}D_u(G_m)
\]
and
\[
D_{\ourmin}(G_m)\, :=\, \min_{u}D_{u}(G_m)\, .
\]

\begin{lem}\label{lem:maxdegdev} 
For all $b\ge \log{n}$, and all $m\le N$, we have
\[
\pr{D_{\ourmax}(G_m)\, >\, 4b^{1/2}t^{1/2}n^{1/2}\, +\,  4b}\, \le \, \exp(-b)\, .
\]
Furthermore, the same bound holds for the event $D_{\ourmin}(G_m)<-4b^{1/2}t^{1/2}n^{1/2}-4b$.
\end{lem}

\begin{proof} Fix a vertex $u\in V(G_m)$ and let $a=4b^{1/2}t^{1/2}n^{1/2}\, +\,  4b$.  The degree $d_u(G_{m})$ has hypergeometric distribution with parameters $N,n-1,m$.  By the bound~\eqr{hyperupper} on the tail of the hypergeometric distribution we have
\[
\pr{D_{u}(G_m)\, >\, a}\, \le\, \exp\left(\frac{-a^2}{2tn+a}\right) \, \le \, \exp(-2b)\, .
\]
A union bound over the $n\le \exp(b)$ vertices completes the proof of the main statement.  The lower tail bound follows by a similar argument.
\end{proof}

Our result on the sum of fourth powers of degree deviations is as follows.

\begin{lem}\label{lem:deg4} There is a constant $C>0$ such that, for all $b\ge n^{1/2}$, and all $m\le N$, we have
\[
\pr{\sum_{u\in V(G_m)}D_{u}(G_m)^4\, >\, Cb n^{2}\min\{b,n\} }\, \le \, \exp(-b)\, .
\]
\end{lem}

\begin{proof}
Fix $b\ge n^{1/2}$.  We define a family of events related to degree deviations.  For each $j\ge 1$, let
\[
a_j\, :=\, 2^{3-j/2} b^{1/2}n^{1/2}\, +\, 2^{3-j/8}n^{5/8}\, .
\]
We may immediately note that\vspace{-0.1cm}
\eq{aj2}
a_j^2\, \ge \, 2^{6-j}bn\, +\, 2^{6-j/4} n^{5/4}
\eqe
and that
\eq{aj4}
a_j^4\, \le\, 2^{15-2j}b^2n^2\, +\, 2^{15-j/2}n^{5/2}\, .\vspace{0.1cm}
\eqe
For each $j\ge 1$ and for each set $U\subseteq V(G_m)$ of cardinality $2^j$ we define $E^{+}_{j,U}$ to be the event that
\[
D_u(G_m)\, >\, a_j \qquad \text{for all } u\in U\, ,
\]
and $E^{-}_{j,U}$ to be the event\vspace{-0.1cm}
\[
D_u(G_m)\, <\, -a_j \qquad \text{for all } u\in U\, .
\]

\noindent\textbf{Claim:} $\pr{E^{+}_{j,U}}\, \le\, \exp(-2b-2^{1+3j/4}n^{1/4})$.\vspace{0.1cm}

\noindent \textbf{Proof of Claim:} Define 
\[
f(G_m)\, =\, \sum_{u\in U} D_{u}(G_m)\, .
\]
We observe that $\Ex{f(G_m)}=0$ and that $f$ is $\psi$-Lipschitz for the function $\psi(e)=|e\cap U|$.  We have $\sum_{e \in E(K_n)} \psi(e)^2\, \le\, 4|U|n\, =\, 2^{j+2}n$.  We are now ready to apply Corollary~\ref{cor:HA}.  Noting that $f(G_m)>2^ja_j$ on the event $E^{+}_{j,U}$, we have\vspace{-0.1cm}
\[
\pr{E^{+}_{j,U}}\, \le\, \pr{f(G_m)\, >\, 2^ja_j}\, \le\, \exp\left(\frac{-2^{2j}a_j^2}{2^{j+5}n}\right)\, .
\]
The claim now follows immediately from~\eqr{aj2}.

Naturally, the same bound holds for $\pr{E^{-}_{j,U}}$.

Now, for each $j\ge 1$, a union bound allows us to bound the probability that any of the events $E^{+}_{j,U}$ or $E^{-}_{j,U}$ occurs for any set $U$ of $2^j$ vertices.   Indeed this probability is at most
\begin{align*}
2\binom{n}{2^j}\exp(-2b-2^{1+3j/4}n^{1/4})\, & \le\, \exp(-2b)\exp(1+2^j+2^j\log(n2^{-j})-2^{1+3j/4}n^{1/4})\\ 
& \le\, \exp(-2b)\, ,
\end{align*}
where we use the fact that $\binom{n}{2^j} \le \left(\exp(1) n 2^{-j}\right)^{2^j}$ and the final inequality is obtained using the bound $\log{x}\le x^{1/4}$ applied with $x=n2^{-j}$.
Taking a union bound over $1\le j\le \log_2{n}$ there is probability at most
\[
\log_2{n} \exp(-2b)\, \le \, \exp(-b)
\]
that any of the events $E^{+}_{j,U}$ or $E^{-}_{j,U}$ occurs.  On the complementary event, there is a partition of the vertices into groups $V_{1},V_2,\dots$ such that $|V_j|\le 2^{j+1}$ and $|D_{u}(G_m)|\le a_j$ for all $u\in V_j$.  Thus, with probability at least $1-\exp(-b)$, we have
\begin{align*}
\sum_{u\in V(G_m)}D_{u}(G_m)^4\, & \le\, \sum_{j= 1}^{\log_2{n}} 2^{j+1}a_j^4\\
 & \le\, \sum_{j= 1}^{\log_2{n}} \Big(2^{16-j}b^2n^2\, +\, 2^{16+j/2}n^{5/2}\Big)\vspace{0.1cm}\\ 
 & \le\, C b^2 n^2\, ,
 \end{align*}
where we have used~\eqr{aj4} to prove the second inequality, and for the third we have taken $C\ge 2^{19}$ and used that $b\ge n^{1/2}$.
 
This proves the lemma in the case that $b\le n$.  If $b>n$, then let $J$ be the smallest integer such that $2^{J}\ge b/n$.  It follows that $2^{-J}\le n/b$. We now argue as above, except using the trivial bound $|D_u(G_m)|\le n$ for the vertices $u\in V_j$ for $j<J$.  Now, with probability at least $1-\exp(-b)$, we have
\begin{align*}
\sum_{u\in V(G_m)}D_{u}(G_m)^4\, & \le\, 2^{J+1}n^4\, +\, \sum_{j\ge J} 2^{j+1}a_j^4\\
 & \le\, 4bn^3\, +\, \sum_{j\ge J} \Big(2^{16-j}b^2n^2\, +\, 2^{16+ j/2}n^{5/2}\Big)\\
 & \le\, 4b n^3\, +\, 2^{17-J}b^2n^2\, +\, 2^{16} n^{3}\phantom{\sum}\\
 &\le \, 4bn^{3} \, +\, 2^{17} b n^{3}\, +\, 2^{16} n^3\phantom{\sum}\\
 & \le\, Cbn^3\, ,\phantom{\sum}
 \end{align*}
where $C$ has been taken to be at least $2^{18}$.  This proves the inequality in the case $b\ge n$, completing the proof.
\end{proof}

\begin{lem}\label{lem:deg2}
There is a constant $C$ such that for all $b\ge 30$, and all $m\le N$, we have
\[
\pr{\sum_{u\in V(G_m)}D_u(G_m)^2\, >\, C b n^2}\, \le\, \exp(-bn) \, .
\]
\end{lem}
 
\begin{proof}
Fix $b\ge 30$.  Let $\ell=\lfloor \log_2{n}\rfloor$.  We shall consider a function $f_{\sigma}$ for each sequence $\sigma\in \{0,\pm1,\pm 2,\pm 4,\dots, \pm 2^{\ell}\}^{V(G_m)}$ defined as follows
\[
f_{\sigma}(G_m)\, =\, \sum_{u\in V(G_m)}\sigma_u D_u(G_m)\, .
\]
Let us set
\[
\|\sigma\|^2\, :=\, \sum_{u\in V(G_m)}\sigma_u^2\, .
\]
To see the connection between these functions and the result of the lemma, consider the choice of $\sigma^{*}$ defined by setting $\sigma^{*}_u=0$ if $|D_u(G_m)|\le n^{1/2}$, and otherwise defined so that $\sigma^{*}_u$ has the same sign as $D_{u}(G_m)$ and $|\sigma^{*}_u|$ is the largest power of two such that $|\sigma^{*}_u|n^{1/2}$ is at most $|D_{u}(G_m)|$.  With this choice of $\sigma^{*}$ we have
\eq{fisbig}
f_{\sigma^{*}}(G_m)\, \ge\, \|\sigma^{*}\|^2 n^{1/2}\, .
\eqe
Furthermore, if $\sum_{u}D_u(G_m)^2\, >\, C b n^2$ with constant $C\ge 129$ then
\eq{sigmaisbig}
\|\sigma^{*}\|^2\, =\, \sum_{u}(\sigma^{*}_u)^2\, \ge\, \sum_{u} \frac{D_u(G_m)^2\, -\, n}{4n} \, >\, 32bn\, .
\eqe
By~\eqr{fisbig},~\eqr{sigmaisbig}, and a union bound, proving the lemma reduces to the problem of proving the following inequality:
\eq{allthat}
\sum_{\sigma:\|\sigma\|^2>32bn}\, \pr{f_{\sigma}(G_m)>\|\sigma\|^2n^{1/2}}\, \le\, \exp(-bn)\, .
\eqe
We first bound this probability for each fixed $\sigma$.

\noindent\textbf{Claim:} $\pr{f_{\sigma}(G_m)>\|\sigma\|^2n^{1/2}}\, \le\, \exp(-\|\sigma\|^2/16)$.\vspace{0.1cm}

\noindent \textbf{Proof of Claim:} The function $f_\sigma$ is $\psi$-Lipschitz for the function $\psi(uw)=|\sigma_u|+|\sigma_w|$, for which $\sum_{e \in E(K_n)} \psi(e)^2\le 2n\sum_{u \in V(G_m)} \sigma_u^2\, =\, 2n\|\sigma\|^2$.  Since $\Ex{f_{\sigma}(G_m)}=0$, it follows from Corollary~\ref{cor:HA} that
\[
\pr{f_{\sigma}(G_m)\, >\, \|\sigma\|^2n^{1/2}}\, \le\, \exp\left(\frac{-\|\sigma\|^4n}{16n\|\sigma\|^2}\right)\, \le\, \exp(-\|\sigma\|^2/16)\, ,
\]
as required, completing the proof of the claim.
 
For $\sigma$ with $\|\sigma\|^2>32bn$ it follows that 
\[
\pr{f_{\sigma}(G_m)\, >\, \|\sigma\|^2 n^{1/2}}\, \le\, \exp\left(\frac{-\|\sigma\|^2}{16}\right)\, \le \, \exp(-bn)\,\exp\left(\frac{-\|\sigma\|^2}{32}\right)\, .
\]
Substituting this bound into~\eqr{allthat} we need only prove that
\eq{allthat2}
\sum_{\sigma:\|\sigma\|^2>32bn}\, \exp\left(\frac{-\|\sigma\|^2}{32}\right)\, \le\, 1\, .
\eqe
We prove this bound by splitting into ``types''.  Given a sequence $x=(x_{-\ell-1},\dots ,x_{\ell+1})$ we say $\sigma$ has \emph{type} $x$ if precisely $x_0$ vertices $u$ have $\sigma_u=0$, precisely $x_j$ have $\sigma_u=2^{j-1}$ and precisely $x_{-j}$ have $\sigma_u=-2^{j-1}$ for each $j\in\{1,\dots ,\ell+1\}$.  Setting
\[
S_{x}\, :=\, \left\{\sigma\, :\, \sigma \, \text{has type }x \,\text{and}\, \|\sigma\|^2>32bn\right\}\, 
\]
and observing that there are at most $n^{3\ell}\,\le\, \exp(n)$ choices of $\sigma$ which have type $x$, it suffices to prove that
\eq{allthat3}
\sum_{\sigma\in S_{x}}\, \exp\left(\frac{-\|\sigma\|^2}{32}\right)\, \le\, \exp(-n)\, 
\eqe
for each type $x$. Note that all $\sigma$ of type $x$ have the same $\|\sigma\|^2$, which is given by
\[
\varphi(x)\, :=\, \sum_{j\neq 0} x_j 4^{|j|+1}\, .
\]
It follows that $S_x$ is empty if $\varphi(x)\le 32bn$.  Fix a type $x$ with $\varphi(x)\ge 32bn$, we prove~\eqr{allthat3} for this type $x$.  We must prove that
\[
|S_x|\, \le\, \exp\left(\frac{\varphi(x)}{32}\, -\, n\right)\, .
\]
We have
\[
|S_x|\, =\, \binom{n}{x_{-\ell-1},\dots ,x_{\ell+1}}\, \le\, \prod_{j\neq 0}\binom{n}{x_j}
\]
and, by the well known inequality $\binom{n}{k}\le (en/k)^k$, we obtain
\[
|S_x|\, \le\, \exp\left( \sum_{j\neq 0}x_j\log(en/x_j)\right)\, \le\, \exp\left(e^{1/2}n^{1/2}\sum_{j\neq 0} x_{j}^{1/2}\right)\, ,
\]
where we have used the inequality $\log{y}\le y^{1/2}$ for $y>0$.

For each $j\neq 0$, we have
\[
|x_j|\, \le\, \min\{n, 4^{1-j}\varphi(x)\}\, .
\]
Using $|x_j|\le n$ for $|j|\le 4$ and $|x_j|\le 4^{1-j}\varphi(x)$ for $|j|\ge 5$, we have
\[
|S_x|\, \le\, \exp\left(16n\, +\, 2n^{1/2}\sum_{|j|\ge 5}2^{1-j}\varphi(x)^{1/2}\right) \, .
\]
Since $\varphi(x)\ge 32bn$, we obtain
\[
|S_x|\, \le\, \exp\left(\frac{\varphi(x)}{2b}\, +\, \frac{\varphi(x)}{16b^{1/2}}\right) \, \le \, \exp\left(\frac{\varphi(x)}{32}\, -\, n\right)\, ,
\]
as required, completing the proof.
\end{proof}

 \subsection{Codegrees}\label{ssec:codegrees}

We now state and prove the analogous results for codegrees.

Recall that $d_{u,w}(G_m)$ denotes the number of common neighbours of vertices $u$ and $w$ in $G_m$, and 
\[
D_{u,w}(G_m)\, :=\, d_{u,w}(G_m)\, -\, \frac{(n-2)(m)_2}{(N)_2}
\] 
is the deviation of $d_{u,w}(G_m)$ from its mean.  Let 
\[
D'_{\ourmax}(G_m)\, :=\, \max_{u,w}D_{u,w}(G_m)
\]
and
\[
D'_{\ourmin}(G_m)\, :=\, \min_{u,w}D_{u,w}(G_m)\, .
\]

\begin{lem}\label{lem:maxcodegdev} 
For all $b\ge 2\log{n}$, and all $m\le N$, we have
\[
\pr{D'_{\ourmax}(G_m)\, >\, 4b^{1/2}t^{1/2}n^{1/2}\, +\,  8b}\, \le \, \exp(-b)\, .
\]
Furthermore, the same bound holds for the event $D'_{\ourmin}(G_m)<-4b^{1/2}t^{1/2}n^{1/2} -  8b$.
\end{lem}

We omit the proof, which is essentially identical to the proof of Lemma~\ref{lem:maxdegdev}.

The codegree version of Lemma~\ref{lem:deg4}, on fourth powers of degree deviations, is as follows.

\begin{lem}\label{lem:codeg4} There is a constant $C>0$ such that, for all $b\ge n^{1/2}$, and all $m\le N$, we have
\[
\pr{\sum_{u,w\in V(G_m)}D_{u,w}(G_m)^4\, >\, Cb n^{3}\min\{b,n\}}\, \le \, \exp(-b)\, .
\]
\end{lem}

The proof is very similar to that of Lemma~\ref{lem:deg4}.  One difference is that in place of the events $E_{j,U}^+$ and $E_{j,U}^-$, we consider events of this type inside matchings. This may appear ad hoc, but if we do not make such a restriction the argument runs into problems when we arrive at the union bound.

\begin{proof} 
Fix $b\ge n^{1/2}$.  Let $M_1,\dots ,M_{n}$ be a sequence of matchings which partition $E(K_n)$.  We define a family of events related to codegree deviations.  For each $j\ge 1$, let
\[
a'_j\, :=\, 2^{4-j/2} b^{1/2}n^{1/2}\, +\, 2^{4-j/8}n^{5/8}\, .
\]
We may immediately note that
\eq{apj2}
(a'_j)^2\, \ge \, 2^{8-j}bn\, +\, 2^{8-j/4} n^{5/4}
\eqe
and that
\eq{apj4}
(a'_j)^4\, \le\, 2^{19-2j}b^2n^2\, +\, 2^{19-j/2}n^{5/2}\, .
\eqe
For each $1 \le j \le \log_2 n$ and for each set $U\subseteq M_1$ of cardinality $2^j$ we define $F^{+}_{j,U}$ to be the event that
\[
D_{u,w}(G_m)\, >\, a_j \qquad \text{for all } uw\in U\, ,
\]
and $F^{-}_{j,U}$ to be the event
\[
D_{u,w}(G_m)\, <\, -a_j \qquad \text{for all } uw\in U\, .
\]

\noindent\textbf{Claim:} $\pr{F^{+}_{j,U}}\, \le\, \exp(-4b-2^{2+3j/4}n^{1/4})$.\vspace{0.1cm}

\noindent \textbf{Proof of Claim:} Define 
\[
f(G_m)\, =\, \sum_{uw\in U} D_{u,w}(G_m)\, .
\]
We observe that $\Ex{f(G_m)}=0$ and that $f$ is $\psi$-Lipschitz for the function $\psi(e)=|e\cap \bigcup U|$, where $\bigcup U$ denotes the set of vertices that occur in an edge of $U$.  We have $\sum \psi(e)^2\, \le\, 2^{j+3}n$.  We are now ready to apply Corollary~\ref{cor:HA}.  Noting that $f(G_m)>2^j a'_j$ on the event $F^{+}_{j,U}$ we have\vspace{-0.1cm}
\[
\pr{F^{+}_{j,U}}\, \le\, \pr{f(G_m)\, >\, 2^ja'_j}\, \le\, \exp\left(\frac{-2^{2j}(a'_j)^2}{2^{j+6}n}\right)\, .
\]
The claim now follows immediately from~\eqr{apj2}.

Naturally, the same bound holds for $\pr{F^{-}_{j,U}}$.

Now, for each $j\ge 1$, a union bound allows us to bound the probability that any of the events $F^{+}_{j,U}$ or $F^{-}_{j,U}$ occurs for any set $U$ of $2^j$ pairs of $M_1$.   Indeed this probability is at most
\begin{align*}
2\binom{n/2}{2^j}\exp(-4b-2^{2+3j/4}n^{1/4})\, & \le\, \exp(-4b) \exp(1+2^j+2^j\log(n2^{-j})-2^{2+3j/4}n^{1/4})\\ 
& \le\, \exp(-4b)\, ,
\end{align*}
where the final inequality is obtained using the bound $\log{x}\le x^{1/4}$ applied with $x=n2^{-j}$.

Taking a union bound over $1\le j\le \log_{2} n$ there is probability at most
\[
\log(n) \exp(-4b)\, \le \, \exp(-2b)
\]
that any of the events $F^{+}_{j,U}$ or $F^{-}_{j,U}$ occurs.  The above argument also holds inside the remaining matchings $M_2,\dots ,M_n$.  Since $n\exp(-2b)\le \exp(-b)$, we have with probability at least $1-\exp(-b)$ that in each matching and for each $j\ge 1$, at most $2^j$ edges $uw$ have $D_{u,w}(G_m)\, >\, a_j$ and at most $2^j$ have $D_{u,w}(G_m)\, <\, -a_j$.

In this case, there is a partition of the edges of $K_n$ into groups $E_{1},\dots$ such that $|E_j|\le 2^{j+1}n$ and $|D_{u,w}(G_m)|\le a_j$ for all $uw\in E_j$.  Thus, with probability at least $1-\exp(-b)$, we have
\begin{align*}
\sum_{uw}D_{u,w}^4(G_m)\, & \le\, \sum_{j =1}^{\log_2 n} 2^{j+1}n(a'_j)^4\\
 & \le\, \sum_{j =1}^{\log_2 n} (2^{20-j}b^2n^3\, +\, 2^{20+j/2}n^{7/2} )\\
 & \le\, C b^2 n^3\, , \phantom{\sum}
 \end{align*}
where we have used~\eqr{apj4} to prove the second inequality, and taken $C\ge 2^{22}$.
 
This proves the lemma in the case that $b\le n$.  If $b>n$, then let $J$ be the smallest integer such that $2^{J}\ge b/n$.  It follows that $2^{-J}\le n/b$. We now argue as above, except using the trivial bound $|D_{u,w}(G_m)|\le n$ for the pairs $uw\in V_j$ for $j<J$.  Now, with probability at least $1-\exp(-b)$, we have
\begin{align*}
\sum_{uw}D_{u,w}^4(G_m)\, & \le\, 2^{J+1}n^5\, +\, \sum_{j= J}^{\log_2 n} 2^{j+1}na_j^4\\
 & \le\, 4bn^4\, +\, \sum_{j= J}^{\log_2 n}(2^{20-j}b^2n^3\, +\, 2^{20 + j/2}n^{7/2} ) \phantom{\sum}\\
 & \le\, 4b n^4\, +\, 2^{21-J}b^2n^3\, +\, 2^{20} n^{4}\phantom{\sum}\\
 &\le \, 4bn^{4} \, +\, 2^{21} b n^{4}\, +\, 2^{20} n^4\phantom{\sum}\\
 & \le\, Cbn^4\, ,\phantom{\sum}
 \end{align*}
where $C$ has been taken to be at least $2^{22}$.  This proves the inequality in the case $b\ge n$, completing the proof.
\end{proof}

Finally, the generalisation of Lemma~\ref{lem:deg2} to codegrees is as follows.

\begin{lem}\label{lem:codeg2}
There is a constant $C$ such that for all $b\ge 30$, and all $m\le N$, we have
\[
\pr{\sum_{u,w}D_{u,w}(G_i)^2\, >\, C b n^3}\, \le\, \exp(-bn) \, .
\]
\end{lem}

\begin{proof} We describe how the proof may be obtained from ideas present in the above proofs.  As in the proof of Lemma~\ref{lem:codeg4}, let $M_1,\dots ,M_n$ be a family of matchings that partition $K_n$.

\noindent \textbf{Claim:} There is a constant $C$ such that 
\[
\pr{\sum_{uw\in M_1}D_{u,w}(G_m)^2 > C b n^2}\, \le\, \exp(-2bn) .
\]

\noindent \textbf{Proof of Claim:} In the same way that the proof of Lemma~\ref{lem:deg4} was adjusted to bound deviation probabilities for $\sum_{uw\in M_1}D_{uw}(G_m)^4$ in the proof of Lemma~\ref{lem:codeg4}, so Lemma~\ref{lem:deg2} may easily be adjusted to prove the claim.

Applying a union bound over the matchings $M_1,\dots ,M_n$ one obtains that with probability at least $1-\exp(-bn)$ we have
\[
\sum_{uw\in M_k}D_{u,w}(G_m)^2\, \le\, C b n^2
\]
for all $k=1,\dots n$.  In this case
\[
\sum_{u,w}D_{u,w}(G_m)^2\, =\, \sum_{k=1}^{n}\sum_{uw\in M_k}D_{u,w}(G_m)^2\, \le\, C b n^3\, ,
\]
as required.
\end{proof}

\section{Approximating the deviation $D_{H}(G_{n,t})$ in terms of \(D_{{\footnotesize \bigwedge}}(G_{n,t})\) and $D_{\triangle}(G_{n,t})$ -- Theorem~\ref{thm:relate}}\label{sec:relate} 

The main aim of this section is to prove Theorem~\ref{thm:relate}, which states that $D_{H}(G_{n,t})$ is well approximated by a certain linear combination $\Lambda_H(G_{n,t})$ of $D_{\owedge}(G_{n,t})$ and $D_{\triangle}(G_{n,t})$.   This result will be extremely useful since, for the range of deviations for which it applies, it essentially reduces the study of all subgraph count deviations $D_{H}(G_{n,t})$ to the cases of two specific graphs, the path of length two and the triangle.



In order to prove Theorem~\ref{thm:relate}, we first prove Theorem~\ref{thm:approx}, which shows that $D_{H}(G_{n,t})$ is very well approximated by 
\begin{align*}
& \Lambda^{*}_H(G_{n,t})\,  \\
& = \, n^{v-3}\sum_{i=1}^{m}\left(t^{e-2}\obinom{H}{\owedge}\frac{(1-t)^2}{(1-s)^2}X_{\owedge}(G_i)\, +\, t^{e-3}\obinom{H}{\triangle} \frac{(1-t)^3}{(1-s)^{3}}\, \big(X_{\triangle}(G_i)-3sX_{\owedge}(G_i)\big) \right)\, ,
\end{align*}
a sum of terms each of which is a linear combination of $X_{\owedge}(G_i)$ and $X_{\triangle}(G_i)$, where $m=\lfloor tN\rfloor$ and $s:=i/N$.  We deduce Theorem~\ref{thm:relate} (for $t\in (0,1/2]$) from Theorem~\ref{thm:approx} by showing that $\Lambda_H(G_{n,t})$ is very close to $\Lambda^{*}_H(G_{n,t})$ deterministically.  It is then straightforward to deduce the remaining cases ($t\in (1/2,1)$) using Corollary~\ref{cor:comp}.

Let us now discuss the task of proving Theorem~\ref{thm:approx}.  Naturally, our proof that $D_H(G_{n,t})$ is well approximated by $\Lambda^{*}_H(G_{n,t})$ begins with the precise martingale expression for $D_H(G_m)$ (given by Theorem~\ref{thm:Mart})
\[
D_{H}(G_m) \, =\, \sum_{i=1}^{m}\, \sum_{F\ssq E(H)}\frac{(N-m)_{e(F)}(m-i)_{e-e(F)}}{(N-i)_e}\, X_F(G_i)\, .
\]
In order to show that the precise expression is well approximated by $\Lambda^{*}_H(G_{n,t})$, we show that each $X_F(G_i)$ can be well approximated by
\eq{Xsdef}
X^*_F(G_i)\, := \, n^{v-3}s^{e(F)-2}\left(\obinom{F}{\owedge}-3\obinom{F}{\triangle}\right) X_{\owedge}(G_i)\, +\, n^{v-3}s^{e(F)-3}\obinom{F}{\triangle}X_{\triangle}(G_i)\, .
\eqe
This statement is made rigorous in Proposition~\ref{prop:Ysmall}.

\begin{definition} For each graph $F$ we define
\eq{Ydef}
Y_{F}(G_i)\, :=\, X_{F}(G_i)\, -\, X^{*}_F(G_i)\, .
\eqe
\end{definition}

We prove that $Y_F(G_i)$ is small, in particular in the $L^2$ sense.  Since it is the graph $G_{i-1}$ that determines the distribution of $Y_F(G_i)|G_{i-1}$, the result will state that it is very unlikely that $G_{i-1}$ is such that $\Ex{Y_F(G_i)^2|G_{i-1}}$ is large.

\begin{prop}\label{prop:Ysmall}
Let $F$ be a graph with $v(F)$ vertices and $e(F)$ edges, and let $t\in (0,1/2]$.  There is a constant $C=C(F)$ such that for all $1\le i\le tN$ and $b\ge 3\log{n}$,  we have
\eq{ayevent}
\pr{\Ex{Y_F(G_i)^2\, \big| \,G_{i-1}}\, >\, Cb n^{2v(F)-6}}\, \le\, \exp(-b)\, .
\eqe
\end{prop}

\begin{remark} The result may be proved for all $t\in (0,1)$, however for our purposes working for $t\in (0,1/2]$ is sufficient.
\end{remark}

In Section~\ref{sec:Ysmall} we prove Proposition~\ref{prop:Ysmall}.  We will then be ready to prove Theorem~\ref{thm:approx} in Section~\ref{sec:approx} and Theorem~\ref{thm:relate} in Section~\ref{ssec:relate}.

\subsection{Proof of Proposition~\ref{prop:Ysmall}}\label{sec:Ysmall}

Fix a graph $F$ with $v(F)$ vertices and $e(F)$ edges. In this subsection we write $v$ for $v(F)$ and $e$ for $e(F)$.  The proof of Proposition~\ref{prop:Ysmall} depends on Lemma~\ref{lem:Yis} and Lemma~\ref{lem:Clem}.  We shall now motivate and state these two lemmas.

Let $e_1,\dots ,e_N$ be the order in which edges are added in the realisation of the Erd\H os-R\'enyi random graph process, so that $G_m=\{e_1,\dots ,e_m\}$.  In particular, in this notation $e_i$ is the edge we add to go from $G_{i-1}$ to $G_i$. Define $A^{*}_F(G_i)$, a linear combination involving the degree and codegree deviation of $e_i$, by
\begin{align*}
A^{*}_F(G_i)\, & :=\, 2e s^{e-1}n^{v-2}\, +\, s^{e-2}n^{v-3}\left(2\obinom{F}{\owedge}-6\obinom{F}{\triangle}\right)(D_u(G_i)+D_w(G_i))\, \\
& \qquad  +\, 6s^{e-3}n^{v-3}\obinom{F}{\triangle}D_{u,w}(G_{i})\, .
\end{align*}
We will prove that $A_F(G_i)$ (which was introduced in~\eqr{Adef}) is usually well approximated by $A^*_F(G_i)$ (see Lemma~\ref{lem:Clem}).  On the other hand, we prove that $Y_F(G_i)$ may be expressed in terms of the difference $A_F(G_i)-A^{*}_F(G_i)$.  

\begin{lem}\label{lem:Yis} 
\[
Y_F(G_i)\, =\, \big( A_F(G_i)-A^{*}_F(G_i)\big)\, -\, \Ex{A_F(G_i)-A^{*}_F(G_i)\, \big|\, G_{i-1}}\, .
\]
\end{lem}

\begin{proof} This expression for $Y_F(G_i)$ follows almost directly from its definition as $Y_F(G_i):= X_{F}(G_i)\, -\, X^{*}_F(G_i)$.  Indeed, the definition~\eqr{Xdef} of $X_F(G_i)$ is
\[
X_F(G_i)\, =\, A_F(G_i)\, -\, \Ex{A_F(G_i)\, |\, G_{i-1}}
\]
and so we need only prove that
\eq{sstp}
X^{*}_F(G_i)\, =\, A^{*}_F(G_i)\, -\, \Ex{A^{*}_F(G_i)\,|\, G_{i-1}}\, .
\eqe
As $X^*_F(G_i)$ is defined~\eqr{Xsdef} as a linear combination of $X_{\owedge}(G_i)$ and $X_{\triangle}(G_i)$ it is useful to note that
\begin{align*}
X_{\owedge}(G_i)\, & =\, A_{\owedge}(G_i)\, -\,  \frac{8(i-1)}{n}\, + \, \Ex{ \frac{8(i-1)}{n}-A_{\owedge}(G_i)\, \big|\,G_{i-1}}\\
& =\, 2\big(D_u(G_{i-1})+D_w(G_{i-1}) \big) \, -\, \Ex{2\big(D_u(G_{i-1})+D_w(G_{i-1}) \big)\, \big|\, G_{i-1}}\, 
\end{align*}
and
\[
X_{\triangle}(G_i)\, =\, 6D_{uw}(G_{i-1})\, -\, \Ex{6D_{uw}(G_{i-1})\, \big|\, G_{i-1}}\, ,
\]
where we have used that
\begin{align}
A_{\owedge}(G_i)\, & =\, 2\big(d_u(G_{i-1})+d_{w}(G_{i-1})\big) \nonumber \\
& =\,  \frac{8(i-1)}{n}\, +\, 2\big(D_u(G_{i-1})+D_w(G_{i-1})\big)\, . \label{eq:Awedgeis}
\end{align}
and
\eq{Atriis}
A_{\triangle}(G_i)\, =\, \frac{6(n-2)(i-1)_2}{(N)_2}\, +\, 6D_{u,w}(G_{i-1})\, .
\eqe

The required equation~\eqr{sstp} now follow simply by substituting these values in the definition of $X^*_F(G_i)$.
\end{proof}

We now state Lemma~\ref{lem:Clem}.  We shall use the quantity $\Delta(e_i)$ defined to be the sum of squares of the degree and codegree deviations associated with edge $e_i$.  That is,
\eq{Deldef}
\Delta(e_i)\, :=\, D_{u}(G_{i-1})^2+D_{w}(G_{i-1})^2+D_{uw}(G_{i-1})^2\, ,
\eqe
where $e_i=\{u,w\}$.

\begin{lem}\label{lem:Clem} Let $F$ be a graph with $v$ vertices and $e$ edges.  There is a constant $C=C(F)$ such that, for all $1\le i\le N$ and $b\ge 1$, the event that
\eq{aevent}
\big|A_F(G_i)-A^{*}_F(G_i)\big|\, >\, Cb^{1/2}n^{v-3}\, +\, Cn^{v-4}\Delta(e_i)
\eqe
has probability at most $\exp(-b)$.
\end{lem}

\begin{proof} The vertex set of $G_{m}$ is $[n]=\{1,\dots ,n\}$.  By symmetry we may assume that the pair $12$ is added as the $i$th edge, i.e., $e_i=12$.  Thus, the event~\eqr{aevent} may be viewed as an event concerning the first $i-1$ edges $e_1,\dots ,e_{i-1}$.  We may reveal this information as follows: we first reveal the neighbourhoods $N_1(G_{i-1})$ and $N_2(G_{i-1})$ of vertices $1$ and $2$ in $G_{i-1}$, and then we reveal the remaining edges.  We shall prove, for any choice on the first step, of $N_1(G_{i-1})$ and $N_2(G_{i-1})$, that the conditional probability that~\eqr{aevent} occurs
is at most $\exp(-b)$.  The result of the lemma then follows by taking expectations.


Let us now fix $N_1:=N_1(G_{i-1})$ and $N_2:=N_2(G_{i-1})$.  We set $d_1=|N_1|$ and $d_2=|N_2|$.  Let us also abbreviate $D_1(G_{i-1}),D_{2}(G_{i-1})$ and $D_{1,2}(G_{i-1})$ to $D_1,D_2$ and $D_{1,2}$ respectively, for the duration of the proof.

Our aim is to show that in selecting the remaining $i-1-d_1-d_2$ edges, in $V(G_m)\setminus \{1,2\}$, there is probability at most $\exp(-b)$ that~\eqr{aevent} occurs.

The proof will use the triangle inequality, in the sense that we bound $|A_F(G_i)-A^{*}_F(G_i)|$ by introducing a third quantity $A^{**}_F(G_i)$ such that
\eq{todouble} 
\pr{\big|A_F(G_i)-A^{**}_F(G_i)\big|\, >\, Cb^{1/2}n^{v-3}\phantom{\Big|} \big|\, N_1,N_2,e_i=12}\,\le \, \exp(-b)\,
\eqe
and
\eq{doubling}
\big|A^{*}_F(G_i)-A^{**}_F(G_i)\big|\, \le\,  Cb^{1/2}n^{v-3}\, +\, Cn^{v-4}\Delta(e_i)
\eqe
deterministically.  We set
\[
A^{**}_F(G_i)\, :=\, \Ex{A_F(G_i)\, \big|\,N_1,N_2,e_i=12}\, .
\]
It is clear (by considering the triangle inequality) that proving the lemma reduces to verifying~\eqr{todouble} and~\eqr{doubling}.

Let us subdivide $A_F(G_i)$ depending on which edge $f$ of $F$ corresponds to the new edge $e_i$, and its orientation with respect to $e_i=12$, which one may think of as oriented $\vec{12}$.  That is, we write 
\[
A_{F,\vec{f}}(G_i)
\]
for the number of embeddings $\phi(F)$ of $F$ created with the addition of $e_i=12$ in which $\phi(\vec{f})=\vec{12}$.  Clearly
\eq{AFassum}
A_F(G_i)\, =\, \sum_{\vec{f}}A_{F,\vec{f}}(G_i)\, ,
\eqe
where the sum is over orientations $\vec{f}$ of edges $f\in E(F)$.  We shall also define $A^{*}_{F,\vec{f}}(G_i)$ and $A^{**}_{F,\vec{f}}(G_{i-1})$ for each $\vec{f}\in \vec{E}(F)$, as follows.
We shall write $\Gamma_1(\vec{f})$ and $\Gamma_2(\vec{f})$ for the neighbourhood in $F$ of the start and end vertex of $\vec{f}$ respectively, and we set
\[
\beta(\vec{f})\, :=\, \big|\Gamma_1\cap \Gamma_2\big|\, ,\quad \alpha_1(\vec{f})\, :=\, \big|\Gamma_1\setminus \Gamma_2\big|\quad \text{and}\quad \alpha_2(\vec{f})\, :=\, \big|\Gamma_2\setminus \Gamma_1\big|\, .
\]
We may now define
\eq{AsFf}
A^{*}_{F,\vec{f}}(G_i)\, =\, s^{e-1}n^{v-2}\, +\, s^{e-2}n^{v-3}\big(\alpha_1(\vec{f}) D_1+\alpha_2(\vec{f}) D_2\big)\, +\, s^{e-3}n^{v-3}\beta(\vec{f})D_{1,2}\, ,
\eqe
One may easily verify that $\sum_{\vec{f}}\alpha_1(\vec{f})=\sum_{\vec{f}}\alpha_2(\vec{f})=2\binom{F}{\owedge}-6\binom{F}{\triangle}$ and $\sum _{\vec{f}}\beta(\vec{f})=6\binom{F}{\triangle}$, from which it follows that
\eq{AFsassum}
A^{*}_F(G_i)\, =\, \sum_{\vec{f}}A^{*}_{F,\vec{f}}(G_i)\, .
\eqe 
We may also define
\[
A^{**}_{F,\vec{f}}(G_i)\, :=\, \Ex{A_{F,\vec{f}}(G_i)\, \big|\,N_1,N_2,e_i=12}\, .
\]
It follows directly from linearity of expectation that
\eq{AFssassum}
A^{**}_F(G_i)\, =\,  \sum_{\vec{f}}A^{**}_{F,\vec{f}}(G_i)\, .
\eqe
Taken together, equations~\eqr{AFassum},~\eqr{AFsassum} and~\eqr{AFssassum} reduce the problem of proving~\eqr{todouble} and~\eqr{doubling} to the problem of proving, for each $\vec{f}\in \vec{E}(F)$,
\eq{tofdouble}
\pr{\big|A_{F,\vec{f}}(G_i)-A^{**}_{F,\vec{f}}(G_i)\big|\, >\, Cb^{1/2}n^{v-3}\phantom{\Big|} \big|\, N_1,N_2,e_i=12}\,\le\, \exp(-b)\,
\eqe
and
\eq{fdoubling}
\big|A^{*}_{F,\vec{f}}(G_i)-A^{**}_{F,\vec{f}}(G_i)\big|\, \le\,  Cn^{v-3}\, +\, Cn^{v-4}\Delta(e_i)
\eqe
deterministically.

Fix $\vec{f}\in \vec{E}(F)$.  Let us write $\alpha_1,\alpha_2$ and $\beta$ for $\alpha_1(\vec{f}),\alpha_2(\vec{f})$ and $\beta(\vec{f})$ respectively, and let $\alpha=\alpha_1+\alpha_2$.  Let us first prove~\eqr{fdoubling} for this $\vec{f}$.  We shall use the notation $\pm E$ to denote an error of up $E$.  For example, we may express $|N_j|$ as $sn+D_j\pm 1$, for $j\in \{1,2\}$ and $|N_1\cap N_2|$ as $s^2 n+D_{1,2}\pm 2$.

We begin with a discussion of $A^{**}_{F,\vec{f}}(G_i)$, the expected number of embeddings $\phi(F)$ created with the addition of edge $e_i=12$ in which $\phi(\vec{f})=\vec{12}$, given $N_1$ and $N_2$.  
Let us observe that, writing $F'$ for the graph obtained by removing the vertices of $\vec{f}$, this is precisely the number of embeddings $\phi(F')$ of $F'$ in 
\[
G'_i\, :=\, G_i[V(G_i)\setminus \{1,2\}]
\]
in which
\[
\phi(\Gamma_1)\, \subseteq \, N_1\quad \text{and}\quad \phi(\Gamma_2)\, \subseteq \, N_2\, .
\]
We may thus calculate that
\begin{align*}
A^{**}_{F,\vec{f}}(G_i)\, =\, (|N_1\cap N_2|)_{\beta}& (|N_1| -\beta)_{\alpha_1}(|N_2|-\beta\pm \alpha_1)_{\alpha_2}(n-2-\alpha-\beta)_{v-2-\alpha-\beta} \phantom{\Bigg(}\\
& \cdot \frac{(i-1-|N_1|-|N_2|+|N_1\cap N_2|)_{e-1-\alpha-2\beta}}{(N')_{e-1-\alpha-2\beta}}\, ,
\end{align*}
where we have written $N'$ for $\binom{n-2}{2}$.
We may now expand each of these term to obtain a main contribution and error terms.  For example, we may express $(|N_1\cap N_2|)_{\beta}$ as
\[
(s^2 n+D_{1,2}\pm 2)_{\beta}\, =\, (s^2 n+D_{1,2}\pm 2)^\beta \, \pm \, \beta^2 n^{\beta-1}\, .
\]
Continuing, and using that $|D_{1,2}|\le n$, we may express $(|N_1\cap N_2|)_{\beta}$ as
\[
s^{2\beta}n^{\beta}\, +\, \beta s^{2\beta-2}n^{\beta-1} D_{1,2}\, \pm\, 2^\beta n^{\beta-2} D_{1,2}^2\, \pm \, 2(3^\beta+\beta^2) n^{\beta-1}\, .
\]
In particular, there is a constant $C_1=C_1(F)$, so that
\[
(|N_1\cap N_2|)_{\beta}\, =\, s^{2\beta}n^{\beta}\, +\, \beta s^{2\beta-2}n^{\beta-1} D_{1,2}\, \pm\, C_1\big( n^{\beta-1} + n^{\beta-2} \Delta(e_i)\big)\, .
\]
We may assume that $C_1=C_1(F)$ is also chosen so that the equivalent statements hold for the remaining terms.  In particular,
\[
(|N_j| -\beta\pm \alpha_1)_{\alpha_j}\, =\, s^{\alpha_j}n^{\alpha_j}\,+\, \alpha_j s^{\alpha_j-1}n^{\alpha_j-1}D_j\, \pm\, C_1 \big( n^{\alpha_j-1} + n^{\alpha_j-2} \Delta(e_i)\big)\, 
\]
for $j=1,2$,
\[
(n-2-\alpha-\beta)_{v-2-\alpha-\beta}\, =\, n^{v-2-\alpha-\beta}\, \pm \, C_1 n^{v-3-\alpha-\beta}
\]
and
\[
 \frac{(i-1-|N_1|-|N_2|+|N_1\cap N_2|)_{e-1-\alpha-2\beta}}{(N')_{e-1-\alpha-2\beta}}\, =\, s^{e-1-\alpha-2\beta}\, \pm\, C_1 n^{-1}\, .
\]
Replacing $C_1$ by a larger constant $C_2$ if necessary, it follows that
\[
A^{**}_{F,\vec{f}}(G_i)\, =\, s^{e-1}n^{v-2}\, +\, s^{e-2}n^{v-3}\big(\alpha_1 D_1+\alpha_2 D_2\big)\, +\, s^{e-3}n^{v-3}\beta D_{1,2}\, \pm \, C_2 \big( n^{v-3} + n^{v-4} \Delta(e_i)\big)\, ,
\]
completing the proof of~\eqr{fdoubling}.

All that remains is to prove~\eqr{tofdouble}.  With $e_i=12$ and the neighbourhoods $N_1$ and $N_2$ fixed we have that $A_{F,\vec{f}}(G_i)$ is a function $f(G)$ of the graph $G=G_{i-1}[V\setminus \{1,2\}]\sim G(n-2,i-1-d_1-d_2)$, and by definition, see~\eqr{AsFf}, we have $\Ex{f(G)}=A^{**}_{F,f}(G_i)$.  Furthermore $f(G)$ is $n^{v-4}$-Lipschitz, in the sense described in Section~\ref{sec:ineqs}.  By Corollary~\ref{cor:HA}, we have that
\[
\pr{|f(G)-\Ex{f(G)}|\, >\, Cb^{1/2}n^{v-3}}\, \le\, 2\exp\left(\frac{-C^2 b n^{2v-6}}{8n^{2v-6}}\right)\, \le\, \exp(-2b)
\]
provided we choose $C\ge 5$.  This precisely proves~\eqr{tofdouble}, completing the proof.
\end{proof}

We now prove Proposition~\ref{prop:Ysmall}.

\begin{proof}[Proof of Proposition~\ref{prop:Ysmall}]  Let $b\ge 3\log{n}$ be fixed.   By Lemma~\ref{lem:Clem} there is a constant $C_1$ such that the event
\eq{could}
\big|A_F(G_i)-A^{*}_F(G_i)\big|\, >\, C_1 b^{1/2}n^{v-3}\, +\, C_1 n^{v-4}\Delta(e_i)
\eqe
has probability at most $\exp(-3b)\le n^{-2}\exp(-2b)$.  We say that $G_{i-1}$ is \emph{$b$-good} if~\eqr{could} does not occur for any choice of $e_i$.  Since there are fewer than $n^2$ choices for $e_i$, it follows that
\[
\pr{G_{i-1}\, \text{is $b$-good}}\,\ge \, 1-\exp(-2b)\, .
\]
If $G_{i-1}$ is $b$-good then we have that
\[
\big| A_F(G_i)-A^{*}_F(G_i)\big|\, \le\, C_1 b^{1/2}n^{v-3}\, +\, C_1 n^{v-4}\Delta(e_i)
\]
for all possible choices of $e_i$, and
\[
\big|\Ex{A_F(G_i)-A^{*}_F(G_i)\, \big|\, G_{i-1}}\big|\, \le\, C_1 b^{1/2}n^{v-3}\, +\, C_1 n^{v-4}\Ex{\Delta(e_i)\,\big|\, G_{i-1}}\, .
\]
It follows immediately from Lemma~\ref{lem:Yis} that
\eq{Ymax}
|Y_F(G_i)|\, \le\, 2C_1 b^{1/2}n^{v-3}\, +\, C_1 n^{v-4}\Delta(e_i) +\, C_1 n^{v-4}\Ex{\Delta(e_i)\,\big|\, G_{i-1}}
\eqe
whenever $G_{i-1}$ is $b$-good.  And, since $(\alpha+\beta+\gamma)^2\le 3(\alpha^2+\beta^2+\gamma^2)$,
\eq{nowY}
Y_F(G_i)^2\, \le\, 12C_1^2 bn^{2v-6}\, +\, 3C_1^2 n^{2v-8}\Delta(e_i)^2\, +\, 3C_1^2 n^{2v-8}\Ex{\Delta(e_i)\, \big|\, G_{i-1}}^2\, 
\eqe
whenever $G_{i-1}$ is $b$-good.  

The first term is already in an appropriate form; we now consider the other two terms.  Recalling that $\Delta(e_i)=D_{u}^2(G_{i-1})+D_{w}^2(G_{i-1})+D_{uw}^2(G_{i-1})$, where $e_i=uw$, it follows that
\[
\Ex{\Delta(e_i)^2\, \big|\,G_{i-1}}\, \le\, 3\Ex{D_u^4(G_{i-1})+D_w^4(G_{i-1})+D_{uw}^4(G_{i-1})\, \big|\, G_{i-1}}\, .
\]

By the Lemmas~\ref{lem:deg4} and~\ref{lem:codeg4}, there is a constant $C_2$ such that each of the events
\eq{deg4}
\sum_{u\in V(G_{i-1})}D_{u}(G_{i-1})^4\, >\, C_2 n^{3}+C_2 b n^{2}\min\{b,n\} \, ,
\eqe
and
\eq{codeg4}
\sum_{u,w\in V(G_{i-1})}D_{u,w}(G_{i-1})^4\, >\, C_2 n^{4}+C_2 b n^{3}\min\{b,n\}\, ,
\eqe
has probability at most $\exp(-2b)$.  We say the $G_{i-1}$ is $b$-\emph{great}, if it is $b$-good, and neither of the events~\eqr{deg4},~\eqr{codeg4} occurs.  We have
\[
\pr{G_{i-1}\, \text{is $b$-great}}\, \ge \, 1-\exp(-2b)-2\exp(-2b)\, \ge \, 1-\exp(-b)\, .
\]
Finally, if $G_{i-1}$ is $b$-great, then since each vertex has probability at most
\[
\frac{n-1}{N-i+1}\, \le\, \frac{n-1}{(1-s)N}\, =\, \frac{2}{(1-s)n}
\]
of being included in $e_i$, and each remaining pair has probability at most $3/(1-s)n^2$ of being $e_i$, we have
\begin{align*}
\Ex{\Delta(e_i)^2\, \big|\,G_{i-1}}\, &\le\, \frac{6}{(1-s)n}\sum_{u\in V(G_{i-1})}D_{u}(G_{i-1})^4\, +\, \frac{9}{(1-s)n^2}\sum_{u,w\in V(G_{i-1})}D_{u,w}(G_{i-1})^4\\
& \le\, 20C_2 n^2 \, +\, 20 C_2 b n \min\{b,n\}\\
& \le \, C_3 bn^2\,  ,
\end{align*}
where $C_3=40C_2$.  Taking conditional expectations in~\eqr{nowY}, and using the bound on $\Ex{\Delta(e_i)^2|G_{i-1}}$, we have
\[
\Ex{Y_F(G_i)^2\, \big|\, G_{i-1}}\, \le\, Cb n^{2v-6}\, ,
\]
where $C=12C_1^2(C_3+1)$, whenever $G_{i-1}$ is $b$-great.  This completes the proof of the proposition.
\end{proof}

\subsection{Proof of Theorem~\ref{thm:approx}}\label{sec:approx}

In this section we show how we may deduce Theorem~\ref{thm:approx} from Proposition~\ref{prop:Ysmall}.  

Let $t\in (0,1/2]$ and let $H$ be a graph with $v$ vertices and $e$ edges.  The main statement of Theorem~\ref{thm:approx} is that there exists a constant $C=C(H)$ such that
\eq{approxagain}
\pr{\big| D_H(G_{n,t})-\Lambda^{*}_H(G_{n,t})\big |\, >\, Cbt^{1/2}n^{v-2}}\, \le\, \exp(-b)
\eqe
for all $3\log{n}\le b\le t^{1/2}n$, where $\Lambda^{*}_H(G_{n,t})=\sum_{i=1}^{m}\X_H(G_i;t)$ is the sum of the increments
\[
\X_H(G_i;t) \, = \, n^{v-3} \left(t^{e-2}\obinom{H}{\owedge}\frac{(1-t)^2}{(1-s)^2}X_{\owedge}(G_i)\, +\, t^{e-3}\obinom{H}{\triangle} \frac{(1-t)^3}{(1-s)^{3}}\, \big(X_{\triangle}(G_i)-3sX_{\owedge}(G_i)\big) \right)\, .
\]

The proof will use the triangle inequality, bounding the difference between $D_H(G_{n,t})$ and $\Lambda^{*}_H(G_{n,t})$ via 
\eq{Lassdef}
\Lambda^{**}_H(G_{n,t})\, :=\, \sum_{i=1}^{m}\, \sum_{F\ssq E(H)}\frac{(1-t)^{e(F)}(t-s)^{e-e(F)}}{(1-s)^e}\, X_F(G_i)\, ,
\eqe
where as usual $m$ denotes $\lfloor tN\rfloor$.

Notice that $\Lambda^{**}_H(G_{n,t})$ is close to the martingale expression for $D_H(G_m)$, given by Theorem~\ref{thm:Mart}, with $m=\lfloor tN\rfloor$, except with coefficients
\[
\frac{(1-t)^{e(F)}(t-s)^{e-e(F)}}{(1-s)^e} \qquad \text{in place of } \qquad \frac{(N-m)_{e(F)}(m-i)_{e-e(F)}}{(N-i)_e}\, .
\]
The following lemma bounds the difference between these coefficients.  For fixed constants $0\le c\le e$, and any $1\le i\le m\le N$, 
\[
\nu_{c,e}(i,m)\, :=\,\frac{(N-m)_{c}(m-i)_{e-c}}{(N-i)_e}\, -\, \frac{(1-t)^{c}(t-s)^{e-c}}{(1-s)^e}\, .
\]

\begin{lem}\label{lem:coeffs}
Let $t\in (0,1/2]$ and $c,e\in \mathbb{N}$.  There is a constant $C=C(c,e)$ such that for all $1\le i\le m\le tN$, we have
\[
|\nu_{c,e}(i,m)|\, \le\, \frac{C}{n^2}\, .
\]
\end{lem}

\begin{proof}  We will show that the constant $C=24e^2$ works for all sufficiently large $n$.  One may then adjust $C$ so that the result holds trivially for all smaller values of $n$.

Set $k=N-i$ and $\ell=N-m$.  We have
\[
\nu_{c,e}(i,m)\, =\, \frac{(\ell)_{c}(k-\ell)_{e-c}}{(k)_e}\, -\, \frac{\ell^{c}(k-\ell)^{e-c}}{k^e}\, =\, \frac{(\ell)_{c}(k-\ell)_{e-c}k^{e-1}-\ell^c(k-\ell)^{e-c}(k-1)_{c-1}}{(k-1)_{e-1}k^e}\, .
\]
The numerator of this expression may be written as
\[
[(\ell)_c-\ell^c]\, (k-\ell)_{e-c}k^{e-1}\, +\, [(k-\ell)_{e-c}-(k-\ell)^{e-c}]\ell^c k^{e-1} \, +\, [k^{e-1}-(k-1)_{e-1}]\ell^c(k-\ell)^{e-c}\, .
\]
Since $\ell\le k$, and the highest order term cancels in each of the square brackets, the numerator has absolute value at most
\[
3e^2 k^{2e-2}\, .
\]
On the other hand the denominator is at least
\[
k^e (k-1)_{e-1}\, \ge\, \frac{1}{2} k^{2e-1}
\]
for all sufficiently large $n$.  And so
\[
|\nu_{c,e}(i,m)|\, \le\, 6e^2 k^{-1}\, =\, \frac{6e^2}{N-m}\, \le\, \frac{12e^2}{(1-t)n^2}\, .
\]
Since $t\in (0,1/2]$, this complete the proof.
\end{proof}

We are now nearly ready to prove Theorem~\ref{thm:approx}.  Before doing so we require one more lemma.  We may view $Y_{F}(G_i)$ as a function of $G_{i-1}$ and $e_i$.  Let us write $\|Y_{F}|G_{i-1}\|_{\infty}$ for the maximum possible value of $|Y_F(G_i)|$ over the possible choices $e\in E(K_n)\setminus E(G_{i-1})$ of the $i$th edge.

\begin{lem}\label{lem:maxY}  Let $F$ be a graph with $v(F)$ vertices and $e(F)$ edges.  There is a constant $C=C(F)$ such that, for all $1\le i\le N$ and all $b\ge 3\log{n}$, the event that 
\[
\|Y_{F}|G_{i-1}\|_{\infty}\, >\, Cb^{1/2}n^{v(F)-3}\, +\, Cbsn^{v(F)-3}\,+\, Cb^{3/2}s^{1/2}n^{v(F)-7/2}\, +\, Cb^2n^{v-4}\, 
\]
has probability at most $\exp(-b)$.
\end{lem}

\begin{proof} With the usage of $b$-good introduced in the proof of Proposition~\ref{prop:Ysmall}, we have that $G_{i-1}$ is $b$-good with probability at least $1-\exp(-2b)$ and we recall~\eqr{Ymax}, which states that for some constant $C_1$ we have
\[
|Y_F(G_i)|\, \le\, 2C_1 b^{1/2}n^{v(F)-3}\, +\, C_1 n^{v(F)-4}\Delta(e_i) +\, C_1 n^{v(F)-4}\Ex{\Delta(e_i)\,\big|\, G_{i-1}}
\]
whenever $G_{i-1}$ is $b$-good.  Let $F_{\Delta}$ be the event that some $e\in E(K_n)\setminus E(G_{i-1})$ has
\[
\Delta(e)\, >\, 1000\Big(bsn\, +\, b^{3/2}s^{1/2}n^{1/2}\, +\, b^2\Big)\, .
\]
It follows easily from Lemma~\ref{lem:maxdegdev} and Lemma~\ref{lem:maxcodegdev} that $\pr{F_{\Delta}}\, \le\, 6\exp(-2b)$.  

We may now observe that there exists a constant $C$ such that the event that
\[
\|Y_{F}|G_{i-1}\|_{\infty}\, >\, Cb^{1/2}n^{v(F)-3}\, +\, Cbsn^{v(F)-3}\,+\, Cb^{3/2}s^{1/2}n^{v(F)-7/2}\, +\, Cb^2n^{v(F)-4}
\]
is contained in $F_{\Delta} \cup \{\text{$G_{i-1}$ is not $b$-good}\}$.  This probability is at most
\[
6\exp(-2b)\, +\, \exp(-2b)\, \le\, \exp(-b)\, ,
\]
as required.
\end{proof}

\begin{proof}[Proof of Theorem~\ref{thm:approx}]
Let $t\in (0,1/2]$.  We shall focus on the proof of the first statement.  To deduce the ``Furthermore'' statement, simply follow the same proof, with the variable $t$ removed, and use $R=2^{e+1}n^{v-2}$.

Fix $3\log{n}\le b\le t^{1/2}n$.  By the triangle inequality it clearly suffices to prove
\eq{approx1}
\big| D_H(G_{n,t})-\Lambda^{**}_H(G_{n,t})\big|\,\le \, C_1 t n^{v-2}\, 
\eqe
deterministically, and
\eq{approx2}
\pr{\big| \Lambda^{**}_H(G_{n,t})-\Lambda^{*}_H(G_{n,t})\big |\, >\, C_2 bt^{1/2} n^{v-2}}\, \le\, \exp(-b)
\eqe
for all $3\log{n}\le b\le t^{1/2}n$, for some constants $C_1$ and $C_2$.

We begin with~\eqr{approx1}.  By Theorem~\ref{thm:Mart}, we have the precise martingale expression for $D_H(G_{n,t})$ given by
\[
D_{H}(G_{n,t}) \, =\, \sum_{i=1}^{m}\, \sum_{F\ssq E(H)}\frac{(N-m)_{e(F)}(m-i)_{e-e(F)}}{(N-i)_e}\, X_F(G_i)\, ,
\]
where $m=\lfloor tN\rfloor$.  It follows that $D_H(G_{n,t})-\Lambda^{**}(G_{n,t})$ is
\[
\sum_{i=1}^{m}\, \sum_{F\ssq E(H)} \nu_{e(F),e}(i,m)\, X_{F}(G_i)\, .
\]
Since each $X_F(G_i)$ is at most $n^{v-2}$ deterministically and $\nu_{e(F),e}(i,m)$ is at most $C_3/n^2$, where $C_3$ is the constant given by Lemma~\ref{lem:coeffs}, it follows that this difference is at most 
\[
m  n^{v-2} \frac{C_1}{n^2}\, \le\, C_1 t n^{v-2}
\] 
deterministically, where $C_1=2^{e} C_3$. 

We now prove~\eqr{approx2}.   Let $3\log{n}\le b\le t^{1/2}n$ be fixed.  The proof proceeds by replacing the $X_{F}(G_i)$ in $\Lambda^{**}_H(G_{n,t})$ by
\[
X^*_F(G_i)\, +\, Y_F(G_i)
\]
where
\[
X^*_F(G_i)\, := \, n^{v-3}s^{e(F)-2}\left(\obinom{F}{\owedge}-3\obinom{F}{\triangle}\right) X_{\owedge}(G_i)\, +\, n^{v-3}s^{e(F)-3}\obinom{F}{\triangle}X_{\triangle}(G_i)\, .
\]
We claim that the $X^*_F(G_i)$ contribute exactly $\Lambda^{*}_{H}(G_{n,t})$, so that:

\noindent\textbf{Claim:} 
\eq{Ldiffis}
\Lambda^{**}_H(G_{n,t})\, -\,\Lambda^{*}_H(G_{n,t})\, =\, \sum_{i=1}^{m}\, \sum_{F\ssq E(H)}\, \frac{(1-t)^{e(F)}(t-s)^{e-e(F)}}{(1-s)^e}Y_F(G_i)\, .
\eqe

\noindent\textbf{Proof of Claim:} We must prove that 
\[
\sum_{i=1}^{m}\, \sum_{F\ssq E(H)}\, \frac{(1-t)^{e(F)}(t-s)^{e-e(F)}}{(1-s)^e}X^*_F(G_i)\, =\, \Lambda^{*}_H(G_{n,t}) \, .
\]
That is, we must prove that 
\[
\sum_{F\ssq E(H)}\, \frac{(1-t)^{e(F)}(t-s)^{e-e(F)}}{(1-s)^e}X^*_F(G_i)\, =\, \X_H(G_i;t)\, ,
\]
for all $i=1,\dots ,m$.  It is clear that both sides are linear combinations of the increments $X_{\owedge}(G_i)$ and $X_{\triangle}(G_i)$, so it suffices to prove they receive the same coefficients on each side.  We begin with $X_{\owedge}(G_i)$, which receives coefficient 
\eq{cris}
n^{v-3}t^{e-2}\obinom{H}{\owedge}\frac{(1-t)^2}{(1-s)^2}\, -\, 3n^{v-3} s t^{e-3}\obinom{H}{\triangle} \frac{(1-t)^3}{(1-s)^{3}}
\eqe
on the right hand side, and coefficient ~\eqr{clwedge}$-3\times$\eqr{cltriangle} on the left, where~\eqr{clwedge} and~\eqr{cltriangle} are given by
\eq{clwedge}
n^{v-3}\sum_{F\subseteq E(H)}\frac{(1-t)^{e(F)}(t-s)^{e-e(F)}}{(1-s)^e} s^{e(F)-2}\obinom{F}{\owedge}
\eqe
and
\eq{cltriangle}
n^{v-3}\sum_{F\subseteq E(H)}\frac{(1-t)^{e(F)}(t-s)^{e-e(F)}}{(1-s)^e} s^{e(F)-2}\obinom{F}{\triangle}\, .
\eqe
Since there is a contribution to~\eqr{clwedge} for each copy of $P_2$ contained in the subgraph $F\subseteq E(H)$, one may sum first over copies of $P_2$ contained in $H$, with each having a contribution equal to the total contribution of subgraphs $F\subseteq E(H)$ which contain it.  Thus~\eqr{clwedge} is $n^{v-3}\binom{H}{\owedge}$ times
\begin{align*}
& \sum_{P\subseteq F\subseteq E(H)}\frac{(1-t)^{e(F)}(t-s)^{e-e(F)}}{(1-s)^e} s^{e(F)-2}\\
&  \qquad = \frac{(1-t)^2(t-s)^{e-2}}{(1-s)^e}\sum_{F' \subseteq E(H) \setminus P}\left(\frac{s(1-t)}{t-s}\right)^{e(F')}\, ,
\end{align*}
where $P$ is some copy of $P_2$ of $H$.  Summing, using the binomial identity, reveals that~\eqr{clwedge} is precisely
\[
n^{v-3}\obinom{H}{\owedge} \frac{t^{e-2}(1-t)^2}{(1-s)^2}\,.
\]
Similarly,~\eqr{cltriangle} is $n^{v-3}\binom{H}{\triangle}$ times
\begin{align*}
& \sum_{T\subseteq F\subseteq E(H)}\frac{(1-t)^{e(F)}(t-s)^{e-e(F)}}{(1-s)^e} s^{e(F)-2} \\
& \qquad = \frac{s(1-t)^3(t-s)^{e-3}}{(1-s)^e}\sum_{F' \subseteq E(H) \setminus T} \left( \frac{s(1-t)}{t-s}\right)^{e(F')}\, ,
\end{align*}
where $T$ is some triangle of $H$.  By the binomial identity, we find that~\eqr{cltriangle} is precisely
\[
n^{v-3}\obinom{H}{\triangle} \frac{s t^{e-3}(1-t)^3}{(1-s)^3}\, .
\]
The coefficient on the left, ~\eqr{clwedge}$-3\times$\eqr{cltriangle}, is equal to that on the right,~\eqr{cris}.

Similar calculations confirm that $X_{\triangle}(G_i)$ receives coefficient
\[
n^{v-3} t^{e-3}\obinom{H}{\triangle} \frac{(1-t)^3}{(1-s)^{3}}
\]
on both sides, completing the proof of the Claim.

Now to complete the proof of~\eqr{approx2}, it suffices to prove, for some constant $C_2$, that
\eq{sumY}
\sum_{i=1}^{m}\, \sum_{F\ssq E(H)}\, \frac{(1-t)^{e(F)}(t-s)^{e-e(F)}}{(1-s)^e}Y_F(G_i)
\eqe
is at most $C_2 bt^{1/2} n^{v-2}$, in absolute value, with probability at least $1-\exp(-b)$.

Let $C_4$ be a constant $3e(H)$ times larger than the largest constant required by Proposition~\ref{prop:Ysmall} for a subgraph $F\subseteq E(H)$ and $C_5$ the equivalent for Lemma~\ref{lem:maxY}.  With these choices we have
\eq{2bound}
\pr{\Ex{Y_F(G_i)^2\, \big|\,G_{i-1}}\, >\, C_4 b n^{2v-6}}\, \le\, \exp(-3e(H)b)\, \le \, n^{-2}2^{-e(H)}\exp(-2b)\, 
\eqe
for all $F\subseteq E(H)$.  
Let $E_H(m)$ be the event that for some $F\subseteq E(H)$ and $1\le i\le m$ either
\[
\Ex{Y_F(G_i)^2\, \big|\,G_{i-1}}\, >\, C_4 b n^{2v-6}
\]
or
\[
\|Y_{F}|G_{i-1}\|_{\infty}\, >\, C_5 b^{1/2}n^{v-3}\, +\, C_5 btn^{v-3}\,+\, C_5 b^{3/2}t^{1/2}n^{v-7/2}\, +\, C_5 b^2n^{v-4}\, 
\]
occurs.  By~\eqr{2bound}, Lemma~\ref{lem:maxY} and a union bound, we have
\[
\pr{E_H(m)}\, \le\, \exp(-2b)\, .
\]
Let us define
\[
Y^*_F(G_i)\, =\, Y_F(G_i)\, 1_{E_H(i)^c}\, .
\]
We observe that the $Y^*_F(i)$ are also martingale increments, in the sense that
\[
\Ex{Y^*_F(G_i)\, \big|\,G_{i-1}}\, =\, 0\, .
\]
We observe further that they satisfy
\[
\Ex{Y^*_F(G_i)^2\, \big|\,G_{i-1}}\, \le\, C_4 b n^{2v-6}
\]
and
\[
|Y^*_F(G_i)|\, \le\, C_5 b^{1/2}n^{v-3}\, +\, C_5 btn^{v-3}\,+\, C_5 b^{3/2}t^{1/2}n^{v-7/2}\, +\, C_5 b^2n^{v-4}\, ,
\]
almost surely, and 
\eq{sumY*}
\sum_{i=1}^{m}\, \sum_{F\ssq E(H)}\, \frac{(1-t)^{e(F)}(t-s)^{e-e(F)}}{(1-s)^e}Y^*_F(G_i)
\eqe
is equal to~\eqr{sumY} on $\Omega\setminus E^Y_H(m)$.

We bound the probability that~\eqr{sumY*} is large using Freedman's inequality, applied to the martingale~\eqr{sumY*}, with increments
\[
\sum_{F\ssq E(H)}\, \frac{(1-t)^{e(F)}(t-s)^{e-e(F)}}{(1-s)^e}Y^*_F(G_i)\, .
\]
Furthermore, since the coefficients are all at most $1$, we have
\[
\Ex{\left(\sum_{F\ssq E(H)}\, \frac{(1-t)^{e(F)}(t-s)^{e-e(F)}}{(1-s)^e}Y^*_F(G_i)\right)^2\, \Bigg|\, G_{i-1}}\, \le\, 4^{e(H)}C_4 b n^{2v-6}
\]
almost surely.  We now apply Freedman's inequality, Lemma~\ref{lem:F}, to~\eqr{sumY*}, with $\alpha=C_2 bt^{1/2} n^{v-2}$, 
\[
\beta\, =\, 4^{e(H)}C_4 b m n^{2v-6}\, \le \, 4^{e(H)}C_4 bt n^{2v-4}
\]
and 
\begin{align*}
R\, & =\, C_6 b^{1/2}n^{v-3}\, +\, C_6 btn^{v-3}\,+\, C_6 b^{3/2}t^{1/2}n^{v-7/2}\, +\, C_6 b^2n^{v-4}  \\
 & \le \, 4C_6 t^{1/2} n^{v-2}\, ,
\end{align*}
where the inequality relies on the condition $b\le t^{1/2} n$, and where we have chosen $C_6=2^{e(H)}C_5$.
We obtain that the probability that~\eqr{sumY*} exceeds $C_2 bt^{1/2} n^{v-2}$ in absolute value is at most
\[
\exp\left(\frac{-C_2^2 b^2 t n^{2v-4}}{4^{e(H)}C_4 bt n^{2v-4}+8 C_2 C_6  btn^{2v-4}}\right) \, \le\, \exp(-b)\, ,
\]
provided $C_2\ge 2^{e(H)+1}C_4 C_6$, completing the proof of the theorem.
\end{proof}

\subsection{Proof of Theorem~\ref{thm:relate}}\label{ssec:relate}

We now show how we may deduce Theorem~\ref{thm:relate} from Theorem~\ref{thm:approx}.  The main statement of Theorem~\ref{thm:relate} is that
\[
\pr{\big|D_{H}(G_{n,t})-\Lambda_{H}(G_{n,t})\big| >Cbt^{1/2}n^{v-2}}\, \le\, \exp(-b)\, 
\]
for some constant $C=C(H)$, and for all $3\log{n}\le b\le t^{1/2}n$, where
\eq{Lambagain}
\Lambda_{H}(G_{n,t})\, :=\, n^{v-3}t^{e-2}\left(\obinom{H}{\owedge}-3\obinom{H}{\triangle}\right)  D_{\owedge}(G_{n,t})\, +\, n^{v-3}t^{e-3}\obinom{H}{\triangle} D_{\triangle}(G_{n,t})\, .
\eqe

We shall use the triangle inequality to control the difference between $D_{H}(G_{n,t})$ and $\Lambda_{H}(G_{n,t})$ via $\Lambda^*_H(G_{n,t})$.

We prove Theorem~\ref{thm:relate} first for $t\in (0,1/2]$, and then show how we may deduce the result for $t\in (1/2,1)$.

\begin{proof}[Proof of Theorem~\ref{thm:relate} for $t\in (0,1/2\rbrack$] 
Let $t\in (0,1/2]$.  We shall focus on proof of the main statement.  To deduce the ``Furthermore'' statement simply follow the same proof, with the variable $t$ removed, and use the ``Furthermore'' part of Theorem~\ref{thm:approx}.

Let $3\log{n}\le b\le t^{1/2}n$ be fixed.
By Theorem~\ref{thm:approx}, there is a constant $C=C(H)$, such that
\[
\pr{\big| D_H(G_{n,t})-\Lambda^{*}_H(G_{n,t})\big |\, >\, Cbt^{1/2} n^{v-2}}\, \le\, \exp(-b)
\]
for all $3\log{n}\le b\le t^{1/2}n$.  So, by the triangle inequality, it suffices to prove the bound
\eq{lambneed}
\big|\Lambda^{*}_H(G_{n,t})-\Lambda_H(G_{n,t})\big|\, \le \, C'tn^{v-2}
\eqe
deterministically, for some constant $C'=C'(H)$.  Using Theorem~\ref{thm:Mart}, the precise martingale expression for $D_H(G_{n,t})$, to expand $D_{\owedge}(G_{n,t})$ and $D_{\triangle}(G_{n,t})$ in terms of $X_{\owedge}(G_i)$ and $X_{\triangle}(G_i)$ we may express $\Lambda_H(G_{n,t})$ as a sum of the form\footnote{One does not need to include terms $X_{F}(G_i)$ for graphs $F$ with $e(F)\le 1$, as $X_F(G_i)=0$ in all such cases.}
\[
\sum_{i=1}^{m}\, \big( \alpha(i,m)X_{\owedge}(G_i)\,+\, \beta(i,m)X_{\triangle}(G_i) \big) \, .
\]
On the other hand $\Lambda^{*}_H(G_{n,t})$ is already of the form
\[
\sum_{i=1}^{m}\, \big( \alpha'(i,m)X_{\owedge}(G_i)\,+\, \beta'(i,m)X_{\triangle}(G_i) \big)\, .
\]
We shall prove that
\eq{alphadif}
\alpha(i,m)-\alpha'(i,m)\, =\, n^{v-3}t^{e-2} \left(\obinom{H}{\owedge}-3\obinom{H}{\triangle}\right)\nu_{2,2}(i,m)+3n^{v-3}t^{e-3}\obinom{H}{\triangle}\nu_{2,3}(i,m)
\eqe
and
\eq{betadif}
\beta(i,m)-\beta'(i,m)\, =\, n^{v-3}t^{e-3}\nu_{3,3}(i,m)\, .
\eqe
By Lemma~\ref{lem:coeffs}, the $\nu$ values are $O(n^{-2})$, and so, based on~\eqr{alphadif} and~\eqr{betadif} these differences are $O(n^{v-5})$.  Since each of $X_{\owedge}(G_i)$ and $X_{\triangle}(G_i)$ has absolute value at most $n$ deterministically, and the sums each have $m\le tN$ terms, this completes the proof of~\eqr{lambneed}, and therefore the whole proof.

All that remains is to verify~\eqr{alphadif} and~\eqr{betadif}.  We observe that~\eqr{betadif} follows immediately from the definitions.  In order to prove~\eqr{alphadif}, let us calculate
\eq{alphaminus}
\alpha(i,m)\, -\, n^{v-3}t^{e-2} \left(\obinom{H}{\owedge}-3\obinom{H}{\triangle}\right)\nu_{2,2}(i,m)\, -\, 3n^{v-3}t^{e-3}\obinom{H}{\triangle}\nu_{2,3}(i,m)\, .
\eqe
We have that
\[
\alpha(i,m)\, =\, n^{v-3}t^{e-2} \left(\obinom{H}{\owedge}-3\obinom{H}{\triangle}\right) \frac{(N-m)_2}{(N-i)_2}\, +\, 3n^{v-3} t^{e-3}\obinom{H}{\triangle}\frac{(N-m)_2(m-i)}{(N-i)_3}\,.
\]
And so, using the definition of $\nu_{c,e}(i,m)$, we have that~\eqr{alphaminus} equals
\[
n^{v-3}t^{e-2} \left(\obinom{H}{\owedge}-3\obinom{H}{\triangle}\right) \frac{(1-t)^2}{(1-s)^2}\, +\, 3n^{v-3} t^{e-3}\obinom{H}{\triangle}\frac{(1-t)^2(t-s)}{(1-s)^3}\, .
\]
Cancelling, we obtain
\[
n^{v-3}t^{e-2}\obinom{H}{\owedge}\frac{(1-t)^2}{(1-s)^2}\, -\, 3n^{v-3}st^{e-3}\obinom{H}{\triangle}\frac{(1-t)^3}{(1-s)^3}
\]
which is precisely $\alpha'(i,m)$, completing the verification of~\eqr{alphadif} and therefore the proof.
\end{proof}

We have now proved Theorem~\ref{thm:relate} for $t\in (0,1/2]$.  We deduce the cases $t\in (1/2,1)$ by considering the complementary graph and using Corollary~\ref{cor:comp}.

\begin{proof}[Proof of Theorem~\ref{thm:relate} for $t\in (1/2,1)$] For this range of $t$ it suffices to prove the ``Furthermore'' statement.  Indeed, up to a change of the constant this implies the main statement.  Fix $t\in (1/2,1]$ and $b\ge 3\log{n}$.  Let $t'=1-t\in (0,1/2]$.  Let $C''$ be the maximum over subgraphs $F\subseteq E(H)$ of the constant obtained by the proof of Theorem~\ref{thm:relate} in the case $t\in (0,1/2]$, let $C'=eC''$ and $C=2^e C'$.  By Theorem~\ref{thm:relate} for $t\in (0,1/2]$ we have
\[
\pr{\big|D_F(G_{n,t'})\, -\, \Lambda_F(G_{n,t'})\big|\, >\, C'bn^{v-2}}\, \le\, \exp(-eb)\, \le \, 2^{-e}\exp(-b)\, .
\]
Thus, by a union bound there is probability at least $1-\exp(-b)$ that
\eq{forallF}
\big|D_F(G_{n,t'})\, -\, \Lambda_F(G_{n,t'})\big|\, \le \, C'bn^{v-2}
\eqe
for all $F\subseteq E(H)$.  We complete the proof by showing that if~\eqr{forallF} holds for $G_{n,t'}=G_{n,t}^c$, then
\eq{then}
\big|D_H(G_{n,t})-\Lambda_H(G_{n,t})\big|\, \le\, Cb n^{v-2}\, .
\eqe
(It is elementary that the complement of $G_{n,t}$ is distributed as $G_{n,t'}$.)

We now prove~\eqr{then} which will complete the proof of the theorem.   We shall use Corollary~\ref{cor:comp} which allows us to relate subgraph count deviations to those in the complement.  By Corollary~\ref{cor:comp} we have that
\[
D_{H}(G_{n,t})\, =\, \sum_{F\subseteq E(H)}(-1)^{e(F)}D_{F}(G_{n,t}^c)\, .
\]
If~\eqr{forallF} holds in $G^c_{n,t}$ (which has the same distribution as $G_{n,t'}$) then
\[
D_H(G_{n,t})\, =\, \sum_{F\subseteq E(H)}(-1)^{e(F)}\Lambda_F(G_{n,t'})\, \pm\, Cbn^{v-2} \, .
\]
We claim that the main sum
\eq{mainsum}
\sum_{F\subseteq E(H)}(-1)^{e(F)}\Lambda_F(G_{n,t'})
\eqe
is equal to $\Lambda_H(G_{n,t})$.  Clearly proving this fact will complete the proof.


By the definition of $\Lambda_F(G_{n,t'})$, see~\eqr{Lambagain}, we can rewrite~\eqr{mainsum} as~\eqr{mainwedge} $+$ \eqr{maintri}, defined by: 
\eq{mainwedge}
n^{v-3}\sum_{F\subseteq E(H)}(-1)^{e(F)}(t')^{e(F)-2}\left(\obinom{F}{\owedge}-3\obinom{F}{\triangle}\right) D_{\owedge}(G_{n,t'})
\eqe
and
\eq{maintri}
n^{v-3}\sum_{F\subseteq E(H)}(-1)^{e(F)}(t')^{e(F)-3}\obinom{F}{\triangle} D_{\triangle}(G_{n,t'})\, .
\eqe
Summing over $P_2$s and triangles of $H$ and using the binomial identity, as in the proof of Theorem~\ref{thm:approx}, we obtain that~\eqr{mainwedge} is equal to
\eq{nowwedge}
n^{v-3}t^{e-2}\obinom{H}{\owedge}D_{\owedge}(G_{n,t'})\, +\, 3n^{v-3}t' t^{e-3}\obinom{H}{\triangle} D_{\owedge}(G_{n,t'})\, ,
\eqe
while~\eqr{maintri} is equal to
\eq{nowtri}
-n^{v-3}t^{e-3}\obinom{H}{\triangle} D_{\triangle}(G_{n,t'})\, .
\eqe
Again using Corollary~\ref{cor:comp}, and using the fact that we are taking $G_{n,t'}$ to be the complement of $G_{n,t}$, we have
\[
D_{\owedge}(G_{n,t'})\, =\, D_{\owedge}(G_{n,t})\qquad \text{and}\qquad D_{\triangle}(G_{n,t'})\, =\, -D_{\triangle}(G_{n,t})+3D_{\owedge}(G_{n,t})\, .
\]
Substituting these values in~\eqr{nowwedge} and~\eqr{nowtri}, we obtain that~\eqr{mainsum} is
\[
n^{v-3}t^{e-2}\left(\obinom{H}{\owedge}-3\obinom{H}{\triangle}\right) D_{\owedge}(G_{n,t})\, +\, n^{v-3}t^{e-3}\obinom{H}{\triangle} D_{\triangle}(G_{n,t})\, .
\]
This proves that~\eqr{mainsum} is equal to $\Lambda_{H}(G_{n,t})$, and therefore completes the proof.
\end{proof}

\subsection{Deducing Theorem~\ref{thm:approxbetter}}

We recall that Theorem~\ref{thm:approx} is stated for $t=t(n)\in (0,1/2)$.  Having proved Theorem~\ref{thm:relate} we may now deduce that Theorem~\ref{thm:approx} applies for $t\in (0,1)$.  This was stated as Theorem~\ref{thm:approxbetter}.

\begin{proof}[Proof of Theorem~\ref{thm:approxbetter}]
In proving~\eqr{lambneed} above, we established that
\[
\big|\Lambda^{*}_H(G_{n,t})-\Lambda_H(G_{n,t})\big|\, \le \, C'tn^{v-2}
\]
deterministically.  It is now immediate by observation that Theorem~\ref{thm:approxbetter} follows from Theorem~\ref{thm:relate} and the triangle inequality.
\end{proof}

\section{A general bound on deviation probabilities -- Theorem~\ref{thm:upto}}\label{sec:upto}

In this section we prove Theorem~\ref{thm:upto}.  We recall that this theorem gives a weaker bound on subgraph count deviations than Theorem~\ref{thm:main}.  However, it applies across the whole range of possible deviations and gives an exponent which is best possible up to multiplication by constant.

Our proof will rely on using Theorem~\ref{thm:approxbetter}, which states that $D_{H}(G_{n,t})$ is well approximated by $\Lambda^{*}_H(G_{n,t})$, and the following proposition.

\begin{prop}\label{prop:forH}
Let $H$ be a graph with $v$ vertices and $e$ edges.  There is a constant $c=c(H)>0$ such that for all $t\in (0,1/2]$, and all $\eta,n \ge c^{-1}$, we have
\[
\pr{\big|\Lambda^{*}_H(G_{n,t})\big|\, >\, \eta n^{v-3/2}}\, \le\, \exp\big(-c\eta \min\{\eta,n^{1/2}\}\big)\,  .
\]
\end{prop}

Let us first show that Theorem~\ref{thm:upto} follows easily from these results.

\begin{proof}[Proof of Theorem~\ref{thm:upto}]
We prove the result for $t\in (0,1/2]$; up to changing the constant, the result then follows for $t\in (1/2,1)$ by Corollary~\ref{cor:comp}.  Now suppose that $t\in (0,1/2]$, has been fixed.  Let us also fix the graph $H$ with $v$ vertices and $e$ edges.

For all $\alpha,n$, we have, by the triangle inequality,
\eq{triineq}
\pr{\big|D_H(G_{n,t})\big|\, >\, \alpha n^{v-3/2}}\, \le\, \pr{\big|\Lambda^{*}_H(G_{n,t})\big|\,  >\, \frac{\alpha}{2}n^{v-3/2}}\, +\, \pr{F(\alpha/2)}\, ,
\eqe
where $F(b)$ is the event that
\[
\big|D_H(G_{n,t})-\Lambda^*_H(G_{m,t})\big|\, >\, bn^{v-3/2}\, .
\]

We bound the first probability by applying Proposition~\ref{prop:forH}. We obtain
\[
\pr{\big|\Lambda^{*}_H(G_{n,t})\big|\,  >\, \frac{\alpha}{2}n^{v-3/2}}\, \le\, \exp\big(-c'\alpha \min\{\alpha, n^{1/2}\}\big)\, ,
\]
where $c'$ is a quarter of the constant of that proposition.

We bound the second probability using Theorem~\ref{thm:approxbetter}.   Let $C=C(H)$, be the constant given by Theorem~\ref{thm:approxbetter}.  By Theorem~\ref{thm:approxbetter}, we have that
\[
\pr{F(\alpha/2)}\, \le\, \exp\left(\frac{-\alpha n^{1/2}}{2C}\right)\, .
\]

Substituting these bounds into~\eqr{triineq}, we obtain
\begin{align*}
\pr{\big|D_H(G_{n,t})\big|\, >\, \alpha n^{v-3/2}}\, & \le\, \exp\big(-c'\alpha \min\{\alpha, n^{1/2}\}\big)\, +\, \exp\left(\frac{-\alpha n^{1/2}}{2C}\right) \\
& \le \, \exp\big(-c\alpha \min\{\alpha, n^{1/2}\}\big)\, ,
\end{align*}
where $c$ is taken to be at most $\min\{c'/2,1/4C\}$.
\end{proof}

All that remains to complete the section is to prove Proposition~\ref{prop:forH}.  We require the following lemma.

\begin{lem}\label{lem:basicvar} There is a constant $C$ such that for all $1\le i\le N/2$, and all $\eta\ge 1$ there is probability at least $1-\exp(-\eta n^{1/2})$ that
\eq{Du2}
\Ex{X_{\owedge}(G_i)^2\, \big|\,G_{i-1}}\, \le\, Cn^{1/2}\max\{\eta,n^{1/2}\}
\eqe
and
\eq{Duw2}
\Ex{X_{\triangle}(G_i)^2\,\big|\, G_{i-1}}\, \le \, Cn^{1/2}\max\{\eta,n^{1/2}\} \, .
\eqe
\end{lem}

\begin{proof}
Let $C'$ be twice the larger of the constants given by Lemma~\ref{lem:deg2} and Lemma~\ref{lem:codeg2}.  Let $E_1$ be the event that 
\[
\sum_{u}D_u(G_{i-1})^2\, >\, C' n^{3/2}\max\{\eta,n^{1/2}\}
\]
and $E_2$ the event that
\[
\sum_{u,w}D_{u,w}(G_{i-1})^2\, >\, C' n^{5/2}\max\{\eta,n^{1/2}\}\, .
\]
By considering the cases $\eta\le n^{1/2}$ and $\eta> n^{1/2}$, it follows from Lemma~\ref{lem:deg2} that
\[
\pr{E_1}\, \le\, \exp(-2n)1_{\eta\le n^{1/2}}\, +\, \exp(-2\eta n^{1/2})1_{\eta> n^{1/2}}\, \le\, \exp(-2\eta n^{1/2})\, .
\]
Using Lemma~\ref{lem:codeg2} one may obtain the same bound on $\pr{E_2}$, so that
\[
\pr{E_1 \cup E_2}\, \le\, 2\exp(-2\eta n^{1/2})\, \le\,  \exp(-\eta n^{1/2})\, .
\]
It therefore suffices to prove that the event that~\eqr{Du2} fails is contained in $E_1$, and the event that~\eqr{Duw2} fails is contained in $E_2$.

Let us now find a bound on $\Ex{X_{\owedge}(G_i)^2|G_{i-1}}$ which will show that~\eqr{Du2} holds in $E_1^c$.  We recall that $X_{\owedge}(G_i)$ is defined by $X_{\owedge}(G_i)=A_{\owedge}(G_i)-\Ex{A_{\owedge}(G_i)|G_{i-1}}$, and so
\begin{align*}
\Ex{X_{\owedge}(G_i)^2\, \big|\,G_{i-1}}\,& = \, \Var(A_{\owedge}(G_i)\, \big|\,G_{i-1}) \phantom{\bigg)}\\
& \le \, \Ex{\left(A_{\owedge}(G_i)\, -\, \frac{8(i-1)}{n}\right)^2\, \Bigg|\,G_{i-1}} \phantom{\bigg)}\, .
\end{align*}
Now, we recall from~\eqr{Awedgeis} that $A_{\owedge}(G_i)= 8(i-1)/n\, +\, 2\big(D_u(G_{i-1})+D_w(G_{i-1})\big)$ where $uw$ is the $i$th edge.  It follows that, on the event $E_1^c$,
\begin{align*}
\Ex{X_{\owedge}(G_i)^2\, \big|\, G_{i-1}}\, & \le \, \frac{1}{N-i+1}\sum_{uw\not\in E(G_{i-1})} 4\big(D_u(G_{i-1})+D_w(G_{i-1})\big)^2 \phantom{\bigg)}\\
& \le \, \frac{8}{N-i+1}\sum_{uw\not\in E(G_{i-1})} \big(D_u(G_{i-1})^2+D_w(G_{i-1})^2\big) \phantom{\bigg)}\\
& \le \, \frac{16(n-1)}{N}\sum_{u}D_u(G_{i-1})^2 \phantom{\bigg)}\\
& \le\, 32C'\, n^{1/2}\max\{\eta,n^{1/2}\}\, , \phantom{\bigg)}
\end{align*} 
For $C\ge 32C'$ it follows that the event
\[
\Ex{X_{\owedge}(G_i)^2\, \big|\,G_{i-1}}\, >\, Cn^{1/2}\max\{\eta,n^{1/2}\}
\] 
is contained in $E_1$, as required.

We recall from~\eqr{Atriis} that $A_{\triangle}(G_i)=6(n-2)(i-1)_2/(N)_2+ 6D_{u,w}(G_{i-1})$.  A calculation as above, using that
\begin{align*}
\Ex{X_{\triangle}(G_i)^2\,\big|\,G_{i-1}}\,& = \, \Var(A_{\triangle}(G_i)\, \big|\,G_{i-1}) \phantom{\bigg)}\\
& \le \, \Ex{\left(A_{\triangle}(G_i)\, -\, \frac{6(n-2)(i-1)_2}{(N)_2}\right)^2\, \Bigg|\,G_{i-1}} \phantom{\bigg)}\\
& =\, \frac{1}{N-i+1}\sum_{uw\not\in E(G_{i-1})} \big(6D_{u,w}(G_{i-1})\big)^2\, ,
\end{align*}
shows that the event~\eqr{Duw2} fails is contained in $E_2$, provided $C\ge 160C'$, completing the proof.
\end{proof}

We now present a proof of Proposition~\ref{prop:forH}

\begin{proof}[Proof of Proposition~\ref{prop:forH}]
The proof is obtained by an application of Freedman's inequality, Lemma~\ref{lem:F}, to
\[
\Lambda^*_H(G_{n,t})\, =\, \sum_{i=1}^{m}\X_H(G_i;t)
\]
where $m=\lfloor tN\rfloor$ and
\[
\X_{H}(G_i;t)\, =\, n^{v-3}t^{e-3}\left(t\obinom{H}{\owedge}\frac{(1-t)^2}{(1-s)^2}X_{\owedge}(G_i)\, +\, \obinom{H}{\triangle} \frac{(1-t)^3}{(1-s)^{3}}\, \big(X_{\triangle}(G_i)-3sX_{\owedge}(G_i)\big) \right)\, .
\]
Since the co-efficients of $X_{\owedge}(G_i)$ and $X_{\triangle}(G_i)$ are at most $Cn^{v-3}$ in absolute value, for some constant $C=C(H)$, we have
\[
\Ex{\X_H(G_i;t)^2\, \big|\,G_{i-1}}\, \le \, 2C^2 n^{2v-6}\, \Ex{X_{\owedge}(G_i)^2\, \big|\, G_{i-1}}\, +\, 2C^2n^{2v-6}\, \Ex{X_{\triangle}(G_i)^2\, \big|\, G_{i-1}}\, .
\]
Writing $E_{\mathrm{var}}(i-1)$ for the event that 
\eq{varbigger}
\Ex{\X_H(G_i;t)^2\, \big|\,G_{i-1}}\, >\, 4C_1 C^2 n^{2v-11/2}\max\{\eta,n^{1/2}\}\, ,
\eqe
where $C_1$ is taken to be the constant of Lemma~\ref{lem:basicvar}, it follows, from Lemma~\ref{lem:basicvar}, that $\pr{E_{\mathrm{var}}(i-1)}\le \exp(-\eta n^{1/2})$.
Now let $E_{\mathrm{var}}$ be the event
\[
\sum_{i=1}^{m}\Ex{\X_H(G_i;t)^2\, \big|\,G_{i-1}}\, > \, 4C_2^3m n^{2v-11/2}\max\{\eta,n^{1/2}\}\, ,
\]
where $C_2=\max\{C,C_1\}$.
By a union bound, we have
\[
\pr{E_{\mathrm{var}}} \, \le\, n^2\exp(-\eta n^{1/2})\,\le \,\exp(-\eta n^{1/2}/2)\, .
\]
One may also note that, since $|X_{\owedge}(G_i)|,|X_{\triangle}(G_i)|\le n$, we have
\[
\big|\X_H(G_i;t)\big|\, \le\, 2Cn^{v-2}\qquad \text{almost surely.}
\]
We now apply Freedman's inequality, Lemma~\ref{lem:F}, with $\alpha=\eta t^{1/2} n^{v-3/2}$, with $\beta=4C_2^3 t n^{2v-7/2}\max\{\eta,n^{1/2}\}$ and with $R=2Cn^{v-2}$.  We obtain
\begin{align*}
& \pr{\Lambda^{*}_H(G_{n,t})\, >\, \eta n^{v-3/2}}\, \\
& \qquad \le\, \exp\left(\frac{-\eta^2 t n^{2v-3}}{8C_2^3 t n^{2v-7/2}\max\{\eta,n^{1/2}\}\, +\, 4C\eta n^{2v-7/2}}\right)\, +\, \pr{E_{\mathrm{var}}}\\ 
& \qquad \le \, \exp\left(\frac{-\eta\min\{\eta,n^{1/2}\}}{12C_2^3}\right)\, +\, \exp(-\eta n^{1/2}/2)\\
& \qquad \le\, 2\exp\big(-4c\eta\min\{\eta,n^{1/2}\}\big)\, ,
\end{align*}
where $c$ was chosen to be at most $1/24C_2^3$.  Since an identical argument applies to bound the probability that $\Lambda^{*}_H(G_{n,t})\, <\, -\eta n^{v-3/2}$, we have
\[
\pr{\big|\Lambda^{*}_H(G_{n,t})\big|\, >\, \eta n^{v-3/2}}\, \le\, 4\exp\big(-4c\eta\min\{\eta,n^{1/2}\}\big)\, \le\, \exp\big(-c\eta\min\{\eta,n^{1/2}\}\big)\, .
\]
This completes the proof.
\end{proof}

\section{Variance and covariance of the increments $X_F(G_i)$}\label{sec:var}

The aim of this section is to prove that the conditional variance $\Var(X_F(G_i)|G_{i-1})$ of $X_F(i)$ and the conditional covariance
\[
\Ex{X_F(G_i)X_{F'}(G_i)\, \big|\,G_{i-1}}\, ,
\] 
of $X_F(G_i)$ and $X_{F'}(G_i)$, are very predictable, in the sense that they are generally close to certain deterministic functions.  Given two graphs, $F$ with $v$ vertices and $e$ edges, and $F'$ with $v'$ vertices and $e'$ edges, let us define 
\eq{Vdef}
V_{F,F'}(i,n)\, :=\, n^{v+v'-5}s^{e+e'-4}(1-s)\big(s\theta_1(F,F')+(1-s)\theta_2(F,F')\big)\, ,
\eqe
where
\eq{Tdef}
\theta_1(F,F')\, :=\, 8\obinom{F}{\owedge}\obinom{F'}{\owedge}\qquad \text{and}\qquad \theta_2(F,F')\, :=\, 36\obinom{F}{\triangle}\obinom{F'}{\triangle}\, .
\eqe

\begin{prop}\label{prop:var} Let $F,F'$ be graphs with $v,v'$ vertices (respectively) and $e,e'$ edges (respectively) and let $t\in (0,1)$.  There is a constant $C=C(F,F',t)$ such that, for all $1\le i\le tN$ and all $3\log{n} \le b\le n/2C$, we have
\[
\pr{\big|\Ex{X_F(G_i)X_{F'}(G_i)\, \big|\, G_{i-1}}\, -\, V_{F,F'}(i,n)\big|\, >\, Cb^{1/2}n^{v+v'-11/2}}\, \le\, \exp(-b)\, .
\]
\end{prop}

Since the path of length two and the triangle play a particularly important role (see Theorem~\ref{thm:approx}, for example), it is perhaps of interest to note that in these cases we have
\begin{align*}
V_{\owedge,\owedge}(i,n)\, & =\, 8ns(1-s)\, ,\phantom{\Big)} \\
V_{\owedge,\triangle}(i,n)\, & =\, 24ns^2(1-s) \qquad \text{and}\phantom{\Big)}\\
V_{\triangle,\triangle}(i,n)\, & =\, 36ns^2 (1-s^2)\, .\phantom{\Big)}
\end{align*}
We may also now explain why we chose to express the terms of $\Lambda^{*}_H(G_{n,t})$ as multiples of $X_{\owedge}(G_i)$ and $X_{\triangle}(G_i)-3sX_{\owedge}(G_i)$.  This representation was chosen because these increments are asymptotically orthogonal in the sense that 
\[
\Ex{X_{\owedge}(G_i)\big(X_{\triangle}(G_i)-3sX_{\owedge}(G_i)\big)\, \big|\, G_{i-1}}=o(n)
\]
with high probability.  This follows from Proposition~\ref{prop:var} and the fact that
\[
V_{\owedge,\triangle}(i,n)\, -\, 3sV_{\owedge,\owedge}(i,n)\, =\, 0\, .
\]

We in fact prove an even more precise result, Proposition~\ref{prop:varbetter}, which includes a second order term related to the current deviation $D_{\owedge}(G_{i-1})$.  We define
\eq{Wdef}
W_{F,F'}(G_{i-1})\, :=\, 8n^{v+v'-7} s^{e+e'-4}\obinom{F}{\owedge}\obinom{F'}{\owedge} D_{\owedge}(G_{i-1})\, .
\eqe

\begin{prop}\label{prop:varbetter} Let $F,F'$ be graphs with $v,v'$ vertices (respectively) and $e,e'$ edges (respectively) and let $t\in (0,1)$.  There is a constant $C=C(F,F',t)$ such that, for all $1\le i\le tN$ and all $3\log{n} \le b\le n/2C$, we have
\[
\pr{\big|\Ex{X_F(G_i)X_{F'}(G_i)\, \big|\,G_{i-1}}\, -\, \big(V_{F,F'}(i,n)+W_{F,F'}(G_{i-1})\big)\big|\, >\, Cbn^{v+v'-6}}\, \le\, \exp(-b)\, .
\]
\end{prop}

The term $W_{F,F'}(G_{i-1})$ is generally much smaller than the main term $V_{F,F'}(i,n)$.  This follows from the fact that $D_{\owedge}(G_{i-1})$ is generally much smaller than $n^2$, which follows from Theorem~\ref{thm:upto}.  

%

Let us observe that Proposition~\ref{prop:var} follows from Proposition~\ref{prop:varbetter} and Theorem~\ref{thm:upto}.  

\begin{proof}[Proof of Proposition~\ref{prop:var}] Let $F,F'$ and $t$ be fixed.  Writing $C_1$ for the constant of Proposition~\ref{prop:varbetter}, and $c$ for the constant associated with $H=\owedge$ in Theorem~\ref{thm:upto}, we define 
\[
C\, =\, 2\left(C_1\,+\,8c^{-1}\obinom{F}{\owedge}\obinom{F'}{\owedge}\right)\, .
\]
By the triangle inequality,
\begin{align*}
&\Big|\Ex{X_F(G_i)X_{F'}(G_i)\, \big|\, G_{i-1}}\, -\, V_{F,F'}(i,n)\Big|\,  \\
& \qquad \le\,  \Big|\Ex{X_F(G_i)X_{F'}(G_i)\, \big|\,G_{i-1}}\, -\, \big(V_{F,F'}(i,n)+W_{F,F'}(G_{i-1})\big)\Big|\, +\, |W_{F,F'}(i,n)|\, ,
\end{align*}
and so the event
\eq{varstp}
\Big|\Ex{X_F(G_i)X_{F'}(G_i)\, \big|\, G_{i-1}}\, -\, V_{F,F'}(i,n)\Big|\, >\, Cb^{1/2}n^{v+v'-11/2}
\eqe
may only occur if
\begin{align*}
\Big|\Ex{X_F(G_i)X_{F'}(G_i)\, \big|\,G_{i-1}}\, -\, \big(V_{F,F'}(i,n)+W_{F,F'}(G_{i-1})\big)\Big|\, & >\, 2C_1 b^{1/2}n^{v+v'-11/2}\\ & \ge\, 2C_1 bn^{v+v'-6}
\end{align*}
or 
\[
\big|D_{\owedge}(G_{i-1})\big|\, >\, 2c^{-1} b^{1/2}n^{3/2}\, .
\]
By Proposition~\ref{prop:varbetter} and Theorem~\ref{thm:upto} respectively these events each have probability at most $\exp(-2b)$.  By a union bound, the event~\eqr{varstp} has probability at most $2\exp(-2b)\le \exp(-b)$, as required.
\end{proof}

We now prove Proposition~\ref{prop:varbetter}.

\begin{proof}[Proof of Proposition~\ref{prop:varbetter}]  Let $F,F'$ and $t$ be fixed.  Let $i\le tN$.
Since
\[
X_F(G_i)\, :=\, A_F(G_i)\, -\, \Ex{A_F(G_i)\, \big|\,G_{i-1}}\, ,
\]
the conditional covariance may be expressed as
\begin{align}\label{eq:covaris}
\Ex{X_F(G_i)X_{F'}(G_i)\, \big|\,G_{i-1}}\,  \quad & \nonumber \\ \phantom{\Bigg{(}}
= \, \Ex{A_F(G_i)A_{F'}(G_i)\, \big|\,G_{i-1}}\, & -\, \Ex{A_F(G_i)\, \big|\,G_{i-1}}\Ex{A_{F'}(G_i)\, \big|\,G_{i-1}}\, .
\end{align}

The proof consists of two main stages.  In the first we express each of the terms of~\eqr{covaris} as a linear combination of terms of the form $L_H(i)$ and $D_{H}(G_i)$.  This first stage results in the expressions~\eqr{expA} and~\eqr{expAA}, which we combine to obtain~\eqr{expX}, an expression for $\Ex{X_F(G_i)X_{F'}(G_i)|G_{i-1}}$ in terms of $L_H(i)$ and $D_{H}(G_i)$.  In the second stage we expand these terms and, after some calculation, arrive at the conclusion.

We begin the first stage by calculating an expression\footnote{This expression is in fact already given by Lemma~\ref{lem:expW}, however reproving it is a useful step towards the more difficult challenge of expressing $\Ex{A_F(G_i)A_{F'}(G_i)|G_{i-1}}$ in the desired form.} for $\Ex{A_{F}(G_i)|G_{i-1}}$.  Let $\Phi_F$ be the set of injective functions $\phi:V(F)\to V(G_{i-1})$.  We write $\phi(e)$ for the image of an edge, i.e., $\phi(uw)=\{\phi(u),\phi(w)\}$.  For $f\in E(F)$, we say that $\phi$ is \emph{$f$-ready} if $\phi(e)$ is an edge of $G_{i-1}$ for every edge of $F\setminus f$ and $\phi(f)$ is a non-edge of $G_{i-1}$.  For $f\in E(F)$, we define
\[
\Phi_F^{f}\, :=\, \{\phi\in \Phi_F\, :\, \text{$\phi$ is $f$-ready}\}\, .
\]
Since a copy of $F$ may only be created at the moment we add its final edge, and each of the $N-i+1$ remaining pairs is added as the $i$th edge with probability $1/(N-i+1)$, we have
\eq{Expassum}
\Ex{A_{F}(G_i)\, \big|\,G_{i-1}}\, =\,\sum_{f\in E(F)}\frac{|\Phi_F^{f}|}{N-i+1}\, .
\eqe
Observe that $|\Phi_F^{f}|$, the number of $f$-ready injective functions $\phi$, is given by
\[
|\Phi_F^{f}|\, =\, N_{F\setminus f}(G_{i-1})\, -\, N_{F}(G_{i-1})\, .
\]
Substituting this into~\eqr{Expassum}, and expanding each $N_H(G_{i-1})$ as $L_H(i-1)+D_{H}(G_{i-1})$, we obtain
\begin{align*}
& \Ex{A_{F}(G_i)\, \big|\,G_{i-1}}\, \\
& =\, \sum_{f\in E(F)}\frac{L_{F\setminus f}(i-1)-L_{F}(i-1)}{N-i+1}\, +\, \frac{1}{N-i+1}\sum_{f\in E(F)}\big( D_{F\setminus f}(G_{i-1})-D_{F}(G_{i-1})\big)\, .
\end{align*}
Since it is easily checked that
\[
\sum_{f\in E(F)}\frac{L_{F\setminus f}(i-1)-L_{F}(i-1)}{N-i+1}\, =\, L_{F}(i)\, -\, L_{F}(i-1)\, ,
\]
we obtain
\eq{expA}
\Ex{A_{F}(G_i)\, \big|\,G_{i-1}}\, =\, \big(L_{F}(i)-L_F(i-1)\big)\, +\, \frac{1}{N-i+1}\sum_{f\in E(F)} \big(D_{F\setminus f}(G_{i-1})-D_{F}(G_{i-1})\big)\, .
\eqe

We continue the first stage by calculating the expression~\eqr{expAA} for $\Ex{A_{F}(G_i)A_{F'}(G_i) |G_{i-1}}$.  To abbreviate the notation we set
\[
\exff\, :=\, \Ex{A_{F}(G_i)A_{F'}(G_i)\, \big|\,G_{i-1}}\, .
\]
Let $\Phi_{F,F'}=\Phi_F\times \Phi_{F'}$, be the set of pairs $(\phi,\phi')$ of injective functions $\phi:V(F)\to V(G_{i-1})$ and $\phi':V(F')\to V(G_{i-1})$.  We say that such a pair $(\phi,\phi')$ is \emph{$(f,f',\vect)$-ready}, for edges $f\in E(F)$, $f'\in E(F')$ and a relative orientation $\vect$ of the edges $f$ and $f'$, if 
\begin{enumerate}
\item[(i)] $\phi(e)\in E(G_{i-1})$ for all $e\in F\setminus f$,
\item[(ii)] $\phi'(e)\in E(G_{i-1})$ for all $e\in F'\setminus f'$ in $G_{i-1}$, and 
\item[(iii)] $\phi(f)$ and $\phi'(f')$ map to the same non-edge of $G_{i-1}$, and have relative orientation $\vect$.  
\end{enumerate}
For $f\in E(F),f'\in E(F')$ and a relative orientation $\vect$, we define
\[
\Phi_{F,F'}^{f,f',\vect}\, :=\, \{(\phi,\phi')\in \Phi_{F,F'}\, :\, (\phi,\phi')\, \text{is $(f,f',\vect)$-ready}\}\, .
\]
Since embeddings of $F$ and $F'$ may only be simultaneously created at the moment we add their final edge, and each of the $N-i+1$ remaining pairs is added as the $i$th edge with probability $1/(N-i+1)$, we have
\[
\exff\, =\,\sum_{f,f',\vect}\frac{|\Phi_{F,F'}^{f,f',\vect}|}{N-i+1}\, .
\]
Observe that 
\[
|\Phi_{F,F'}^{f,f',\vect}|\, ,
\] 
the number of $(f,f',\vect)$-ready pairs $(\phi,\phi')$ includes a count over those pairs $(\phi,\phi')$ whose images overlap in exactly two vertices, those whose images overlap in three vertices, and those pairs that overlap in four or more vertices.  The count of pairs with overlap exactly two vertices is 
\[
N_{\Gamma^{o}(F,F':f,f',\vect)}(G_{i-1})\, -\, N_{\Gamma(F,F':f,f',\vect)}(G_{i-1})\, .
\]
where $\Gamma(F,F':f,f',\vect)$ is the graph obtained by joining $F$ and $F'$ by identifying $f$ and $f'$ using the relative orientation $\vect$, and $\Gamma^{o}(F,F':f,f',\vect)$ is obtained by then removing the identified edge.  The count of pairs with overlap exactly three vertices is
\[
\sum_{u\in V(F)\setminus f, u'\in V(F')\setminus f'} \big( N_{\Gamma^{o}(F,F':f,f',u,u',\vect)}(G_{i-1})\, -\, N_{\Gamma(F,F':f,f',u,u',\vect)}(G_{i-1}) \big)\, 
\]
where $\Gamma(F,F':f,f',u,u',\vect)$ is the graph obtained by joining $F$ and $F'$ by identifying $f$ and $f'$ using the relative orientation $\vect$ and also identifying $u$ and $u'$, and $\Gamma^{o}(F,F':f,f',u,u',\vect)$ is obtained by then removing the identified edge.
It follows that
\begin{align*}
\exff \, & = \frac{1}{N-i+1} \sum_{f,f',\vect}\big(N_{\Gamma^{o}(F,F':f,f',\vect)}(G_{i-1})\, -\, N_{\Gamma(F,F':f,f',\vect)}(G_{i-1})\big)\\ 
& +\, \frac{1}{N-i+1}\sum_{f,f',u,u',\vect}\big(N_{\Gamma^{o}(F,F':f,f',u,u',\vect)}(G_{i-1})\, -\, N_{\Gamma(F,F':f,f',u,u',\vect)}(G_{i-1})\big)\\ 
& +\, O(n^{v+v'-6})\, ,
\end{align*}
where the error term $O(n^{v+v'-6})$ comes from the fact that there are $O(n^{v+v'-4})$ pairs that overlap in four or more vertices, and $N-i+1\ge N-tN+1 =\Theta(n^2)$, as $t\in (0,1)$ is fixed.
Expanding the terms $N_H(G_{i-1})$ we obtain the desired expression for $\Ex{A_{F}(G_i)A_{F'}(G_i)|G_{i-1}}$:


\begin{align}\label{eq:expAA}
& \exff \,  =\, \frac{1}{N-i+1} \sum_{f,f',\vect}\big(L_{\Gamma^{o}(F,F':f,f',\vect)}(i-1)\,- \, L_{\Gamma(F,F':f,f',\vect)}(i-1)\big)\nonumber  \\ &\quad \,+\, \frac{1}{N-i+1}\sum_{f,f',u,u',\vect}\big(L_{\Gamma^{o}(F,F':f,f',u,u',\vect)}(i-1)\, -\, L_{\Gamma(F,F':f,f',u,u',\vect)}(i-1)\big)\nonumber \\
&\quad + \,\frac{1}{N-i+1} \sum_{f,f',\vect}\big(D_{\Gamma^{o}(F,F':f,f',\vect)}(G_{i-1})\, -\, D_{\Gamma(F,F':f,f',\vect)}(G_{i-1})\big)\nonumber \\ 
&\quad +\, \frac{1}{N-i+1}\sum_{f,f',u,u',\vect}\big(D_{\Gamma^{o}(F,F':f,f',u,u',\vect)}(G_{i-1})\, -\, D_{\Gamma(F,F':f,f',u,u',\vect)}(G_{i-1})\big)\nonumber \\ 
&\quad  +\, O(n^{v+v'-6})\,  .
\end{align}

Combining~\eqr{expA} and~\eqr{expAA} and substituting into~\eqr{covaris}, we obtain the following expression for $\Ex{X_F(G_i)X_{F'}(G_i)|G_{i-1}}$:
\begin{align}\label{eq:expX}
& \Ex{X_F(G_i)X_{F'}(G_i)\, \big|\,G_{i-1}} \nonumber \\
& \, =\, \frac{1}{N-i+1} \sum_{f,f',\vect}\big(L_{\Gamma^{o}(F,F':f,f',\vect)}(i-1)\, -\, L_{\Gamma(F,F':f,f',\vect)}(i-1)\big) \nonumber \\ &\quad +\, \frac{1}{N-i+1}\sum_{f,f',u,u',\vect}\big(L_{\Gamma^{o}(F,F':f,f',u,u',\vect)}(i-1)\, -\, L_{\Gamma(F,F':f,f',u,u',\vect)}(i-1)\big)\nonumber \\
&\quad + \,\frac{1}{N-i+1} \sum_{f,f',\vect}\big(D_{\Gamma^{o}(F,F':f,f',\vect)}(G_{i-1})\, -\, D_{\Gamma(F,F':f,f',\vect)}(G_{i-1})\big)\nonumber \\ &\quad +\, \frac{1}{N-i+1}\sum_{f,f',u,u',\vect}\big(D_{\Gamma^{o}(F,F':f,f',u,u',\vect)}(G_{i-1})\, -\, D_{\Gamma(F,F':f,f',u,u',\vect)}(G_{i-1})\big)\nonumber \\
& \quad - \left(\big(L_{F}(i)-L_F(i-1)\big)\, +\, \frac{1}{N-i+1}\sum_{f\in E(F)} \big(D_{F\setminus f}(G_{i-1})-D_{F}(G_{i-1})\big)\right)\nonumber \\
& \quad \times \left(\big(L_{F'}(i)-L_{F'}(i-1)\big)\, +\, \frac{1}{N-i+1}\sum_{f'\in E(F')} \big(D_{F'\setminus f'}(G_{i-1})-D_{F'}(G_{i-1})\big)\right)\nonumber \\
 & \quad +\, O(n^{v+v'-6})\, .
\end{align}

We now begin the second stage of the proof.  Essentially we must understand the terms in~\eqr{expX}, and calculate what remains after cancellations.  Our hope is that all the terms involving deviations reduce to $W_{F,F'}(G_{i-1})$, up to a small error term.  To prove this we use Theorem~\ref{thm:relate} and Theorem~\ref{thm:upto}.  

By Theorem~\ref{thm:relate} and Theorem~\ref{thm:upto} there exists, for each $v^{*}$, a constant $C_1=C_1(v^{*})$ such that for all $1\le b\le n/2C_1$ and all graphs $H$ on at most $v^*$ vertices, each of the events
\eq{devml}
\big|D_{H}(G_{n,s})-\Lambda_{H}(G_{n,s})\big| \, \le \, C_1 bn^{v(H)-2}
\eqe
and
\eq{justdev}
\big|D_H(G_{n,s})\big|\, \le \, C_1 b^{1/2}n^{v(H)-3/2}
\eqe
fail with probability at most $\exp(-((v^{*})^2+2)b)$.  Let $E_{v^{*}}(b)$ be the event that both of~\eqr{devml},~\eqr{justdev} hold for any graph $H$ on at most $v^*$ vertices.  By a straightforward union bound over the (at most $2^{(v^*)^2}$) graphs $H$ on at most $v^{*}$ vertices.  We have that
\[
\pr{E_{v^*}(b)}\, \ge \, 1\, -\, \exp(-b)\, .
\]


To complete the proof it suffices to prove that there is a constant $C$ such that, on the event $E_{v+v'}(b)$, we have
\[
\big|\Ex{X_F(G_i)X_{F'}(G_i)\,\big|\,G_{i-1}}\, -\, \big(V_{F,F'}(i,n)+W_{F,F'}(G_{i-1})\big)\big|\, \le \, Cbn^{v+v'-6}\, .
\]

Let us continue our calculation of $\Ex{X_F(G_i)X_{F'}(G_i)|G_{i-1}}$ by expanding and cancelling the terms of~\eqr{expX}.  We begin by calculating the total contribution~\eqr{Lcontrib} of the terms involving only the $L_H(i-1)$.  
Using that both $\Gamma^{o}(F,F':f,f',\vect)$ and $\Gamma(F,F':f,f',\vect)$ have $v+v'-2$ vertices, and that $e(\Gamma^{o}(F,F':f,f',\vect))=e+e'-2$ and $e(\Gamma(F,F':f,f',\vect))=e+e'-1$, it is easily verified that
\[
L_{\Gamma^{o}(F,F':f,f',\vect)}(i-1)\, -\, L_{\Gamma(F,F':f,f',\vect)}(i-1)\, =\, \frac{(n)_{v+v'-2}(i-1)_{e+e'-2}(N-i+1)}{(N)_{e+e'-1}}\, ,
\]
for every choice of $f,f',\vect$.  It follows that
\begin{align}\label{eq:contL1}
& \frac{1}{N-i+1}  \sum_{f,f',\vect}\big(L_{\Gamma^{o}(F,F':f,f',\vect)}(i-1)\, -\, L_{\Gamma(F,F':f,f',\vect)}(i-1)\big)\phantom{\Bigg(} \\ & 
=\,  \frac{2e e' (n)_{v+v'-2}(i-1)_{e+e'-2}}{(N)_{e+e'-1}}\phantom{\Bigg(} \nonumber \\
& =\, \frac{2 e e' (n)_{v+v'-2}s^{e+e'-2}}{N}\, +\, O(n^{v+v'-6}) \phantom{\Bigg(} \nonumber \\
& =\, 4e e' n^{-2} \left(1+\frac{1}{n}\right) \left(n^{v+v'-2} -\obinom{v+v'-2}{2} n^{v+v'-3}\right)s^{e+e'-2} \, +\, O(n^{v+v'-6})\phantom{\Bigg(} \nonumber \\
& =\, 4e e'n^{v+v'-4}s^{e+e'-2} + 2e e'\big(2-(v+v'-2)(v+v'-3)\big) n^{v+v'-5}s^{e+e'-2} + O(n^{v+v'-6})\, .\phantom{\Bigg(}\nonumber
\end{align}
The main negative term comes from the product
\[
-\big(L_{F}(i)-L_F(i-1)\big)\big(L_{F'}(i)-L_{F'}(i-1)\big)\, .
\]
Since
\[
L_{F}(i)-L_F(i-1)\, =\, \frac{e(n)_v(i-1)_{e-1}}{(N)_{e}}\, =\, 2en^{-2} \left(1+\frac{1}{n}\right) (n)_v s^{e-1}\, +\, O(n^{v-4})\, ,
\]
this main negative term is
\[
-4ee' n^{-4}\left(1+\frac{2}{n}\right)(n)_v(n)_{v'}s^{e+e'-2}\, +\, O(n^{v+v'-6})
\]
which may be expressed as
\eq{contL2}
-4e e'n^{v+v'-4}s^{e+e'-2}\, +\, 2 e e'\big(v(v-1)+v'(v'-1)-4\big) n^{v+v'-5}s^{e+e'-2}\, +\, O(n^{v+v'-6}).
\eqe
The final contribution from terms purely involving the terms $L_H(i-1)$ is
\[
\frac{1}{N-i+1}\sum_{f,f',u,u',\vect}\big(L_{\Gamma^{o}(F,F':f,f',u,u',\vect)}(i-1)\, -\, L_{\Gamma(F,F':f,f',u,u',\vect)}(i-1)\big)\, .
\]
The value of the summand depends on the number of extra overlaps of edges that occur in the identification.  For $j=0,1,2$, let $\lambda_j$ be the number of sequences $f,f',u,u',\vect$ in which $\Gamma^{o}(F,F':f,f',u,u',\vect)$ has $e+e'-2-j$ edges, meaning that $j$ edges other than $f$ and $f'$ are lost in the identification.  The contribution of 
\[
\frac{1}{N-i+1}\big(L_{\Gamma^{o}(F,F':f,f',u,u',\vect)}(i-1)\, -\, L_{\Gamma(F,F':f,f',u,u',\vect)}(i-1)\big)
\]
is 
\[
2n^{v+v'-5} s^{e+e'-2-j} \, +\, O(n^{v+v'-6})\, 
\]
in the case of $j$ extra edges being lost in the identification.  It follows that the total contribution of these terms is
\eq{contL3}
2n^{v+v'-5}\big(\lambda_0 s^{e+e'-2}+\lambda_1 s^{e+e'-3}+\lambda_2 s^{e+e'-4}\big)\, +\, O(n^{v+v'-6})\, .
\eqe
Summing all contributions to~\eqr{expX} from terms involving only the $L_H(i-1)$, i.e., summing~\eqr{contL1},~\eqr{contL2} and~\eqr{contL3}, we obtain
\[
2n^{v+v'-5}\left(\big(\lambda_0-2e e'(v-2)(v'-2)\big) s^{e+e'-2}+\lambda_1 s^{e+e'-3}+\lambda_2 s^{e+e'-4}\right)\, +\, O(n^{v+v'-6})\, .
\]
Using that $\lambda_0+\lambda_1+\lambda_2 = 2e e'(v-2)(v'-2)$, this total contribution is
\begin{align*}
&2n^{v+v'-5}\left(\big(-\lambda_1-\lambda_2\big) s^{e+e'-2}+\lambda_1 s^{e+e'-3}+\lambda_2 s^{e+e'-4}\right)\, +\, O(n^{v+v'-6}) \phantom{\bigg)}\\
&\, =\, 2n^{v+v'-5}\big(\lambda_1 s^{e+e'-3}(1-s)\, +\,\lambda_2 s^{e+e'-4}(1-s^2)\big)\, +\, O(n^{v+v'-6})\,. \phantom{\bigg)}
\end{align*}
We may now relate $\lambda_1$ and $\lambda_2$ to the parameters $\theta_1(F,F')$ and $\theta_2(F,F')$ that occur in the definition of $V_{F,F'}(i,n)$.

\noindent\textbf{Claim:} We have $2\lambda_1=\theta_1(F,F')-2\theta_2(F,F')$ and $2\lambda_2=\theta_2(F,F')$. 

\noindent\textbf{Proof of Claim:} Let $\rho_1$ be the number of pairs of an edge $f$ of $F$ and a disjoint vertex $u$ such there is precisely one edge between the endpoints of $f$ and $u$.  Let $\rho_2$ be the number of such pairs in which both possible edges are present, i.e., $f\cup\{u\}$ is a triangle in $F$, and let $\rho'_1$ and $\rho'_2$ be the equivalent quantities in $F'$.  It is easily verified that
\[
\rho_1\, =\, 2\obinom{F}{\owedge}-6\obinom{F}{\triangle}\qquad \text{and}\qquad \rho_2\, =\, 3\obinom{F}{\triangle}\, .
\]

Let us now prove that $2\lambda_2=\theta_2(F,F')$.  We recall that $\lambda_2$ counts the number of choices $f,f',u,u',\vect$ such that the overlap contains two extra edges.  This occurs if and only if $f\cup\{u\}$ and $f'\cup\{u'\}$ are triangles in their respective graphs, and so $\lambda_2=2\rho_2\rho'_2$, where the factor of $2$ has come from counting the two possible orientations.  It follows that
\[
2\lambda_2\, =\, 4\rho_2\rho'_2\, =\, 36\obinom{F}{\triangle}\obinom{F'}{\triangle}\, =\, \theta_2(F,F')\, .
\]

We now turn to $\lambda_1$, which counts the number of choices $f,f',u,u',\vect$ such that the overlap contains exactly one extra edge.  This occurs for one of the two orientations if there is one edge between $f$ and $u$ and likewise between $f'$ and $u'$, and with both orientations if one of the two is a triangle.  Thus
\[
\lambda_1\, =\, \rho_1\rho'_1\, +\, 2\rho_1 \rho'_2\, +\, 2\rho_2\rho'_1\, .
\]
Substituting in the values of $\rho_1,\rho_2,\rho'_1,\rho'_2$ we obtain
\[
2\lambda_1\, =\, 2\big(2\obinom{F}{\owedge}-6\obinom{F}{\triangle}\big)\big(2\obinom{F'}{\owedge}-6\obinom{F'}{\triangle}\big)\, +\, 12 \big(2\obinom{F}{\owedge}-6\obinom{F}{\triangle}\big)\obinom{F'}{\triangle}\,  +\, 12 \big(2\obinom{F'}{\owedge}-6\obinom{F'}{\triangle}\big)\obinom{F}{\triangle}\, ,
\]
which is $\theta_1(F,F')- 2\theta_2(F,F')$, completing the proof of the claim. \vspace{0.3cm}

Using the claim, the total contribution of the terms involving only the $L_H(i-1)$ is
\begin{align}\label{eq:Lcontrib}
&n^{v+v'-5}s^{e+e'-4}\Big((s-s^2)\big(\theta_1(F,F')- 2\theta_2(F,F')\big)\, +\, (1-s^2)\theta_2(F,F')\Big)\, = \phantom{\bigg)} \nonumber \\
& \phantom{\bigg)}  n^{v+v'-5}s^{e+e'-4}(1-s)\Big(s\theta_1(F,F')\, +\, (1-s)\theta_2(F,F')\Big)\, =\, V_{F,F'}(i,n)\, . 
\end{align}

We now turn to terms involving deviations $D_{H}(G_{i-1})$.  On the event $E_{v+v'}(b)$, that both of~\eqr{devml} and~\eqr{justdev} hold for all graphs on at most $v+v'$ vertices, we have that the deviation $D_H(G_{i-1})$ is given by
\eq{Lambdais}
n^{v(H)-3}s^{e(H)-2}\obinom{H}{\owedge}  D_{\owedge}(G_{i-1})\, +\, n^{v(H)-3}s^{e(H)-3}\obinom{H}{\triangle} \big(D_{\triangle}(G_{i-1})-3s D_{\owedge}(G_{i-1})\big)\, \pm\, C_1b n^{v(H)-2} 
\eqe
for all graphs $H$ on at most $v+v'$ vertices, and so, in particular for any graph included in~\eqr{expX}.  Thus, we need only to determine the coefficients of $D_{\owedge}(G_{i-1})$ and $D_{\triangle}(G_{i-1})$ obtained after summing the terms of~\eqr{expX} that involve deviations.  We note that the terms with overlap at least $3$, in the sum over $f,f',u,u',\vect$ for example, are at most 
\[
C_1 b^{1/2}n^{v+v'-13/2}\, =\, O(n^{v+v'-6})
\] 
on $E_{v+v'}(b)$.  We may also safely ignore the term
\[
-\frac{1}{(N-i+1)^2}\sum_{f\in E(F)} \big(D_{F\setminus f}(G_{i-1})-D_{F}(G_{i-1})\big)\, \sum_{f'\in E(F')} \big(D_{F'\setminus f'}(G_{i-1})-D_{F'}(G_{i-1})\big)
\]
which has absolute value at most
\[
C_1^2 b n^{v+v'-7}\, =\, O(n^{v+v'-6})
\]
on $E_{v+v'}(b)$.  The remaining terms are
\begin{align*}
&\frac{1}{N-i+1} \sum_{f,f',\vect}\big(D_{\Gamma^{o}(F,F':f,f',\vect)}(G_{i-1})\, -\, D_{\Gamma(F,F':f,f',\vect)}(G_{i-1})\big)\\ 
& - \big(L_{F}(i)-L_F(i-1)\big)\,\frac{1}{N-i+1}\sum_{f'\in E(F')} \big(D_{F'\setminus f'}(G_{i-1})-D_{F'}(G_{i-1})\big)\\
& -\, \big(L_{F'}(i)-L_{F'}(i-1)\big) \frac{1}{N-i+1}\sum_{f\in E(F)} \big(D_{F\setminus f}(G_{i-1})-D_{F}(G_{i-1})\big)
\end{align*}
By expanding each $D_H(G_{i-1})$ using~\eqr{Lambdais} we will find an expression for the remaining terms of~\eqr{expX} as a combination
\[
\beta_1 D_{\owedge}(G_{i-1})\, +\, \beta_2 \big(D_{\triangle}(G_{i-1})-3s D_{\owedge}(G_{i-1})\big)\, +\, O(bn^{v+v'-6})
\]
on $E_{v+v'}(b)$.

Let us first calculate $\beta_1$.  Using that 
\[
L_{F}(i)-L_{F}(i-1)\, =\, \frac{e(n)_{v}(i-1)_{e-1}}{(N)_{e}}\, =\, 2en^{v-2}s^{e-1}\, +\, O(n^{v-3})
\]
and expanding each $D_H(G_{i-1})$ using~\eqr{Lambdais}, we find that 
\begin{align}\label{eq:betais}
(N-i+1)\beta_1\, &=\, n^{v+v'-5}s^{e+e'-4}\sum_{f,f',\vect} \left(\obinom{\Gamma^{o}(F,F':f,f',\vect)}{\owedge}\, -\, s\obinom{\Gamma(F,F':f,f',\vect)}{\owedge}\right)\nonumber \\
& \, -\, 2en^{v+v'-5}s^{e+e'-4}\sum_{f'\in E(F')} \left(\obinom{F'\setminus f'}{\owedge}\, -\, s\obinom{F'}{\owedge}\right)\nonumber \\
& \, -\, 2e'n^{v+v'-5}s^{e+e'-4}\sum_{f\in E(F)} \left(\obinom{F\setminus f}{\owedge}\, -\, s\obinom{F}{\owedge}\right)\, .
\end{align}
We may count the contribution of the first sum as follows, each $P_2$ in $F$ is counted $2(e-2)e'$ times by the first term and $2ee'$ times by the second, while $P_2$s in $F'$ are counted $2e(e'-2)$ and $2ee'$ times respectively.  The other way to find a $P_2$ in these graphs is crossing between $F$ and $F'$; a little thought shows that there are
\[
4\obinom{F}{\owedge}\obinom{F'}{\owedge}
\]
such contributions to each of the two terms.  Thus the result of the first sum is
\[
2ee'\left(\obinom{F}{\owedge}+\obinom{F'}{\owedge}\right)(1-s)\, +\, 4\obinom{F}{\owedge}\obinom{F'}{\owedge}(1-s)\, -\,2e'\obinom{F}{\owedge}\, -\,2e\obinom{F'}{\owedge}\, .
\]
The equivalent results for the second and third terms are
\[
e'\obinom{F'}{\owedge}(1-s)\, -\, 2\obinom{F'}{\owedge}
\]
and
\[
e\obinom{F}{\owedge}(1-s)\, -\, 2\obinom{F}{\owedge}
\]
respectively.  Substituting these values in~\eqr{betais} we obtain
\[
(N-i+1)\beta_1\, =\, 4n^{v+v'-5}\obinom{F}{\owedge}\obinom{F'}{\owedge} s^{e+e'-4} (1-s)\, ,
\]
and so
\[
\beta_1\, =\, 8n^{v+v'-7} s^{e+e'-4}  \obinom{F}{\owedge}\obinom{F'}{\owedge}\, +\, O(n^{v+v'-8})\, .
\]
This is consistent with our aim to prove that the contribution of terms involving deviations is given by $W_{F,F'}(G_{i-1})$ up to $O(bn^{v+v'-6})$.  All that remains is to prove that $\beta_2=0$.  That is, the terms which contribute a multiple of $D_{\triangle}(G_{i-1})-3s D_{\owedge}(G_{i-1})$ in the expansion cancel.  We have
\begin{align*}
\beta_2(N-i+1)\, &=\, n^{v+v'-5}s^{e+e'-5}\sum_{f,f',\vect} \left(\obinom{\Gamma^{o}(F,F':f,f',\vect)}{\triangle}\, -\, s\obinom{\Gamma(F,F':f,f',\vect)}{\triangle}\right)\nonumber \\
& \, -\, 2en^{v+v'-5}s^{e+e'-5}\sum_{f'\in E(F')} \left(\obinom{F'\setminus f'}{\triangle}\, -\, s\obinom{F'}{\triangle}\right)\nonumber \\
& \, -\, 2e'n^{v+v'-5}s^{e+e'-5}\sum_{f\in E(F)} \left(\obinom{F\setminus f}{\triangle}\, -\, s\obinom{F}{\triangle}\right)\, .
\end{align*}
The calculation is as above; however, since no triangles can cross between $F$ and $F'$, we obtain only terms that cancel.  The result is that $\beta_2=0$.  This confirms that the contribution of terms in~\eqr{expX} that involve deviations contribute
\[
W_{F,F'}(G_{i-1})\, +\, O(bn^{v+v'-6})
\]
on $E_{v+v'}(b)$.  Combining this with~\eqr{Lcontrib} we obtain 
\[
\big|\Ex{X_F(G_i)X_{F'}(G_i)\,\big|\,G_{i-1}}\, -\, \big(V_{F,F'}(i,n)+W_{F,F'}(G_{i-1})\big)\big|\, \le \, Cbn^{v+v'-6}
\]
for an appropriately chosen constant $C$ on $E_{v+v'}(b)$, an event with probability at least $1-\exp(-b)$, as required.
\end{proof}

\section{Probability of subgraph count deviations -- Theorem~\ref{thm:main}}\label{sec:main}

In this section we bring together the various threads and complete our proof of Theorem~\ref{thm:main}.  Armed with Theorem~\ref{thm:approxbetter} it will suffice to prove the analogous statement with $\Lambda^{*}_H(G_{n,t})$ in place of $D_H(G_{n,t})$.

\begin{prop}\label{prop:mainLambda} Let $t=t(n)\in (0,1)$ be a sequence bounded away from $1$, let $H$ be graph with $v$ vertices, $e$ edges, and $\binom{H}{\owedge}\ge 1$.    Then
\[
\pr{\Lambda^{*}_H(G_{n,t})\, >\, \alpha_n n^{v-3/2}}\, =\, \exp\big(-\gamma_H(t) \alpha_n^2 (1+o(1))\big)\, ,
\]
for every sequence $(\alpha_n:n\ge 1)$ with $t^{e-3/2}\ll \alpha_n\ll t^{e+2}n^{1/2}$.  Furthermore, the same holds for $\pr{\Lambda^{*}_H(G_{n,t})\, <\, -\alpha_n n^{v-3/2}}$.
\end{prop}

Here the expression $\gamma_H(t)$ is as defined in the introduction, namely:
\[
\gamma_H(t)\, :=\, \left(4\obinom{H}{\owedge}^2 t^{2e-2}(1-t)^2\, +\, 12\obinom{H}{\triangle}^2 t^{2e-3}(1-t)^3\right)^{-1}\, .
\]
Let us observe that indeed Theorem~\ref{thm:main} follows from Proposition~\ref{prop:mainLambda} and Theorem~\ref{thm:approxbetter}.

\begin{proof}[Proof of Theorem~\ref{thm:main}]
Let us fix $t=t(n)\in (0,1)$ and a graph $H$ with $v$ vertices, $e$ edges, and $\binom{H}{\owedge}\ge 1$.  Let the sequence $(\alpha_n:n\ge 1)$ with 
\[
\max\{t^{1/2}n^{-1/2}\log{n}, t^{e-3/2}\}\, \ll\, \alpha_n\, \ll\, \min\{t^{2e-5/2}n^{1/2},t^{e+2}n^{1/2}\}
\] 
be given.  Finally, let us also fix $\eps>0$.  We may suppose $\eps<1/10$.

We first show the upper bound on $\pr{D_H(G_{n,t})\, >\, \alpha_n n^{v-3/2}}$.
We begin by observing that
\begin{align*}
\pr{D_H(G_{n,t})\, >\, \alpha_n n^{v-3/2}}\, \le &\, \,  \pr{\Lambda^{*}_H(G_{n,t})\, >\, (1-\eps)\alpha_n n^{v-3/2}}\\ & \quad  +\, \pr{\big|D_H(G_{n,t})-\Lambda^{*}_H(G_{n,t})\big|>\eps \alpha_n n^{v-3/2}}\, .\phantom{\bigg)}
\end{align*}
Now, by Proposition~\ref{prop:mainLambda}, we have
\begin{align*}
\pr{\Lambda^{*}_H(G_{n,t})\, >\, (1-\eps)\alpha_n n^{v-3/2}}\, & \le\, \exp\big(-\gamma_H(t) \alpha_n^2 (1-\eps+o(1))^2\big)\\
& \le\, \exp\big(-\gamma_H(t) \alpha_n^2 (1-3\eps)\big)
\end{align*}
for all sufficiently large $n$.  On the other hand, we shall apply Theorem~\ref{thm:approxbetter} (as stated in Theorem~\ref{thm:approx}) with $b=\eps \alpha_n C^{-1}t^{-1/2}n^{1/2}$ to bound
\[
\pr{\big|D_H(G_{n,t})-\Lambda^{*}_H(G_{n,t})\big|>\eps \alpha_n n^{v-3/2}}\, .
\]
It is easily checked that the conditions on $\alpha_n$ ensure that $3\log{n}\le b\le t^{1/2} n$, for all $n$ sufficiently large.  By Theorem~\ref{thm:approxbetter} we have
\[
\pr{\big|D_H(G_{n,t})-\Lambda^{*}_H(G_{n,t})\big|>\eps \alpha_n n^{v-3/2}}\, \le\, \exp(-c\eps t^{-1/2} \alpha_n n^{1/2})
\]
for some constant $c>0$.  Since $\alpha_n\ll t^{2e-5/2}n^{1/2}$, we have that
\[
c\eps t^{-1/2} \alpha_n n^{1/2}\gg \gamma_H(t) \alpha_n^2 (1-3\eps)
\]
and so, for all sufficiently large $n$,
\[
\pr{D_H(G_{n,t})\, >\, \alpha_n n^{v-3/2}}\, \le\, (1+\eps)\exp\big(-\gamma_H(t) \alpha_n^2 (1-3\eps)\big)\, .
\]
Since $\eps$ is arbitrary, and $\gamma_H(t)\alpha_n^2\gg 1$, we have
\[
\pr{D_H(G_{n,t})\, >\, \alpha_n n^{v-3/2}}\, \le\, \exp\big(-\gamma_H(t) \alpha_n^2 (1-o(1))\big)\, .
\]
The proof of the lower bound follows immediately by the same argument, and the fact that
\begin{align*}
\pr{D_H(G_{n,t})\, >\, \alpha_n n^{v-3/2}}\, & \ge\, \pr{\Lambda^{*}_H(G_{n,t})\, >\, (1+\eps)\alpha_n n^{v-3/2}}\\ &\qquad -\, \pr{\big|D_H(G_{n,t})-\Lambda^{*}_H(G_{n,t})\big|>\eps \alpha_n n^{v-3/2}}\, .\phantom{\bigg)}
\end{align*}
This completes the proof.
\end{proof}

Our remaining task is to prove Proposition~\ref{prop:mainLambda}.  Let us recall that $\Lambda^{*}_H(G_{n,t})$ is the martingale expression
\[
\Lambda^{*}_H(G_{n,t})\, = \, \sum_{i=1}^{\lfloor tN\rfloor }\left(\kappa_{H,n}^{t}(i) X_{\owedge}(G_i)\, +\, \rho^{t}_{H,n}(i)\, \big(X_{\triangle}(G_i)-3sX_{\owedge}(G_i)\big) \right)\, ,
\]
where
\[
\kappa^{t}_{H,n}(i)\, :=\, n^{v-3}t^{e-2}\obinom{H}{\owedge}\frac{(1-t)^2}{(1-s)^2}
\]
and
\[
\rho^{t}_{H,n}(i)\, :=\, n^{v-3}t^{e-3}\obinom{H}{\triangle} \frac{(1-t)^3}{(1-s)^{3}}\, .
\]

We prove a general statement on the probability of deviations of martingales of this general form.  

\begin{prop}\label{prop:maingen} Let $t=t(n)\in (0,1)$ be a sequence bounded away from $1$, let ${\bf \kappa}=(\kappa^{t}_{n})_{n\ge 1}$ and ${\bf \rho}=(\rho^{t}_{n})_{n\ge 1}$ be two sequences of functions such that $\kappa^{t}_{n}:\{1,\dots ,\lfloor tN\rfloor \}\to [-C,C]$ for some constant $C\in \mathbb{R}$ and $\rho^{t}_{n}:\{1,\dots ,\lfloor tN\rfloor \}\to [-Ct^{-1},Ct^{-1}]$, and suppose there exists $\eta>0$ such that
\eq{average}
\sum_{i=1}^{\lfloor tN\rfloor}|\kappa^{t}_{n}(i)|+t|\rho^{t}_{n}(i)|\, \ge \, \eta t N
\eqe
for all sufficiently large $n$.
Then 
\[
S_n^t\,:=\, \sum_{i=1}^{\lfloor tN\rfloor }\left(\kappa_{n}^{t}(i) X_{\owedge}(G_i)\, +\, \rho^{t}_{n}(i)\, \big(X_{\triangle}(G_i)-3sX_{\owedge}(G_i)\big) \right)
\]
satisfies
\[
\pr{S_n^t\, >\, \alpha_n n^{3/2}}\, =\, \exp\left(\frac{-\alpha_n^2 (1+o(1))}{2\tau_{{\bf \kappa},{\bf \rho}}}\right)\, ,
\]
for every sequence $(\alpha_n:n\ge 1)$ with $t^{1/2} \ll \alpha_n\ll t^4 n^{1/2}$, where
\[
\tau_{{\bf \kappa},{\bf \rho}}\, :=\, n^{-2} \sum_{i=1}^{\lfloor tN\rfloor}\Big(8s(1-s)\kappa^t_n(i)^2\, +\, 36s^2(1-s)^2\rho_n^t(i)^2\Big) \, .
\]
Furthermore the same holds for $\pr{S_n^t\, <\, -\alpha_n n^{v-3/2}}$.
\end{prop}

Let us observe that indeed Proposition~\ref{prop:mainLambda} follows from Proposition~\ref{prop:maingen}.

\begin{proof}[Proof of Proposition~\ref{prop:mainLambda}]
Let ${\bf \kappa}$ be the sequence of functions
\[
\kappa^t_n(i)\, =\, n^{3-v}t^{2-e}\kappa^{t}_{H,n}(i)\, =\, \obinom{H}{\owedge}\frac{(1-t)^2}{(1-s)^2}
\]
and let ${\bf \rho}$ be
\[
\rho^t_n(i)\, =\, n^{3-v}t^{2-e} \rho^{t}_{H,n}(i)\, =\, t^{-1}\obinom{H}{\triangle} \frac{(1-t)^3}{(1-s)^{3}}\, .
\]
It is easily verified that the average $\frac{1}{\lfloor tN \rfloor} \sum_{i=1}^{\lfloor tN \rfloor} \kappa^{t}_{n}(i)$ is bounded away from $0$.  

Since $n^{v-3}t^{e-2} S^t_n=\Lambda^*_H(G_{n,t})$, we have
\[
\pr{\Lambda^{*}_H(G_{n,t})\, >\, \alpha_n n^{v-3/2}}\, =\, \pr{S_n^t\, >\, \alpha_n t^{2-e}n^{3/2}}\, .
\]
In order to apply Proposition~\ref{prop:maingen} we must verify that $t^{1/2} \ll \alpha_n t^{2-e} \ll t^4 n^{1/2}$.  This follows immediately from the condition that $t^{e-3/2}\ll \alpha_n\ll t^{e+2} n^{1/2}$.  And so, by an application of Proposition~\ref{prop:maingen}, we have
\[
\pr{\Lambda^{*}_H(G_{n,t})\, >\, \alpha_n n^{v-3/2}} \, =\, \exp\left(\frac{-\alpha_n^2 t^{4-2e} (1+o(1))}{2\tau_{{\bf \kappa},{\bf \rho}}}\right)\, ,
\]
where
\[
\tau_{{\bf \kappa},{\bf \rho}}\, :=\, n^{-2} \sum_{i=1}^{\lfloor tN\rfloor}\Big(8s(1-s)\kappa^t_n(i)^2\, +\, 36s^2(1-s)^2\rho_n^t(i)^2\Big) \, .
\]
All that remains is to prove that
\eq{rtbp}
\gamma_H(t)\, =\, \frac{1+o(1)}{2t^{2e-4}\tau_{{\bf \kappa},{\bf \rho}}}\, .
\eqe
Substituting the values of $\kappa^t_n(i)$ and $\rho^t_n(i)$ into the definition of $\tau_{{\bf \kappa},{\bf \rho}}$ we obtain
\begin{align*}
t^{2e-4}\tau_{{\bf \kappa},{\bf \rho}}\,  =\, & t^{2e-4} n^{-2} \sum_{i=1}^{\lfloor tN\rfloor}8s(1-s)\left(\obinom{H}{\owedge}\frac{(1-t)^2}{(1-s)^2}\right)^2\\ & \, +\, t^{2e-4} n^{-2}\sum_{i=1}^{\lfloor tN\rfloor} 36s^2(1-s)^2\left(t^{-1}\obinom{H}{\triangle} \frac{(1-t)^3}{(1-s)^{3}}\right)^2\, .
\end{align*}
The contribution of the first term is
\begin{align*}
& \frac{(4+o(1))}{N}t^{2e-4}(1-t)^4\obinom{H}{\owedge}^2\sum_{i=1}^{\lfloor tN\rfloor}\frac{s}{(1-s)^3} \phantom{\Bigg)}\\
 =\, & (4+o(1))t^{2e-4}(1-t)^4\obinom{H}{\owedge}^2 \int_{0}^{t}\frac{s}{(1-s)^3}\, ds \phantom{\Bigg)}\\
 =\, & (2+o(1))t^{2e-2}(1-t)^2\obinom{H}{\owedge}^2\, , \phantom{\Bigg)}
 \end{align*}
where we have used that the integral has value $t^2/2(1-t)^2$ (as may be seen using the substitution $u=1-s$, for example).
The contribution of the second term is
\begin{align*}
& \frac{(18+o(1))}{N}t^{2e-6}(1-t)^6 \obinom{H}{\triangle}^2 \sum_{i=1}^{\lfloor tN\rfloor} \frac{s^2}{(1-s)^4} \phantom{\Bigg)} \\
=\, & (18+o(1))t^{2e-6}(1-t)^6 \obinom{H}{\triangle}^2 \int_{0}^{t}\frac{s^2}{(1-s)^4} \, ds \phantom{\Bigg)}\\
=\, & (6+o(1))t^{2e-3}(1-t)^3\obinom{H}{\triangle}^2 \, , \phantom{\Bigg)}
\end{align*}
where we have used that the integral has value $t^3/3(1-t)^3$.
Summing these two contributions, we have
\[
t^{2e-4}\tau_{{\bf \kappa},{\bf \rho}}\, =\, (2+o(1))t^{2e-2}(1-t)^2\obinom{H}{\owedge}^2\, +\, (6+o(1))t^{2e-3}(1-t)^3\obinom{H}{\triangle}^2\, .
\]
By observation,~\eqr{rtbp} holds, and so the proof is complete.
\end{proof}

Our final task is to prove Proposition~\ref{prop:maingen}.  This proof will use the inequalities of Freedman stated in Section~\ref{sec:ineqs}.

\begin{proof}[Proof of Proposition~\ref{prop:maingen}]
Let $t\in (0,1)$ and the sequences ${\bf \kappa}=(\kappa^{t}_{n})_{n\ge 1}$ and ${\bf \rho}=(\rho^{t}_{n})_{n\ge 1}$ be fixed.  We may assume that $\kappa^{t}_{n}:\{1,\dots ,\lfloor tN\rfloor \}\to [-C,C]$ and $\rho^{t}_{n}:\{1,\dots ,\lfloor tN\rfloor \}\to [-Ct^{-1},Ct^{-1}]$ are such that~\eqr{average} holds.  Let us also fix $\eps>0$.  It will be useful at times to note that
\eq{range}
\Omega(t^2)\, \le\, \tau_{{\bf \kappa},{\bf \rho}}\, \le\, O(t)
\eqe
which follows easily from the definition of $\tau_{{\bf \kappa},{\bf \rho}}$ and the conditions on $\kappa^{t}_{n}$ and $\rho^{t}_{n}$. 

We must prove an upper bound and a lower bound on the probability of a deviation of the final value $S_n^t$ of the martingale
\[
S_n^t\,:=\, \sum_{i=1}^{\lfloor tN\rfloor }\left(\kappa_{n}^{t}(i) X_{\owedge}(G_i)\, +\, \rho^{t}_{n}(i)\, \big(X_{\triangle}(G_i)-3sX_{\owedge}(G_i)\big) \right)\, .
\]
Fix the sequence $(\alpha_n:n\ge 1)$ with $t^{1/2} \ll \alpha_n\ll t^4 n^{1/2}$.

We first prove the upper bound on the probability
\[
\pr{S_n^t\, >\, \alpha_n n^{3/2}}\, ,
\]
by an application of Freedman's inequality, Lemma~\ref{lem:F}.  We have that $S_n^t$ is the final value of a martingale
\[
S_n^t\,:=\, \sum_{i=1}^{\lfloor tN\rfloor }X(i)\, ,
\]
with increments
\[
X(i)\,:=\, \kappa_{n}^{t}(i) X_{\owedge}(G_i)\, +\, \rho^{t}_{n}(i)\, \big(X_{\triangle}(G_i)-3sX_{\owedge}(G_i)\big)\, .
\]
In order to apply Freedman's inequality we need to assess the quantity
\[
V(\lfloor tN\rfloor):=\sum_{i=1}^{\lfloor tN\rfloor}\,\, \Ex{\,  X(i)^2\, \big|\, G_{i-1}}\, .
\]
Let $E_n^t(\eps)$ be the event that
\[
\big|V(\lfloor tN\rfloor)\, -\, n^3\tau_{{\bf \kappa},{\bf \rho}}\big|\, \le\, \eps t^2n^3\, .
\]

We bound the probability of $E_n^t(\eps)$ using the following claim.  The quantities $V_{F,F'}(i,n)$ are as defined by~\eqr{Vdef}.

\noindent\textbf{Claim:} Let $\delta\, =\, \eps/15C^2$.  If
\begin{align*}
 \big|\Ex{X_{\owedge}(G_i)^2\,\big|\,G_{i-1}}\, -\, V_{\owedge,\owedge}(i,n)\big|\,  &\le\, \delta tn \phantom{\Big|} \\
 \big|\Ex{X_{\owedge}(G_i)X_{\triangle}(G_i)\,\big|\,G_{i-1}}\, -\, V_{\owedge,\triangle}(i,n)\big|\,  &\le\, \delta t^{2}n \qquad \text{and} \phantom{\Big|}\\
  \big|\Ex{X_{\triangle}(G_i)X_{\triangle}(G_i)\,\big|\,G_{i-1}}\, -\, V_{\triangle,\triangle}(i,n)\big|\,  &\le\, \delta t^{3}n \phantom{\Big|} ,
\end{align*}
for all $1\le i\le \lfloor tN\rfloor$, then $E_n^t(\eps)$ occurs.

\noindent\textbf{Proof of Claim:} We may express $\Ex{\,  X(i)^2\, |\, G_{i-1}}$ as
\begin{align*}
\kappa_{n}^{t}(i)^2\,\, \Ex{X_{\owedge}(G_i)^2\,\big|\,G_{i-1}}\, & +\, 2\kappa_{n}^{t}(i)\rho^{t}_{n}(i)\, \Ex{X_{\owedge}(G_i)\big(X_{\triangle}(G_i)-3sX_{\owedge}(G_i)\big)\,\big|\,G_{i-1}}\\ 
& +\, \rho^{t}_{n}(i)^2\,\, \Ex{\big(X_{\triangle}(G_i)-3sX_{\owedge}(G_i)\big)^2\,\big|\,G_{i-1}}\, .
\end{align*}
By the assumption of the claim, it follows that 
\begin{align*}
\Ex{\,  X(i)^2\, \big|\, G_{i-1}}\, =\, & \kappa_{n}^{t}(i)^2\, V_{\owedge,\owedge}(i,n)\, \\
& +\, 2\kappa_{n}^{t}(i)\rho^{t}_{n}(i)\, \big(V_{\owedge,\triangle}(i,n)\, -\, 3s V_{\owedge,\owedge}(i,n)\big)\\
& +\, \rho^{t}_{n}(i)^2\, \big(V_{\triangle,\triangle}(i,n)\, -\, 6sV_{\owedge,\triangle}(i,n)\, +\, 9s^2 V_{\owedge,\owedge}(i,n)\big)\\
& \pm \, 30C^2 \delta tn\, .
\end{align*}
Substituting in the values
\begin{align*}
V_{\owedge,\owedge}(i,n)\, & =\, 8ns(1-s)\, ,\phantom{\Big)} \\
V_{\owedge,\triangle}(i,n)\, & =\, 24ns^2(1-s) \qquad \text{and}\phantom{\Big)}\\
V_{\triangle,\triangle}(i,n)\, & =\, 36ns^2 (1-s^2)\, ,\phantom{\Big)}
\end{align*}
we obtain that
\[
\Ex{\,  X(i)^2\, \big|\, G_{i-1}}\, = \, 8ns(1-s)\kappa^t_n(i)^2\, +\, 36ns^2(1-s)^2\rho_n^t(i)^2\, \pm\, 30C^2 \delta tn\, .
\]
Now summing over $i=1,\dots,\lfloor tN\rfloor$ we obtain
\[
\big|V(\lfloor tN\rfloor)\, -\, n^3\tau_{{\bf \kappa},{\bf \rho}}\big|\, \le\, 30C^2 \delta tn \lfloor tN\rfloor\, \le \, \eps t^2 n^3\,  ,
\]
as required.  This completes the proof of the claim.

Let $C_1\ge \max\{15,C\}$ be at least the maximum constant of Proposition~\ref{prop:var} for cases with $F,F'\in \{\owedge,\triangle\}$ and with $t=\sup_{n} t(n)$.  By the claim, the event $E_n^t(\eps)^c$ may only occur if one of the events in the condition of the claim fails to occur.  Using Proposition~\ref{prop:var} to bound the probability of such events, we obtain that, for all sufficiently large $n$,
\begin{align}\label{eq:Entec}
\pr{E_n^t(\eps)^c}\, & \le \, 3tN\exp\left(\frac{-\eps^2 t^6 n}{15^2 C^4 C_1^2}\right) \nonumber \\
& \le\, \exp\left(\frac{-\eps^2 t^6 n}{C_1^8}\right)\, .
\end{align}
We are now ready to apply Lemma~\ref{lem:F} and obtain our upper bound.  For an upper bound on $\|X(i)\|_{\infty}$ we simply use that 
\[
|X(i)|\,\le\, 4C|X_{\owedge}(i)|\, +\, C t^{-1} |X_{\triangle}(i)|\, \le\, 5C t^{-1} n\, .
\]
We observe that
\[
\pr{S_n^t\, >\, \alpha_n n^{3/2}}\, \le \, \pr{ \{S_n^t\, >\, \alpha_n n^{3/2} \} \cap E_n^t(\eps)}\, +\, \pr{E_n^t(\eps)^c}\, .
\]
We bound the first probability by applying Freedman's inequality with $\alpha=\alpha_n n^{3/2}$, $\beta=n^3(\tau_{\kappa, \rho}+\eps t^2)$ and $R=5C t^{-1} n$, and the second by~\eqref{eq:Entec}.  We obtain
\[
\pr{S_n^t\, >\, \alpha_n n^{3/2}}\, \le\, \exp\left(\frac{-\alpha_n^2 n^3}{2n^3(\tau_{{\bf \kappa},{\bf \rho}}+\eps t^2)\, +\, 10Ct^{-1}\alpha_n n^{5/2}}\right)\, +\, \exp\left(\frac{-\eps^2 t^6 n}{C_1^8}\right)\, .
\]
for all sufficiently large $n$.

By the upper bound condition of $\alpha_n$ and~\eqr{range} we have that $\alpha_n\ll t^4 n^{1/2}=O(t^3\tau_{{\bf \kappa},{\bf \rho}}^{1/2} n^{1/2})$.  It follows that the second exponential above is $o(1)$ times the first exponential, and the second term of the denominator in the first exponential is $o(1)$ times the first term.
And so
\[
\pr{S_n^t\, >\, \alpha_n n^{3/2}}\, \le\, \exp\left(\frac{-\alpha_n^2 (1-O(\eps))}{2\tau_{{\bf \kappa},{\bf \rho}}}\right)
\]
for all sufficiently large $n$.  Since $\eps>0$ is arbitrary this completes the proof of the upper bound
\[
\pr{S_n^t\, >\, \alpha_n n^{3/2}}\, \le\, \exp\left(\frac{-\alpha_n^2 (1+o(1))}{2\tau_{{\bf \kappa},{\bf \rho}}}\right)\, .
\]

We now prove the lower bound.  In principle the proof of the lower bound should be straightforward, essentially equivalent to the proof of the upper bound, except with Lemma~\ref{lem:CF} being used instead of Lemma~\ref{lem:F}.  One subtlety is that such a direct application of Lemma~\ref{lem:CF} would give a lower bound on the probability of a deviation occurring \emph{before} a certain time, rather than \emph{at} time $\lfloor tN\rfloor$.  In particular, it will allow us to obtain a lower bound on the probability of the event $F_n^t(\eps)$ that
\[
\exists \ell \le \lfloor tN\rfloor \quad \text{such that}\quad \sum_{i=1}^{\ell}X(i)\, >\, (1+\eps)\alpha_n n^{3/2}\, .
\]
By an application of Freedman's inequality to the part of the martingale that occurs after first crossing $(1+\eps)\alpha_n n^{3/2}$, one easily verifies that there is at least probability $1/2$ that the martingale remains above $\alpha_n n^{3/2}$, for all sufficiently large $n$.  And so,
\[
\pr{S_n^t\, >\, \alpha_n n^{3/2}}\, \ge \, \frac{1}{2}\, \pr{F_n^t(\eps)}
\]
for all sufficiently large $n$.  Thus, to complete the proof we need only prove that
\[
\pr{F_n^t(\eps)}\,  \ge\, \exp\left(\frac{-\alpha_n^2 (1+O(\eps))}{2\tau_{{\bf \kappa},{\bf \rho}}}\right)\, 
\]
for all sufficiently large $n$.

We recall that the statement of Lemma~\ref{lem:CF} provides a lower bound on the probability
\[
\pr{T_{\alpha}\le \beta}
\]
where $T_{\alpha}$ is defined by
\[
T_{\alpha}\, =\, \sum_{i=1}^{m_{\alpha}}\,\, \Ex{\,  |X(i)|^2\, \big|\, \F_{i-1}}\, .
\]
where $m_{\alpha}$ is the least $m$ such that the martingale exceeds $\alpha$.  If we take 
\[
\alpha\, =\, (1+\eps)\alpha_n n^{3/2}
\]
and $\beta=n^3(\tau_{{\bf \kappa},{\bf \rho}}-\eps t^2)$, then it is easily observed that event $T_{\alpha}\le \beta$ is contained in $E_n^t(\eps)\cup F_n^t(\eps)$.  So we have that
\[
\pr{F_n^t(\eps)}\,  \ge\, \pr{T_{\alpha}\le \beta}\, -\, \pr{E_n^t(\eps)^c}\, .
\]
Applying Lemma~\ref{lem:CF}, we obtain
\[
\pr{F_n^t(\eps)}\,  \ge\,  \frac{1}{2} \exp\left(\frac{-\alpha^2(1+4\delta)}{2\beta}\right)\, -\, \exp\left(\frac{-\alpha_n^2}{\eps\tau_{{\bf \kappa},{\bf \rho}}}\right)
\]
where $\delta>0$ is minimal such that $\beta/\alpha \ge 9R\delta^{-2}$ and $\alpha^2/\beta \ge 16\delta^{-2}\log(64\delta^{-2})$.  Substituting the values of $\alpha$ and $\beta$ we obtain
\[
\pr{F_n^t(\eps)}\,  \ge\, \frac{1}{3}\exp\left(\frac{-\alpha_n^2(1+O(\eps)+O(\delta))}{2\tau_{{\bf \kappa},{\bf \rho}}}\right)\, .
\]
From the definition of $\alpha_n$ it is easily verified that $\delta=o(1)$, and so
\[
\pr{F_n^t(\eps)}\,  \ge\, \exp\left(\frac{-\alpha_n^2(1+O(\eps))}{2\tau_{{\bf \kappa},{\bf \rho}}}\right)\, 
\]
for all sufficiently large $n$, as required.  This completes the proof of the main statement of the proposition.

The furthermore part of the statement follows immediately by multiplying the functions by $-1$ and applying the main part.
\end{proof}

\section{Moderate deviations of subgraph counts in $G(n,p)$}\label{sec:pproofs}


As discussed in the sketch proof in the introduction, the proofs of both Theorem~\ref{thm:smalldelta} and Theorem~\ref{thm:largerdelta} are based around the identity
\eq{identity}
\pr{D_H(G_p)\, >\, \delta_n p^e(n)_v}\, =\, \sum_{m=0}^{N}b_N(m)\, \pr{N_H(G_m)\, >\, (1+\delta_n)p^e (n)_v}\, ,
\eqe
and in particular in finding which terms make the largest contribution to the sum.  With this in mind we define
\[
m_*\, :=\, pN(1+\delta_n)^{1/e}\, 
\]
and note that this is approximately (up to a small additive constant) the value of $m$ at which no deviation is required in $G_m$ in order that $N_H(G_m)>(1+\delta_n) p^e(n)_v$.  Indeed, since 
\[
\frac{(m-e)^e}{N^e}\, <\, \frac{(m)_e}{(N)_e}\, \le\, \frac{m^e}{N^e}
\]
we have $L_{H}(m_*)\le (1+\delta_n) p^e(n)_v$ and $L_{H}(m_*+e)> (1+\delta_n) p^e(n)_v$.

It will be useful in the proof of Theorem~\ref{thm:smalldelta} to also consider a value $m_-$ slightly less than $m_*$ and a value $m_+$ slightly larger.  We define
\[
m_-\, :=\, \fl{m_*\, -\, \delta_n^{-1/2}n^{1/4}}
\]
and
\[
m_+\, :=\, \fl{m_*\, +\, \delta_n^{-1/2}n^{1/4}}\, .
\]
The intuition behind the definitions of $m_-$ and $m_+$ is simply that their difference from $m_*$ is between order $n^{1/2}$ (the amount one must change $m$ to have a significant effect on $L_H(m)$) and order $\delta_n^{-1}$ (the amount one can change $m$ before it has a significant effect on the tail bound for binomial deviations).  In other words, the probability that $G_p$ has at least $m_+$ edges is asymptotically equivalent to the probability it has at least $m_-$ edges, and yet the event $N_H(G_m)>(1+\delta_n) p^e(n)_v$ changes from being very unlikely to very likely as $m$ grows from $m_-$ to $m_+$.

In addition we define
\[
x(m)\, :=\, \frac{m\, -\, pN}{\sqrt{Npq}}\, ,
\]
and we set $x_*:=x(m_*)$, $x_-:=x(m_-)$ and $x_+:=x(m_+)$.

We split the proof of Theorem~\ref{thm:smalldelta} into two parts (Section~\ref{sec:LBS} and Section~\ref{sec:UBS}), corresponding to the lower bound and upper bound.  Theorem~\ref{thm:largerdelta} is proved in Section~\ref{sec:L}.

\subsection{Lower bound of Theorem~\ref{thm:smalldelta}}\label{sec:LBS}

Let the sequence $n^{-1} \ll \delta_n \ll n^{-1/2}$ be given.

Since the second term, $\pr{N_H(G_m)\, >\, (1+\delta_n)p^e (n)_v}$, in the expression~\eqr{identity} is increasing in $m$ it follows that
\begin{align*}
\pr{D_H(G_p)\, >\, \delta_n p^e(n)_v}\, &\ge\, \sum_{m=m_+}^{N}b_N(m)\, \pr{N_H(G_{m_+})\, >\, (1+\delta_n)p^e (n)_v}\\
& =\, B_N(m_+)\, \pr{N_H(G_{m_+})\, >\, (1+\delta_n)p^e (n)_v} \phantom{\Big|}
\, .
\end{align*}

The proof of the lower bound therefore reduces to proving the following two lemmas.

\begin{lem} \label{lem:Bnm+}
\[
B_N(m_+)\, =\, \sqrt{\frac{e^2 q}{\pi p}} \exp\left(- \frac{\delta_n^2 pn^2}{4 e^2q}\, +\, \frac{\big((3e-2)-(3e-1)p\big) \delta_n^3 pn^2}{12e^3q^2}\, -\, \log (n\delta_n)\, +\, o(1)  \right)\, .
\]
\end{lem}


\begin{lem}\label{lem:NHm+}
\[
\pr{N_H(G_{m_+})\, >\, (1+\delta_n)p^e (n)_v}\, =\, 1+o(1)\, .
\]
\end{lem}

Let us first see that Lemma~\ref{lem:NHm+} follows easily from Theorem~\ref{thm:upto}.

\begin{proof} The event that
\eq{nodevisdev}
N_H(G_{m_+})\, \le\, (1+\delta_n)p^e (n)_v
\eqe
will correspond to a large negative deviation $D_H(G_{m_+})$.  Indeed, we observe that
\begin{align*}
L_H(m_+)\, &=\, \frac{(n)_v(m_+)_e}{(N)_e}\\
& =\, \frac{(n)_v}{(N)_e}\, (\fl{m_*\, +\, \delta_n^{-1/2}n^{1/4})})_e\\
& \ge\, \frac{(n)_v}{(N)_e}\, (m_*\, +\, \delta_n^{-1/2}n^{1/4})-e)^e\\
& \ge\, \frac{(n)_v}{(N)_e}\, (m_*^e\, +\, \delta_n^{-1/2}n^{1/4}m_*^{e-1})\\
& \ge \, (n)_v p^e(1+\delta_n)\, +\, \Omega(\delta_n^{-1/2} n^{v-7/4})\, .
\end{align*}
And so~\eqr{nodevisdev} is contained in the event
\[
D_H(G_{m_+})\, \le\, -\Omega(\delta_n^{-1/2} n^{v-7/4})\, .
\]
Since $\delta_n^{-1/2} n^{v-7/4}\, =\, \omega(n^{v-3/2})$, this event has probability $o(1)$ by Theorem~\ref{thm:upto}, as required.
\end{proof}

\begin{remark} In fact one does not need Theorem~\ref{thm:upto} to obtain Lemma~\ref{lem:NHm+}.  An alternative would be to use Chebyshev's inequality and the fact that the variance of $D_H(G_m)$ is $O(n^{2v-3})$.
\end{remark}

Before proving Lemma~\ref{lem:Bnm+} let us examine more closely the values of $x_*$ and $x_+$.  We recall that $m_*\, := \, pN(1+\delta_n)^{1/e}$, and so has expansion
\[
m_*\, =\, pN\left(1\, +\, \frac{\delta_n}{e}\, -\, \frac{\delta_n^2(e-1)}{2e^2}\, +\, O(\delta_n^3)\right)\, .
\]
It follows that $x_*$ may be expressed as
\begin{align}
x_*\, & =\,  \frac{\delta_np^{1/2}N^{1/2}}{eq^{1/2}}\, -\, \frac{\delta_n^2(e-1)p^{1/2}N^{1/2}}{2e^2q^{1/2}}\, +\, O(\delta_n^3 n) \nonumber \\
& =\,   \frac{\delta_np^{1/2}n}{e\sqrt{2}q^{1/2}}\, -\, \frac{\delta_n^2(e-1)p^{1/2}n}{2\sqrt{2}e^2q^{1/2}}\, +\, O(\delta_n^3 n)\, +\, O(n^{-1})\, .     \label{xstaras}
\end{align}
As $m_+-m_*$ is $\delta_n^{-1/2}n^{1/4} \pm 1$ it is clear that $x_+-x_*$ is $\delta_n^{-1/2}n^{1/4}/\sqrt{Npq} \, \pm n^{-1}$.  In fact we will only need that $x_+-x_*=O(\delta_n^{-1/2}n^{-3/4})$, so that
\eq{xplusas}
x_+\, =\,  \frac{\delta_np^{1/2}n}{e\sqrt{2}q^{1/2}}\, -\, \frac{\delta_n^2(e-1)p^{1/2}n}{2\sqrt{2}e^2q^{1/2}}\, +\, O(\delta_n^{-1/2}n^{-3/4})\, .
\eqe

\begin{proof}[Proof of Lemma~\ref{lem:Bnm+}]
By Theorem~\ref{thm:bah} we have
\eq{bgives}
B_N(m_+)\, =\, (1 + o(1)) \frac{1}{x_+ \sqrt{2\pi}} \exp \left(- \frac{x_+^2}{2}\,  -\, E(x_+,N,1) \right)\, ,
\eqe
where we have used that $x_+\, =\, \Theta(\delta_n n) \ll N^{1/4}$ to truncate the infinite sum $E(x_+,N)$ to $E(x_+,N,1)$.

From our expression~\eqr{xplusas} for $x_+$ we have that
\[
\frac{x_+^2}{2}\, =\, \frac{\delta_n^2 pn^2}{4 e^2q}\, -\, \frac{\delta_n^3(e-1)pn^2}{4 e^3q}\, +\, o(1)
\]
and
\[
x_+^3\, =\, \frac{\delta_n^3 p^{3/2} n^3}{2^{3/2}e^{3}q^{3/2}}\, +\, o(n)\,  .
\]

It is straightforward to calculate that
\[
\frac{x_+^2}{2}\,  +\, E(x_+,N,1)\, =\,  \frac{\delta_n^2 pn^2}{4 e^2q}\, +\, \frac{\big((3e-1)p-(3e-2)\big) \delta_n^3 pn^2}{12e^3q^2}\, .
\]
Substituting this into~\eqr{bgives} and using that $x_+\, =\, \delta_np^{1/2}n/e\sqrt{2}q^{1/2}\, +\, o(1)$, we obtain the desired result.
\end{proof}

\subsection{Upper bound of Theorem~\ref{thm:smalldelta}}\label{sec:UBS}

A key observation is that the expression for $B_N(m_+)$ given in Lemma~\ref{lem:Bnm+} is also an expression for $B_N(m_-)$ (as the difference between the two is contained in the $o(1)$ term).  This follows easily from the proof of Lemma~\ref{lem:Bnm+} and the observation that $x_-$ may also be expressed as
\[
x_-\, =\,  \frac{\delta_np^{1/2}n}{e\sqrt{2}q^{1/2}}\, -\, \frac{\delta_n^2(e-1)p^{1/2}n}{2\sqrt{2}e^2q^{1/2}}\, -\, O(\delta_n^{-1/2}n^{-3/4})\, .
\]

Applying the trivial upper bound $\pr{N_H(G_m)\, >\, (1+\delta_n)p^e (n)_v}\le 1$ for $m\ge m_-$ we obtain from the identity~\eqr{identity} that  
\begin{align*}
&\pr{D_H(G_p)\, >\, \delta_n p^e(n)_v } \\
& \qquad \le\, B_N(m_-) \, +\, \sum_{m=0}^{m_- -1 }b_N(m)\, \pr{N_H(G_{m})\, >\, (1+\delta_n)p^e (n)_v}\, \\
& \qquad=\, \sqrt{\frac{e^2q}{\pi p}} \exp\left(- \frac{\delta_n^2 pn^2}{4 e^2q}\, +\, \frac{\big((3e-2)-(3e-1)p\big) \delta_n^3 pn^2}{12e^3q^2}\, -\, \log (n\delta_n)\, +\, o(1)  \right)\\ 
& \qquad \qquad +\, \sum_{m=0}^{m_- -1 }b_N(m)\, \pr{N_H(G_{m})\, >\, (1+\delta_n)p^e (n)_v}\, .
\end{align*}
It therefore suffices to prove that
\eq{madeof}
\sum_{m=0}^{m_- -1 }b_N(m)\, \pr{N_H(G_{m})\, >\, (1+\delta_n)p^e (n)_v}\, =\,  o(1) \exp\left(- r_0(p,n,\delta_n) \right)\, 
\eqe
where we have set
\[
r_0(p,n,\delta_n)\, :=\, \frac{\delta_n^2 pn^2}{4 e^2q}\, +\, \frac{\big((3e-2)-(3e-1)p\big) \delta_n^3 pn^2}{12e^3q^2}\, -\, \log (n\delta_n)\, .
\]
We bound the sum by showing that the contribution of terms with $m\le m_- -n$ is small and by dividing the terms $m_- -n \le m< m_-$ into intervals.  If we write $m=m_* - f$, we may calculate that
\begin{align*}
L_H(m)\, &=\, \frac{(n)_v(m)_e}{(N)_e}\\
& \le \, \frac{(n)_v (m_* -f)^e}{N^e}\\
& \le \, \frac{(n)_v p^e(m_*^e\, -\, fm_*^{e-1})}{N^e}\\
& =\, (1+\delta_n)p^e (n)_v\, -\, \Omega(fn^{v-2})\, .
\end{align*}
It follows that the event $N_H(G_{m})\, >\, (1+\delta_n)p^e (n)_v$ corresponds to a deviation $D_H(G_m)\, >\, \Omega(fn^{v-2})$ which has probability at most
\eq{uptogives}
\exp(-\Omega(f^2/n))
\eqe
by Theorem~\ref{thm:upto}.

This gives us immediately that the contribution to~\eqr{madeof} of terms with $m\le m_- -n$ is at most $e^{-\Omega(n)}$ which is certainly $o(1)\exp(-r(p,n,\delta_n))$, as required.

We split the remaining values of $m$ into intervals $I_k:= \{m_{k+1},\dots ,m_k -1\}$ where $m_k:= \fl{m_*-k \delta_n^{-1/2}n^{1/4}}$. To complete the proof it clearly suffices to show that
\[
\sum_{k=1}^{\lceil \delta_n^{1/2}n^{3/4}\rceil} B_N(m_{k+1})\pr{N_H(G_{m_k})\, >\, (1+\delta_n)p^e (n)_v}\, =\,  o(1) \exp(- r(p,n,\delta_n))\, .
\]

To bound $B_N(m_{k+1})$ we calculate that $x_{k}:=x(m_k)$ satisfies
\[
x_{k}\, =\,  \frac{\delta_np^{1/2}n}{e\sqrt{2}q^{1/2}}\, -\, \frac{\delta_n^2(e-1)p^{1/2}n}{2\sqrt{2}e^2q^{1/2}}\, -\, O(k\delta_n^{-1/2}n^{-3/4})\, ,
\]
and so working as in the proof of Lemma~\ref{lem:Bnm+} we obtain
\[
B_N(m_k)\, \le \, O(1) \exp\left(- r_0(p,n,\delta_n)\, +\, O(k\delta_n^{1/2}n^{1/4})\right)\, .
\]
On the other hand, we have from~\eqr{uptogives} that
\[
\pr{N_H(G_{m_k})\, >\, (1+\delta_n)p^e (n)_v}\, =\, \exp\big(-\Omega(k^2\delta_n^{-1}n^{-1/2})\big)\, .
\]
It follows that
\begin{align*}
\sum_{k=1}^{\lceil \delta_n^{1/2}n^{3/4}\rceil} & B_N(m_{k+1})\, \pr{N_H(G_{m_k}) > (1+\delta_n)p^e (n)_v}\, \\
& =\,  O(1) \sum_{k=1}^{\lceil \delta_n^{1/2}n^{3/4}\rceil} \exp\left(- r_0(p,n,\delta_n)\, +\, O((k+1)\delta_n^{1/2}n^{1/4})\, -\, \Omega(k^2\delta_n^{-1}n^{-1/2})\right)  \\
& \le \, O(1) \sum_{k=1}^{\infty} \exp\left(- r_0(p,n,\delta_n)\, +\, k\delta_n^{-1}n^{-1/2}\big(O(\delta_n^{3/2}n^{3/4})\, -\, \Omega(k)\big)\right)\\ 
& \le \, O(1)\sum_{k=1}^{\infty} \exp\left(- r_0(p,n,\delta_n)\, -\, \Omega\big(k^2 \delta_n^{-1}n^{-1/2}\big)\right)\\ 
& =\, o(1)\exp(- r_0(p,n,\delta_n))\, .
\end{align*}

%

\subsection{Proof of Theorem~\ref{thm:largerdelta}}\label{sec:L}

Let the sequence $n^{-1} \ll \delta_n \ll 1$ be given.

We set
\[
r(p,n,\delta_n)\, := \, \frac{x_*^2}{2}\, +\, E(x_*,N)\, -\, \frac{\delta_n^2 n}{16\gamma_H(p) e^4p^{2e-2}q^2}\, .
\]
Our aim is to prove that
\[
\pr{D_H(G_{n,p}) > \delta_n p^e(n)_v}\, =\, \exp\big(-r(p,n,\delta_n)\, +\, o(\delta_n^2 n)\, +\, O(\log{n})\big)\, .
\]
Since we have included a $O(\log{n})$ error term in the exponent, which is equivalent to a multiplicative factor of $n^{O(1)}$, the sum given in~\eqr{identity} is equivalent to its largest term, and so it suffices to prove that
\[
\max_m b_N(m)\, \pr{N_H(G_{m})> (1+\delta_n)p^e (n)_v} \, =\, \exp\big(-r(p,n,\delta_n)\, +\, o(\delta_n^2 n)\, +\, O(\log{n})\big)\, .
\]

The maximum is achieved with $m$ slightly smaller than $m_*$.  We explore values of $m$ of the form $m_* -f$.

It will be useful to isolate a subset of the terms of $r(p,n,\delta_n)$, we set
\[
s(p,n,\delta_n)\, =\, \frac{x_*^2}{2}\, +\, E(x_*,N)\, ,
\]
and note that
\[
b_N(m_*)\, =\, \exp\big(-s(p,n,\delta_n)\, +\, O(\log{n})\big)\, .
\]

Let us calculate expressions for $b_N(m)$ and $\pr{N_H(G_{m})\, >\, (1+\delta_n)p^e (n)_v}$ for $m=m_*-f$.
We note that
\[
x(m)\, =\, x_*\, -\, f\sigma^{-1}
\]
where $\sigma=\sqrt{Npq}$, and so, by Theorem~\ref{thm:bah}, we have
\begin{align}
b_N(m)\, &=\, \exp\left(-\frac{(x_*-f\sigma^{-1})^2}{2}\, -\, E(x_*-f\sigma^{-1},N)\, +\, O(\log{n})\right)\nonumber \\
& =\, \exp\left( -s(p,n,\delta_n)\, +\, (1+o(1))f x_*\sigma^{-1}\, +\, O(\log{n})\right)\, .\label{eq:bNmis}
\end{align}

In order to get an expression for $\pr{N_H(G_{m})\, >\, (1+\delta_n)p^e (n)_v}$ we first calculate
\begin{align*}
L_H(m) \, &=\, \frac{(n)_v(m)_e}{(N)_e} \\
& =\, \frac{(n)_v (m_*-f)^e}{N^e} \, +\, O(n^{v-2})\\
& =\, (1+\delta_n)p^e (n)_v\, -\, ef\, (n)_v \frac{m_*^{e-1}}{N^e}\, +\, O(n^{v-2})\\
& =\, (1+\delta_n)p^e (n)_v\, -\, (1+o(1))\frac{ef\, (n)_v p^{e-1}}{N}\, +\, O(n^{v-2})\\
& =\, (1+\delta_n)p^e (n)_v\, -\, (1+o(1))\, 2efp^{e-1}n^{v-2} \, +\, O(n^{v-2})\, .
\end{align*}
Therefore the event $N_H(G_{m})\, >\, (1+\delta_n)p^e (n)_v$ corresponds to a deviation 
\[
D_H(G_m)\, >\, (1+o(1))\, 2efp^{e-1}n^{v-2}\, .
\]
If $f=\Omega(n)$ then 
\[
\pr{N_H(G_{m})\, >\, (1+\delta_n)p^e (n)_v}\, =\, \exp(-\Omega(f))\, 
\]
by Theorem~\ref{thm:upto}.  For $f\ll n$ we may apply Theorem~\ref{thm:main} to obtain
\eq{NHmis}
\pr{N_H(G_{m})\, >\, (1+\delta_n)p^e (n)_v}\, =\, \exp\left(\frac{-(4+o(1))\gamma_H(p)e^2f^2p^{2e-2}}{n}\right)\, .
\eqe

We may immediately observe that the maximum will not occur with $f=\Omega(n)$.  Indeed, for such $f$, we get
\[
\exp \left(-s(p,n,\delta_n) + (1+o(1))fx_*\sigma^{-1} - \Omega(f)\right)
\]
and $fx_*/\sigma-\Omega(f)=o(f)-\Omega(f)<0$ for all sufficiently large $n$.  We may therefore assume $f\ll n$.

In this range ($f\ll n$), we may combine~\eqr{bNmis} and~\eqr{NHmis} to obtain
\begin{align*}
& b_N(m)\, \pr{N_H(G_{m})\,  >\, (1+\delta_n)p^e (n)_v}\phantom{\Big|} \\
& = \, \exp\left(-s(p,n,\delta_n)\, +\, (1+o(1))f x_*\sigma^{-1}\, - \frac{(4+o(1))\gamma_H(p)e^2f^2p^{2e-2}}{n} +\, O(\log{n})\right)\, .
\end{align*}
If $\delta_n\ll n^{-1/2}\sqrt{\log{n}}$ then $x_*\ll \sqrt{n\log{n}}$ and the maximum of the terms involving $f$ is $o(\log{n})$, and can be absorbed into the error term.  If $\delta_n=\Omega(n^{-1/2}\sqrt{\log{n}})$ then the maximum is obtained with
\[
f_*\, =\, (1+o(1)) \frac{x_* n}{8\gamma_H(p) \sigma e^2 p^{2e-2}}\, .
\]
Setting $m=m_*-f_*$, this maximum is given by
\begin{align*}
b_N(m)\, &\pr{N_H(G_{m})> (1+\delta_n)p^e (n)_v}\phantom{\Big|}\\ 
& =\, \exp\left(-s(p,n,\delta_n)\, +\, (1+o(1))\frac{x_*^2 n}{16\gamma_H(p)\sigma^2e^2 p^{2e-2}}\,  +\, O(\log{n})\right) \, .
\end{align*}
Since $x_*\sim \delta_n p^{1/2}n/\sqrt{2q}e$ and $\sigma^2\sim pqn^2/2$ we have shown that at the maximum we have
\begin{align*}
b_N(m)\, &\pr{N_H(G_{m}) > (1+\delta_n)p^e (n)_v}\phantom{\Big|}\\
& =\, \exp\left(-s(p,n,\delta_n)\, +\, (1+o(1))\frac{\delta_n^2 n}{16\gamma_H(p)e^4 p^{2e-2}q^2}\,  +\, O(\log{n})\right) \, ,
\end{align*}
as required.  This completes the proof.

\subsection*{Acknowledgements}  C.G.\ and S.G.\ were supported by EPSRC grant EP/J019496/1.  C.G.\ was also supported by EPSRC Fellowship EP/N004833/1.  S.G. was also supported by research support from PUC-Rio, CNPq bolsa de produtividade em pesquisa (Proc. 310656/2016-8) and FAPERJ Jovem cientista do nosso estado (Proc. 202.713/2018).  A.S.\ was supported by a Leverhulme Trust Research Fellowship.  The authors would like to thank Jos\'e D.\ Alvarado for giving some helpful comments on a draft of the paper.

\section{Appendix}

We prove Theorem~\ref{thm:bah}.

In the context in which it was presented and applied it was more natural to state the theorem for $\mathrm{Bin}(N,p)$ where $N$ denotes $\binom{n}{2}$.  Let us revert to lower case $n$ for the proof, so that
\[
b_n(k)\, :=\, \pr{\mathrm{Bin}(n,p) = k}
\]
and
\[
B_n(k)\,  :=\, \pr{\mathrm{Bin}(n,p) \ge k}\, .
\]
As we stated before the statement of Theorem~\ref{thm:bah}, $p$ may may either be a constant $p\in (0,1)$ or a function $p=p_n$.

Set $\sigma_n=\sqrt{npq}$.  We must prove that
\eq{bnis}
b_n(\fl{pn+x_n\sigma_n})\,  =\, (1 + o(1))\frac{1}{\sqrt{2 \pi \sigma_n^2}} \exp \left(-\frac{x_n^2}{2} - E(x_n,n) \right)
\eqe
and
\eq{BNis}
B_n(pn + x_n \sigma_n)\, =\, (1 + o(1)) \frac{1}{x_n \sqrt{2\pi}} \exp \left(- \frac{x_n^2}{2} - E(x_n,n) \right)\, ,
\eqe
for all $1\ll x_n\ll \sigma_n$, where
\[
E(x,n)\, =\, \sum_{i=1}^{\infty} \frac{(p^{i+1} + (-1)^i q^{i+1})x^{i+2}}{(i+1)(i+2) p^{i/2} q^{i/2} n^{i/2}}\, .
\]

Let us also remark that if we keep track of the error terms in the proof then we obtain
\[
B_n(pn + x_n \sigma_n) = \left(1 + O\left(\frac{x_n}{\sqrt{np}} + \frac{1}{x_n^2}\right) \right)  \frac{1}{x_n \sqrt{2\pi}} \exp \left(- \frac{x_n^2}{2} - E(x_n,n) \right).
\]

\begin{proof}[Proof of Theorem~\ref{thm:bah}]  Let us note immediately that 
\[
E(x_n,n,J)\, =\, E(x_n,n)\, +\, o(1)
\]
in the case that $x_n\ll (pqn)^{1/2} (pqn)^{-1/(J+3)}$, and so the ``Furthermore'' statement follows immediately from the main statements.

Both of the main asymptotic identities will follow from the fact that
\[
A_n(np + x_n\sigma_n) C_n(x_n)\,  \to\, 1
\]
where
\begin{align*}
A_n(k) & :=  \frac{(k+1)q}{k+1 - (n+1)p}\binom{n}{k} p^k q^{n-k}, \\
\intertext{and}
C_n(x) & := x \sqrt{2\pi} \exp \left(\frac{x^2}{2} + E(x,n) \right)
\end{align*}
In fact we prove that 
\[
A_n(np + x_n\sigma_n) C^{+}_n(x_n)\,  \to\, 1
\]
where
\[
C^{+}_n(x) := \left(x + \sqrt{\frac{q}{np}} \right) \sqrt{2\pi} \exp \left(\frac{x^2}{2} + E(x,n) \right)
\]
which is clearly equivalent as $1\le C^{+}_n(x)/C_n(x) \le 1\, +\, O(\sqrt{q/pnx})\, =\, 1+o(1)$. 

Setting $k_n\, =\, np\, +\, x_n\sigma_n$ we observe that
\[
\frac{k_n +1}{k_n+1 - (n+1)p} \, =\, \frac{np + x_n \sigma_n + 1}{x_n\sigma_n + q} \, =\, \frac{\sqrt{\frac{np}{q}} + x_n + \frac{1}{\sigma_n}}{x_n + \sqrt{\frac{q}{np}}}\, , 
\]
and so
\[
\left(x_n + \sqrt{\frac{q}{np}} \right) \frac{k_n +1}{k_n+1 - (n+1)p} = \sqrt{\frac{np}{q}}\left(1 + x_n\sqrt{\frac{q}{np}} + \frac{1}{np} \right)\, .
\]
It follows that
\begin{align*}
& A_n(k_n)\, C^{+}_n(x_n) \\
& \qquad =\, \binom{n}{k_n} p^{k_n} q^{n-k_n} \left(1 + x_n \sqrt{\frac{q}{np}} + \frac{1}{np}\right) \sqrt{2\pi \sigma_n^2} \exp\left(x^2_n/2 + E(x_n,n)\right)\, .
\end{align*}
We have $n, k_n, n-k_n \to \infty$ and so we may apply Stirling's approximation to the three factorials involved in the binomial coefficient to obtain that the right-hand side is asymptotically equivalent to
\[
\frac{ \left(1 + x_n \sqrt{\frac{q}{np}} + \frac{1}{np} \right) } { \left( 1 + \frac{(q-p)}{\sqrt{pq}} \frac{x_n}{\sqrt{n}} - \frac{x_n^2}{n} \right)^{1/2} } \cdot \frac{ \exp \left(x_n^2/2 + E(x_n,n) \right) }
{\left( 1 + x_n \sqrt{\frac{q}{np}} \right)^{np + x_n \sigma_n} \left( 1 - x_n \sqrt{\frac{p}{nq}} \right)^{nq -x_n\sigma_n} }.
\]
The first term in this product is equal to $1 + O(x_n/\sqrt{np})$.  On the other hand the logarithm of the denominator of the second term in the product is
\begin{align*}
& (np + x_n \sigma_n ) \log \left(1 + x_n \sqrt{\frac{q}{np}} \right) + (nq - x_n\sigma_n ) \log \left(1 - x_n \sqrt{\frac{p}{nq}} \right) \\
& = (np + x_n \sigma_n) \sum_{i=1}^{\infty} (-1)^{i+1} \frac{1}{i} \left(\frac{q}{p}\right)^{i/2} \left(\frac{x_n}{\sqrt{n}}\right)^i  - (nq - x_n \sigma_n) \sum_{i=1}^{\infty} \frac{1}{i} \left(\frac{p}{q}\right)^{i/2} \left(\frac{x_n}{\sqrt{n}} \right)^i  \\
& = \sum_{j=0}^{\infty} \frac{(p^{j+1} + (-1)^j q^{j+1})x_n^{j+2}}{(j+1)(j+2) p^{j/2} q^{j/2} n^{j/2}}\\
& = \frac{x_n^2}{2}\, +\, E(x_n,n)\,
\end{align*}
provided $n$ is such that $x_n/\sqrt{n} < \min\{\sqrt{q/p},\sqrt{p/q}\}$ (which is certainly the case for all $n$ sufficiently large). Hence,
\[
A_n(k_n)\, C^{+}_n(x_n)  \,   = \, (1+ O(x_n/\sqrt{np})) \,  \to \, 1\, .
\]
Now observe that 
\[
b_n(np + x_n\sigma_n) \, =\, \frac{k_n+1 - (n+1)p}{(k_n+1)q} A_n(k_n)
\]
and that
\[
\frac{k_n+1 - (n+1)p}{(k_n+1)q} \frac{\sigma_n}{x_n} \to 1
\]
as $n \to \infty$, from which it follows that
\[
b_n(np + x_n \sigma_n) = (1+ o(1)) \frac{1}{\sqrt{2 \pi \sigma_n^2}} \exp \left(-\frac{x_n^2}{2} - E(x_n,n) \right).
\]
Finally, by Theorem 1 of Bahadur~\cite{Bah}, we have
\[
1 \,\le\, \frac{A_n(k_n)}{B_n(k_n)}\, < \, 1 + x_n^{-2}\, .
\]
It follows that
\[
B_n(np + x_n \sigma_n)\, = \, (1+o(1))  \frac{1}{x_n \sqrt{2\pi}} \exp \left(- \frac{x_n^2}{2} - E(x_n,n) \right). \qedhere
\]
\end{proof}

\end{document}